\documentclass[11pt]{amsart}
\usepackage{a4wide}

\usepackage{times}
\usepackage{amssymb}
\usepackage{graphicx,xspace}
\usepackage{epsfig}
\usepackage[usenames,dvipsnames]{xcolor}
\usepackage{tikz}
\usepackage[T1]{fontenc}
\usepackage[utf8]{inputenc}
\usepackage{bm}
\usepackage[boxsize=1em]{ytableau}

\usepackage{hyperref}
\usepackage{todonotes}

\usepackage[capitalize]{cleveref}

\makeatletter
\def\@tocline#1#2#3#4#5#6#7{\relax
  \ifnum #1>\c@tocdepth % then omit
  \else
    \par \addpenalty\@secpenalty\addvspace{#2}%
    \begingroup \hyphenpenalty\@M
    \@ifempty{#4}{%
      \@tempdima\csname r@tocindent\number#1\endcsname\relax
    }{%
      \@tempdima#4\relax
    }%
    \parindent\z@ \leftskip#3\relax \advance\leftskip\@tempdima\relax
    \rightskip\@pnumwidth plus4em \parfillskip-\@pnumwidth
    #5\leavevmode\hskip-\@tempdima #6\nobreak\relax
    \dotfill\hbox to\@pnumwidth{\@tocpagenum{#7}}\par %original has \dotfill instead of \hfill
    \nobreak
    \endgroup
  \fi}
\makeatother

\theoremstyle{plain}

\newtheorem{lemma}{Lemma}[section]
\newtheorem{theorem}[lemma]{Theorem}
\newtheorem{corollary}[lemma]{Corollary}
\newtheorem{proposition}[lemma]{Proposition}
\newtheorem{definition}[lemma]{Definition}

\theoremstyle{remark}
\newtheorem{remark}[lemma]{Remark}
\newtheorem{example}[lemma]{Example}
\newtheorem{observation}[lemma]{Observation}

\def\Z{\mathbb{Z}}
\def\Q{\mathbb{Q}}
\def\eps{\varepsilon}
\def\Eps{\mbox{\LARGE $\varepsilon$}}
\def\Si{\bm{\sigma}}
\def\Mu{\bm{\mu}}
\def\si{\sigma}

\def\Sn{\mathfrak{S}}
\def\vv{\mathbf{v}}
\def\xx{\mathbf{x}}
\def\yy{\mathbf{y}}
\def\XX{\mathbf{X}}
\def\UU{\mathbf{U}}
\def\ll{\bm{\ell}}
\def\R{\mathbb{R}}
\def\Q{\mathbb{Q}}
\def\CCC{\mathcal{C}}
\def\proba{\mathbb{P}}
\def\esper{\mathbb{E}}
\DeclareMathOperator{\Var}{Var}
\def\Exc{e}

\def\patterntree{t_0}
\def\One{\bm{1}}

\DeclareMathOperator{\id}{id}
\DeclareMathOperator{\occ}{\widetilde{occ}}
\DeclareMathOperator{\perm}{perm}
\DeclareMathOperator{\Perm}{Perm}
\DeclareMathOperator{\Tree}{Tree}
\DeclareMathOperator{\pat}{pat}
\DeclareMathOperator{\Part}{Part}
\DeclareMathOperator{\dfs}{dfs}
\DeclareMathOperator{\ProbTree}{Pr^{Tree}}
\DeclareMathOperator{\ProbPerm}{Pr^{Perm}}
\DeclareMathOperator{\Cat}{Cat}
\DeclareMathOperator{\Omin}{Omin}
\DeclareMathOperator{\Pos}{Pos}
\DeclareMathOperator{\PPos}{ {\bf Pos}}
\DeclareMathOperator{\Val}{Val}
\DeclareMathOperator{\VVal}{ {\bf Val}}
\DeclareMathOperator{\gap}{gap}
\DeclareMathOperator{\argmin}{argmin}
\newcommand{\Blue}[1]{\textcolor{blue}{#1}}
\newcommand{\Green}[1]{\textcolor{green!60!black}{#1}}
\newcommand{\Red}[1]{\textcolor{red}{#1}}
\def\N{N}

\title[The Brownian limit of separable permutations]{The Brownian limit of separable permutations}

\author{Frédérique Bassino}
       \address{Université Paris 13, Sorbonne Paris Cité, LIPN, CNRS UMR 7030, F-93430 Villetaneuse, France}
       \email{bassino@lipn.univ-paris13.fr}

 \author{Mathilde Bouvel}
       \address{Institut für Mathematik, Universität Zürich, Winterthurerstr. 190, CH-8057 Zürich, Switzerland}
       \email{mathilde.bouvel@math.uzh.ch}

 \author{Valentin Féray}
       \address{Institut für Mathematik, Universität Zürich, Winterthurerstr. 190, CH-8057 Zürich, Switzerland}
       \email{valentin.feray@math.uzh.ch}

 \author{Lucas Gerin}
       \address{CMAP, \'Ecole Polytechnique, CNRS, Route de Saclay, 91128 Palaiseau Cedex, France}
       \email{gerin@cmap.polytechnique.fr}

 \author{Adeline Pierrot}
 \address{LRI, Université Paris-Sud, Bat. 650 Ada Lovelace, 91405 Orsay Cedex, France}
       \email{adeline.pierrot@lri.fr}



\begin{document}

\begin{abstract}
We study uniform random permutations in an important class of pattern-avoiding permutations:
the separable permutations. 
We describe
the asymptotics of the number of occurrences of any fixed given pattern
in such a random permutation in terms of the Brownian excursion.
In the recent terminology of permutons, our work can be interpreted as the convergence
of  uniform random separable permutations towards a "Brownian separable permuton". 
\vspace{-8mm}
\end{abstract}

\maketitle

\begin{figure}[ht]
\begin{center}
\includegraphics[scale=0.3]{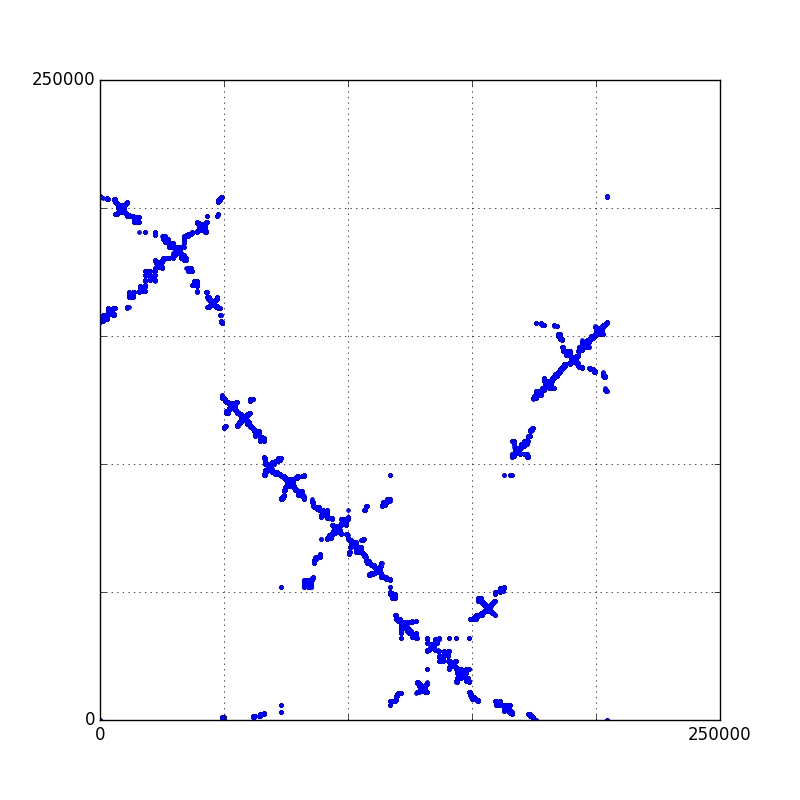} \qquad \includegraphics[scale=0.3]{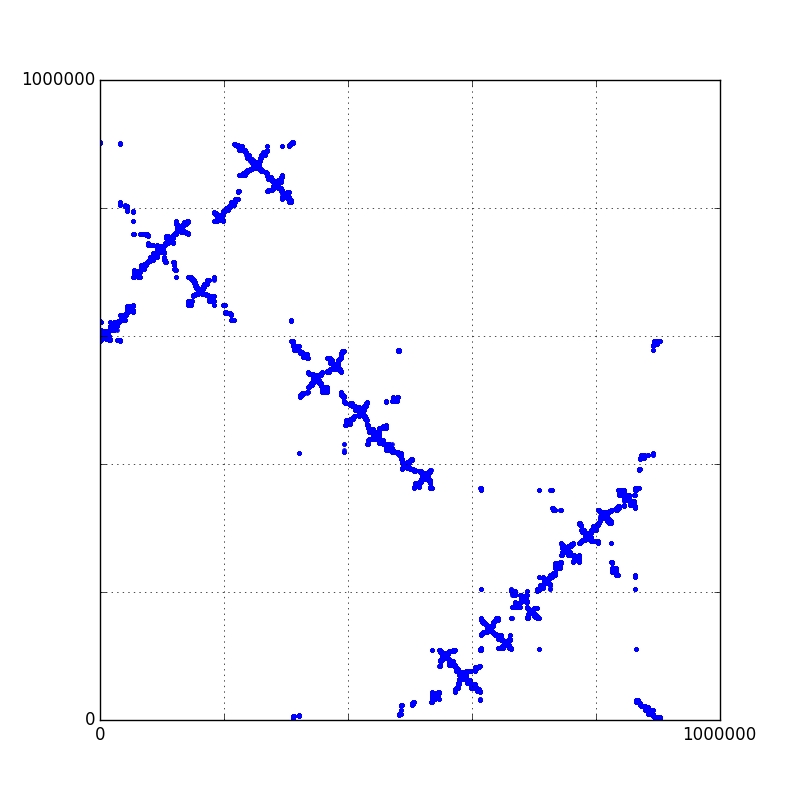}
\end{center}
\caption{Two uniform random separable permutations of sizes respectively $n=204\, 523$ and $n=903\, 073$ (a permutation $\sigma$ is represented here with its \emph{diagram}: for every $i \leq n$, there is a dot at coordinates $(i,\sigma_i)$).}

\label{fig:simu_LargeSeparable}
\end{figure}

\tableofcontents

\section{Introduction}

The aim of this article is to study the asymptotic properties of an important class of pattern-avoiding permutations: the separable permutations. 
Our main result is the description of the asymptotics in $n$
of the number of occurrences of any fixed given pattern in a uniform separable permutation of $n$ elements.

\subsection{Pattern-avoiding permutations}
We first give some definitions. 
For any $n$, the set of permutations of $[n]:= \{1,2,\ldots, n\}$ is denoted by $\Sn_n$. 
We write permutations of $\Sn_n$ in one-line notation as $\sigma = \sigma_1 \sigma_2 \dots \sigma_n$. 
For a permutation $\sigma$ in $\Sn_n$, the \emph{size} $n$ of $\sigma$ is denoted by $|\sigma|$. 
For $\sigma\in \Sn_n$, and $I\subset [n]$ of cardinality $k$, let $\pat_I(\sigma)$ be the permutation of $\Sn_k$ induced by $\{\sigma_i,\ i\in I\}$. 
For example for $\sigma=65831247$ and $I=\{2,5,7\}$ we have
$$
\pat_{\{2,5,7\}}\left(6\overline{5}83\overline{1}2\overline{4}7\right)=312
$$
since the values in the subsequence $\sigma_2 \sigma_5 \sigma_7=514$ are in the same relative order as in the permutation $312$. 
A permutation $\pi = \pat_I(\sigma)$ is a \emph{pattern} involved in $\sigma$, 
and the subsequence $(\sigma_i)_{i \in I}$ is an \emph{occurrence} of $\pi$ in $\sigma$. 

All along this paper, we use letter $\sigma$ for a (large) permutation of size $n$, and letter $\pi$ for a pattern of size $k\leq n$.
We denote by $\occ(\pi,\sigma)$ the proportion of occurrences of a pattern $\pi$ in $\sigma$.
More formally
\[\occ(\pi,\sigma) = \frac{1}{\binom{n}{k}} \, \mathrm{card}\{I \subset [n] \text{ of cardinality }k \text{ such that } \pat_I(\sigma)=\pi\}.\]
Equivalently $\occ(\pi,\sigma)$ is the probability to have $\pat_I(\sigma)=\pi$ 
if $I$ is randomly and uniformly chosen among the $\binom{n}{k}$ subsets of $[n]$ with $k$ elements.
If $|\pi|>|\sigma|$, we set conveniently $\occ(\pi,\sigma)=0$.

We say that $\sigma$ \emph{avoids} $\tau$ if there is no occurrence of $\tau$ in $\sigma$, \emph{i.e.}, $\occ(\tau,\sigma)=0$. 
For any (finite or infinite) set of patterns $\tau_1,\tau_2, \dots$, we denote by $\mathrm{Av}_n(\tau_1,\tau_2, \dots)$ the set of permutations of size $n$ that avoid all the $\tau_i$'s. 
Then, $\mathrm{Av}(\tau_1,\tau_2, \dots) = \cup_{n} \mathrm{Av}_n(\tau_1,\tau_2, \dots)$ is called a \emph{class} of (pattern-avoiding) permutations.
Equivalently\footnote{
This statement is folklore in the literature on permutation patterns. 
A proof can be found in \cite[Paragraph 5.1.2]{BonaBook} for instance. 
See also \cite[Paragraph 7.2.3]{BonaBook} for a proof that an infinite set of excluded patterns is sometimes necessary to describe a class.}, a class of permutations is a set $\mathcal{C}$ of permutations such that, 
for any $\sigma \in \mathcal{C}$ and any pattern $\pi$ of $\sigma$, it holds that $\pi \in \mathcal{C}$. 

Classes of permutations have been intensively studied for their combinatorial and algorithmic properties over the last 50 years. 
An account of the past and current research on these classes can be found in~\cite{BonaBook,Kitaev,Vatter}. 
Finding the enumeration of specific classes, defined by the avoidance of a small number of small patterns, 
has been one of the first problems studied in this field. 
It started with the proof that $\mathrm{Av}(\tau)$ is counted by the Catalan numbers, for any $\tau \in \Sn_3$, 
and the research on this topic still continues, as witnessed by the summary~\cite{Wikipedia}. 
The combinatorics of classes of permutations has however expanded in several other directions, 
including a general approach to the study of classes of permutations based on various notions of structure, 
like the \emph{substitution decomposition} (\cite[Proposition 2]{AA05} or \cite[Section 3.2]{Vatter}) to which we shall return later in this introduction. 

The probabilistic study of classes of permutations is much more recent 
and, just like their combinatorial study at its beginning, it focuses on the study of specific classes with small excluded patterns. 
More precisely, the probabilistic counterpart of the study of specific classes
is centered on the following interesting question: 
Given a fixed pattern $\tau$, what can we say about a typical $\sigma$ in $\mathrm{Av}_n(\tau)$ (for large $n$)? 
Recently, many authors have considered this problem
for different choices of small patterns $\tau$. We mention a few of them. 
\begin{itemize}
  \item The question was initiated in a paper of Madras and Liu
    \cite{MadrasLiu} in relation with a Monte-Carlo algorithm
    to approximate growth rates 
     of permutation classes.
    In subsequent papers, Atapour and Madras \cite{AtapourMadras}
    and Madras and Pehlivan \cite{MadrasPehlivan} started
    the study of uniform permutations in $\mathrm{Av}_n(\tau)$
    for small patterns $\tau$.
\item In parallel, Miner and Pak \cite{MinerPak}
  described very precisely the asymptotic shape of a uniform element in $\mathrm{Av}_n(\tau)$ for the $6$ patterns $\tau$ in $\Sn_3$. 
  Dokos and Pak \cite{DokosPak} have then obtained similar results for random doubly alternating Baxter permutations.
Note also that Miner and Pak discuss at the end of their paper a possible connection with the Brownian excursion.
\item Such a connection between $\mathrm{Av}_n(\tau)$ for $\tau\in \Sn_3$ and the Brownian excursion is explained by Hoffman, Rizzolo and Slivken \cite{HoffmanBrownian1}. 
Many combinatorial consequences are given, in particular a precise description
of fixed points of such permutations \cite{HoffmanBrownian2}.
\item In a parallel line of research,
  B\'ona \cite{Bona1,Bona2} investigates the behaviour of $\mathbb{E}[\occ(\pi,\sigma)]$ for $\sigma$ uniform in $\mathrm{Av}_n(132)$, and several fixed $\pi$'s. Similar results for other permutation classes and 
  various patterns $\pi$ have then been obtained by Homberger \cite{Homberger},
  Chang, Eu and Fu \cite{ChenEuFu} and Rudolf \cite{Rudolph}.
\item The question of finding limiting distributions for $\occ(\pi,\sigma)$ for $\sigma$ uniform in $\mathrm{Av}_n(\tau)$,
  rather than studying only its expectation, was raised
  by Janson, Nakamura and Zeilberger in \cite{JansonNakamuraZeilberger}: the authors gave some algorithms to find limits of moments 
  for small $\pi$ and $\tau$.
A bit later,
  Janson \cite{JansonPermutations} has given for every $\pi$ the asymptotic behaviour of the random variable $\occ(\pi,\sigma)$ for $\sigma$ uniform in $\mathrm{Av}_n(132)$. For instance, he expresses
in terms of the Brownian excursion area the asymptotic behaviour of $\occ(12,\sigma)$.
\item In his recent thesis, Bevan describes the limit shape
  of permutations in so-called {\em connected monotone grid classes}
  \cite[Chapter 6]{BevanPhD}.
  This result is the first that deals with an infinite
  family of permutation classes.
\item Even if it does not involve strictly speaking
  {\em pattern-avoiding permutations}, we mention the recent work
  of Kenyon, Kral', Radin and Winkler \cite{EntropyPermutation}.
  They prove a large deviation theorem for permutations
  seen as probability measures on the square.
  Their result yields limit shapes of random permutations
  with fixed densities of some fixed pattern $\pi_1,\cdots,\pi_r$.
  This parallels similar results on graphons, which are well-studied objects
  in random graph theory.
\end{itemize}

In the current paper,
we are specifically interested in the class of \emph{separable} permutations. 

\begin{definition}
A permutation $\sigma$ is \emph{separable} if $\sigma$ avoids both $2413$ and $3142$.
\end{definition}

We obtain results similar to those of Janson \cite{JansonPermutations} for $\mathrm{Av}(132)$,
namely we study occurrences of any pattern $\pi$
in uniform random separable permutations. 
There is however an important difference between our work and all classes studied so far: 
random permutations in any of these previously studied classes have a limit which is \emph{deterministic} at first order, 
whereas random separable permutations have a limit which is \emph{non-deterministic} at first order. 
The limit of random separable permutations will be discussed in \cref{subsec:Intro_Permutons}, 
and the proof that it is not deterministic will be given in \cref{ssec:variance}. 
This is also visible on \cref{fig:simu_LargeSeparable}, which shows two large typical separable permutations
obtained using a Boltzmann random sampler. 

There are several reasons that motivate our choice of studying the class of separable permutations, in addition to it being one of the most studied classes after $\mathrm{Av}(\tau)$ for $\tau$ of size $3$. 
Separable permutations have a very nice and robust combinatorial structure: they can be completely decomposed using direct sums and skew sums, 
and therefore can be represented as signed Schr\"oder trees (see \cref{sec:PermutationsAndSchroderTrees}). 
This encoding with trees is essential in proving a variety of results about separable permutations in different fields, for instance: 
\begin{itemize}
 \item the algorithmic problem of \textsc{Permutation Pattern Matching} is NP-hard in general, but polynomial on separable permutations~\cite{BBL98}; 
 \item from an enumerative combinatorics point of view, in addition to being simple to count, separable permutations display remarkable {\em equipopularity} properties, see~\cite{equipop}. 
\end{itemize}
Besides, separable permutations appear naturally in several problems,
at first sight independent from permutation pattern theory:
\begin{itemize}
  \item as the permutations sortable by certain sorting devices (pop-stacks in series) \cite{Avis}; 
  \item as space-filling permutation matrices in bootstrap percolation \cite{Bootstrap};
  \item as possible polynomial interchanges (i.e. studying in which possible ways
     the relative order of the values of a family of polynomials can be modified
     when crossing a common zero) \cite{Ghys}.
\end{itemize}
Finally, the class of separable permutations is the simplest case of a non-trivial \emph{substitution-closed} (also called wreath-closed~\cite{AA05}) class, 
and we believe that the results obtained here
might be extended to any substitution-closed class;
see the discussion on universality in \cref{subsec:perspectives}.

\subsection{Overview of our results}
Throughout this paper, let us denote by $\Si_n$ a uniform separable permutation of size $n$. 
Our goal is to describe the limit of $\Si_n$ when $n$ goes to infinity. 
Our main result gives,
for any $\pi$, the asymptotics of the distribution of $\occ(\pi,\Si_n)$ when $n$ tends to infinity.
We will see in \cref{subsec:Intro_Permutons} an equivalent formulation
in terms of weak convergence of probability measures on the square.

Our main theorem is the following:

\begin{theorem}\label{thm:main}
Let $\Si_n$ be a uniform separable permutation of size $n$.
There exists a collection of random variables $(\Lambda_\pi)$, $\pi$ ranging over all permutations, defined on the same probability space, 
such that for all $\pi$, $0\leq \Lambda_\pi\leq 1$ and:
\begin{enumerate}
\item \label{item:main_variance} If $\pi$ is a separable permutation of size at least $2$, $\Lambda_\pi$ is a non-deterministic random variable. \\
(If $\pi$ is the permutation of size $1$, $\Lambda_\pi=1$.  
If $\pi$ is not separable, $\Lambda_\pi=0$.)
\item \label{item:main_distrib} For all $\pi$, when $n\to+\infty$,
$$
\occ(\pi,\Si_n) \stackrel{d}{\to} \Lambda_\pi,
$$
where $\stackrel{d}{\to}$ denotes the convergence in distribution.
\item \label{item:main_joint} Moreover the convergence holds jointly, that is: for any finite sequence of permutations $(\pi_1,\dots,\pi_r)$,
$$
\left(\occ(\pi_1,\Si_n), \dots,\occ(\pi_r,\Si_n)\right) \stackrel{d}{\to} \left(\Lambda_{\pi_1},\dots,\Lambda_{\pi_r}\right).
$$
(On the right-hand side, the $\Lambda_{\pi_i}$'s are not independent.) 
\end{enumerate}
\end{theorem}

Theorem~\ref{thm:main} is not just an existential result: in this
paper we give for any pattern $\pi$ a construction of
$\Lambda_\pi$ (\cref{dfn:Lambda_pi} p.\pageref{dfn:Lambda_pi}) that
can be briefly explained as follows.
There is a natural way (reviewed in \cref{Section:Extracting}) to extract a (signed) tree with $|\pi|$ leaves from a realization of the (signed) Brownian excursion. 
The variable $\Lambda_\pi$ is the probability that this tree is one of the \emph{separation trees} of $\pi$ (see \cref{sec:PermutationsAndSchroderTrees} for the definition).
 
Statement~\ref{item:main_variance}  of the theorem is proved in \cref{ssec:variance}, 
while Statement~\ref{item:main_joint} (Statement \ref{item:main_distrib} being just a special case) is proved in Sections~\ref{sec:moments} to~\ref{sec:main_proof}, 
following the proof schema detailed in \cref{sec:OutlineProof}. 

Theorem~\ref{thm:main} shows in particular that, for every separable pattern $\pi$ of size $k$,
the number of occurrences of $\pi$ in $\Si_n$ must be renormalized
by $n^k$ to have a non-trivial limit in distribution.
This is in contrast with the result of Janson \cite[Theorem 2.1]{JansonPermutations} for $\sigma$ uniform in $\mathrm{Av}_n(132)$.
Indeed, in his result, the numbers of occurrences of different patterns
of the same size are normalized by different powers of $n$
to have non-trivial limits in distribution.

\medskip

In addition to the convergence in distribution,
we also prove the convergence of all joint moments
(in fact, we first prove the convergence of joint moments,
and then deduce the convergence in distribution).
This is especially interesting since 
the joint moments in the limit can be computed explicitly.

More precisely, all joint moments can be expressed in the limit from expectations of $\Lambda_\pi$'s (see \cref{prop:computing_moments}), 
and the expectation of $\Lambda_\pi$ can be expressed in terms
of the number $N_\pi$ of signed binary trees associated with the permutation $\pi$
(these are also sometimes called {\em separation trees} of $\pi$; see \cref{sec:PermutationsAndSchroderTrees} for their definition). 
The latter is proved in \cref{prop:expectation_Lamdba_pi}, which reads as follows: 
\begin{theorem}\label{thm:mainEsperance}
For any permutation $\pi$ of size $k$,
$$
\mathbb{E}\left[\occ(\pi,\Si_n)\right]\stackrel{n\to +\infty}{\longrightarrow} \frac{N_\pi}{2^{k-1} \Cat_{k-1}},
$$
where we denote by $\Cat_{k} := \frac{1}{k+1}\binom{2k}{k}$ the $k$-th Catalan number.
\end{theorem}

\begin{remark}
Our proof of \cref{thm:mainEsperance} involves the random variables $\Lambda_\pi$,
which are constructed from the Brownian excursion.
We do not know if there exists a proof using only discrete arguments.
\end{remark}

In addition to being defined combinatorially, the numbers $N_\pi$ are easy to compute; see \cref{obs:N_pi}. 
This gives for example
\begin{align*}
  \lim_{n \to +\infty} \mathbb{E}\left[\occ(12,\Si_n)\right]&=
  \lim_{n \to +\infty} \mathbb{E}\left[\occ(21,\Si_n)\right] = \frac12; \\
  \lim_{n \to +\infty} \mathbb{E}\left[\occ(123,\Si_n)\right]&=
  \lim_{n \to +\infty} \mathbb{E}\left[\occ(321,\Si_n)\right] = \frac14; \\
\lim_{n \to +\infty} \mathbb{E}\left[\occ(132,\Si_n)\right]&=
  \lim_{n \to +\infty} \mathbb{E}\left[\occ(213,\Si_n)\right] = \frac18\\
  \lim_{n \to +\infty} \mathbb{E}\left[\occ(231,\Si_n)\right] &= 
  \lim_{n \to +\infty} \mathbb{E}\left[\occ(312,\Si_n)\right] = \frac18.
\end{align*}
Limits of higher (joint) moments can also be computed explicitly, as shown in \cref{prop:computing_moments}.
For example, we obtain
\begin{align*}
  \lim_{n \to +\infty} &\Var \left[\occ(12,\Si_n)\right]= \frac{1}{30};\qquad
  \lim_{n \to +\infty} \Var\left[\occ(132,\Si_n)\right] = \frac{3}{560} ; \\
  \lim_{n \to +\infty} &\mathbb{E}\left[ \occ(12,\Si_n) \cdot \occ(123,\Si_n) \right]= \frac{43}{280}.
\end{align*}

These values have been computed automatically with a Sage program \cite{sage}
that the authors can make available on request. 
We refer to \cref{ssec:Exp_LPi} for a discussion on the computation of joint moments,
both from the theoretical and algorithmic points of view.

For the curious reader, we give the first few moments of $\Lambda_{12}$
(that is the limits of the first few moments of $\occ(12,\Si_n)$):
$$
\mathbb{E}[\Lambda_{12}]=\frac{1}{2},\quad
\mathbb{E}[\Lambda_{12}^2]=\frac{17}{60},\quad
\mathbb{E}[\Lambda_{12}^3]=\frac{7}{40},\quad
\mathbb{E}[\Lambda_{12}^4]=\frac{6361}{55440},\quad
\mathbb{E}[\Lambda_{12}^5]=\frac{1741}{22176}.
$$
We did not recognize the first moments of any "usual" distribution on $[0,1]$.

\subsection{Interpretation of our main result in terms of permutons}
\label{subsec:Intro_Permutons}

We recall the notion of {\em permutons} introduced in \cite{Permutons}.
Note that, in \cite{Permutons},
permutons are called {\em limit permutations} and 
have two equivalent definitions (see \cite[Section 2.3]{Permutons});
the name {\em permuton} was proposed in \cite{GraphonPermuton} 
in analogy with the graph analogue {\em graphon}.
\begin{definition}\label{dfn:permuton}
  A {\em permuton} is a probability measure $\mu$ on the square $[0,1]^2$
  with uniform marginals,
  \emph{i.e.}, for any $a$ and $b$ with $0 \le a \le b \le 1$,
  \[\mu([a,b] \times [0,1]) = \mu([0,1] \times [a,b]) = b-a.\]
\end{definition}
One can associate a permuton $\mu_\sigma$ to a permutation $\sigma$ of size $n$.
We define $\mu_\sigma$ as having density $n$ on each square $[(i-1)/n,i/n] \times [(\sigma(i)-1)/n,\sigma(i)/n]$
(for $1 \le i \le n$) and density $0$ elsewhere;
that is, each of these squares has total weight $1/n$,
uniformly distributed in it.
This is a natural way to encode and rescale a permutation,
very close to the graphical representation that we use on
\cref{fig:simu_closed_or_not,fig:simu_LargeSeparable}. 

Since permutons are measures, it is natural to speak about weak convergence of permutons.
We will see in \cref{sec:Permutons} that our main result,
combined with previous results on permutons, implies the following:
\begin{theorem}
  Let $\Si_n$ be a uniform random separable permutation of size $n$.
  There exists a random permuton $\Mu$ such that
  $\mu_{\Si_n}$ tends to $\Mu$ in distribution 
  in the weak convergence topology.
  Moreover, $\Mu$ is not deterministic.
  \label{thm:Main_Permuton}
\end{theorem}
The distribution of this permuton $\Mu$ is unique,
since it is defined as a limit in distribution.
The proof of the existence of $\Mu$ is not constructive, 
but involves the variables $\Lambda_\pi$,
which are themselves built using a realization of the Brownian excursion.
We therefore call $\Mu$ the {\em Brownian separable permuton}.

\begin{remark}
There are many examples of convergence of large combinatorial structures
towards continuum objects built from Brownian motion (or related processes: the Brownian bridge and excursion).
One can mention 
\begin{itemize}
\item Random mappings and Brownian bridges;
\item Random trees and Brownian trees;
\item Random graphs and Brownian motion.
\end{itemize}
A very good reference for these topics is \cite{Pitman}.
\cref{thm:Main_Permuton} fits naturally in this body of literature.

Note that the fact that separable permutations have a Brownian limit
in some sense should not come as a surprise.
Indeed, separable permutations of size $n$
can be encoded by (signed) \emph{Schr\"oder trees} with $n$ leaves
(see \cref{sec:PermutationsAndSchroderTrees}).
Like for many families of trees, 
the limit of Schr\"oder trees with a fixed number of leaves 
(leaving signs aside)
is related to the Brownian excursion: 
more precisely,
Pitman and Rizzolo \cite{PitmanRizzolo} and Kortchemski \cite{Igor} proved that the \emph{contour} of a uniform Schröder tree with $n$ leaves tends to the Brownian excursion. 
This result is essential in our approach. 
\end{remark}

\subsection{Perspectives}
\label{subsec:perspectives}
We think that the Brownian separable permuton $\Mu$ is an interesting object
and is worth being studied.
In particular, we would like to address the following questions.
\begin{description}
  \item[Construction of $\Mu$]
    At the moment, $\Mu$ is defined in a indirect way,
    as limit of discrete objects.
    Is there a way to define/construct the random measure $\Mu$
    directly in the continuum, {\em e.g.} as a function of the (signed)
    Brownian excursion?
  \item[Properties of $\Mu$] 
    It would be interesting to find some almost-sure properties of $\Mu$.
    Is it absolutely continuous or singular with respect to Lebesgue measure on the square? 
    (Note that, since its marginals are uniform, it cannot have atoms.) 
    One also expects that $\Mu$ is fractal in some sense,
    because of the link with the Brownian excursion
    and the visual aspect of the simulations in \cref{fig:simu_LargeSeparable}. 
    This raises the following question:
    what is the Hausdorff dimension of its support?
\end{description}

\begin{description}
  \item[Universality of $\Mu$] 
    We believe that $\Mu$
    is the limit (in the sense of permutons)
    of many other classes of random pattern-avoiding permutations.
Recall indeed that the class of separable permutations is the first non-trivial case of a substitution-closed class of permutations. 
Such classes are those whose structure is well-understood using the substitution decomposition mentioned earlier in this introduction, 
even more so when they contain a finite number of so-called \emph{simple} permutations (see~\cite[Section 3.2]{Vatter}). 
For any substitution-closed class, the permutations it contains may be represented by trees, called \emph{(substitution) decomposition trees}. 
These generalize signed Schröder trees by introducing other types of vertices, labeled by the simple permutations in the class. 
As the encoding of separable permutations by Schröder trees is crucial in this work,
our results might extend to all substitution-closed classes containing finitely many simple permutations. 
Figure~\ref{fig:simu_closed_or_not} shows two typical permutations: one in a substitution-closed class, one in a class that is not closed under substitution. 
On these examples, it seems clear that the first one looks similar to the separable case, whereas the second one does not. 
\end{description}

\begin{figure}[ht]
\begin{center}
\includegraphics[scale=0.3]{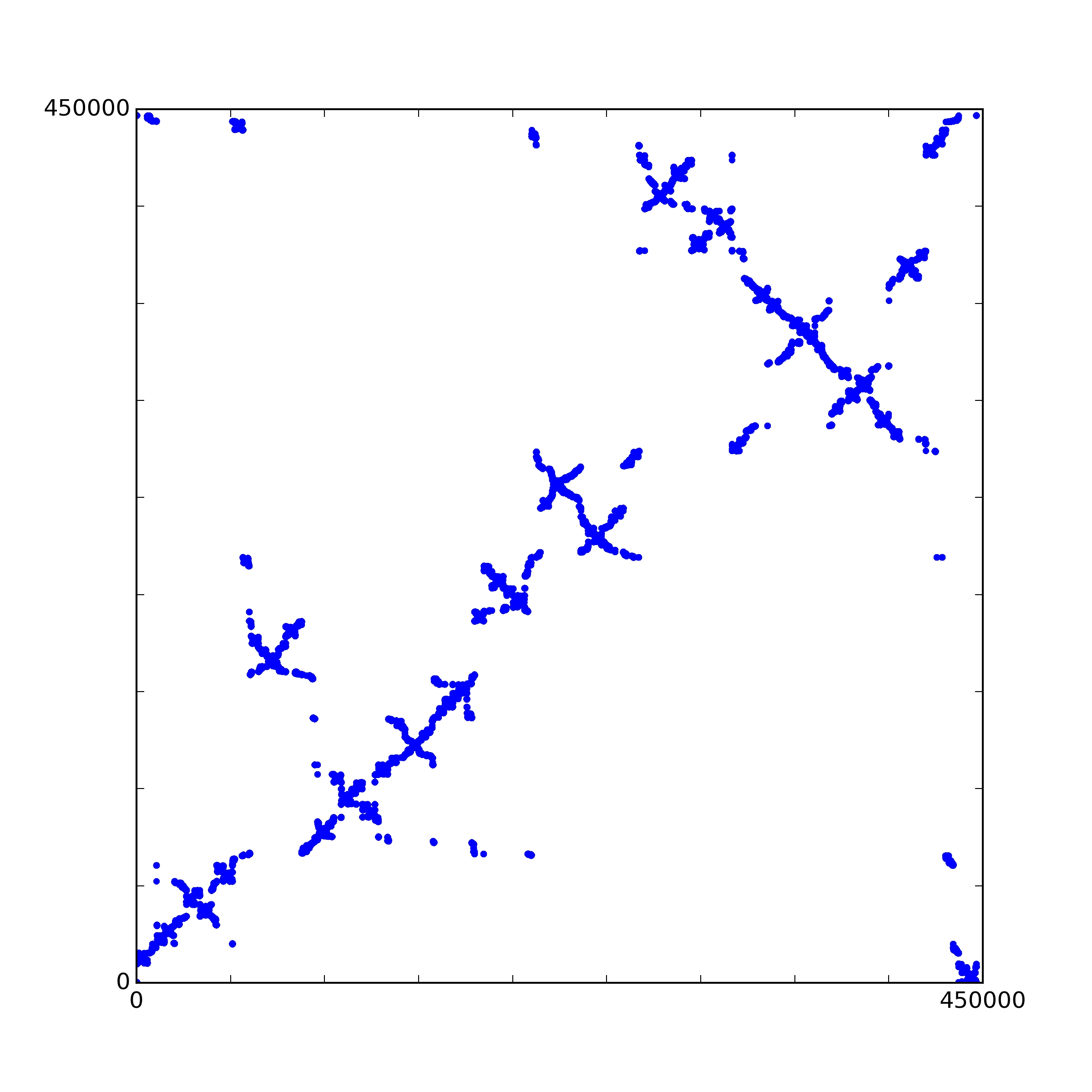} \qquad \includegraphics[scale=0.3]{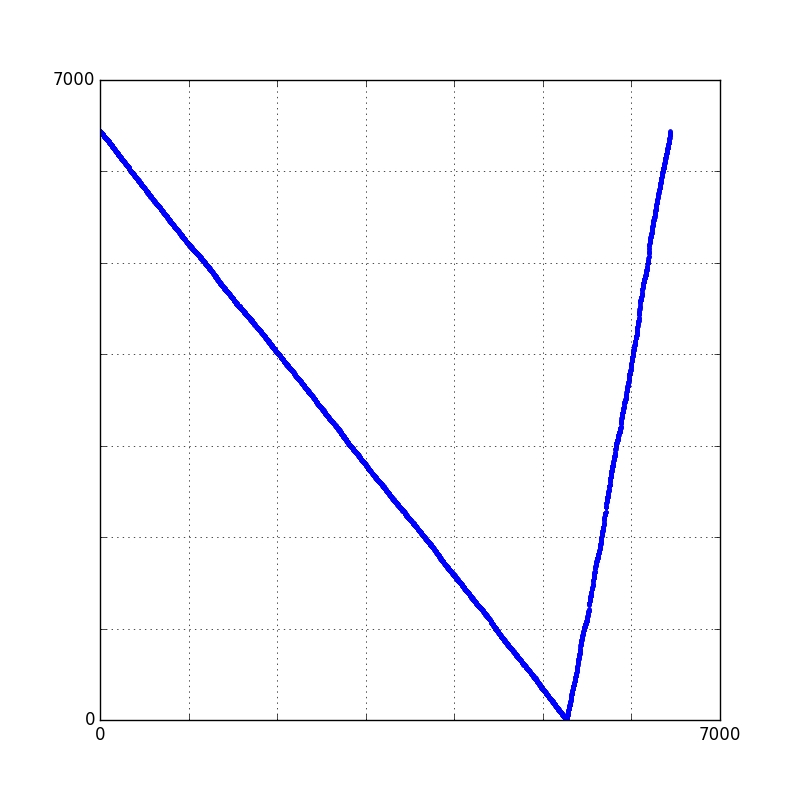}
\end{center}
\caption{On the left, a typical permutation of size $446\, 699$ in the substitution-closed class whose simple permutations are $2413$, $3142$ and $24153$. 
On the right, a typical permutation of size $6\, 441$ in the non-substitution-closed class $\mathrm{Av}(2413, 1243, 2341, 531642, 41352)$.}
\label{fig:simu_closed_or_not}
\end{figure}

\subsection{Outline of the paper}
Our paper is organized as follows.
\begin{itemize}
\item \cref{sec:preliminaries} gives all the preliminaries needed to define
the limit random variables $\Lambda_\pi$
and records a lot of easy useful facts about permutations and trees. 
\item Section~\ref{sec:main_result} defines the variables $\Lambda_\pi$ and presents the structure of the proof of Theorem~\ref{thm:main}.\ref{item:main_joint}.
\item Sections~\ref{sec:moments} to~\ref{sec:signs} go through the several steps of this proof (the outline of the proof itself is given in \cref{sec:OutlineProof}).
\item We gather all the arguments and conclude the proof of Theorem~\ref{thm:main}.\ref{item:main_joint} in Section~\ref{sec:main_proof}. 
\item \cref{sec:Permutons} contains the proof of the permuton interpretation of our main result:
  \cref{thm:Main_Permuton}.
\item Section~\ref{sec:properties_of_moments} studies some properties of $\Lambda_\pi$: 
combinatorial formulas for the moments in \cref{ssec:Exp_LPi} and proof of \cref{thm:main}.\ref{item:main_variance} ($\Lambda_\pi$ is not deterministic) in \cref{ssec:variance}. 
\item We collect several useful properties of the Brownian excursion in \cref{appendix}.
\end{itemize}

%%%%%%%%%%%%%%%%%%%%%%%%%%%%%%%%%%%%%%%%%%%%%%%%%%%%%%%%%%%%%%%%%%%%%%%%%%%%%%%%%%%%%%%%%%%%%%%%%%%%%%%%%%
\section{Permutations, trees and excursions}
\label{sec:preliminaries}
\subsection{Basics on trees}

\begin{definition}\label{dfn:trees}
A \emph{Schröder tree} is either a leaf, or consists of a root vertex 
with an ordered $k$-tuple of subtrees attached to the root ($k \geq 2$),
which are themselves Schröder trees. 

Non-leaf vertices of a tree are called \emph{internal vertices}. 

We consider only \emph{finite} Schröder trees, \emph{i.e.}, those having finitely many leaves and internal vertices. 

In our context, the \emph{size} of a tree $t$ is its number of leaves. 
It is denoted $|t|$, 
whereas $\#t$ denotes the number of vertices of $t$ (including both leaves and internal vertices). 
\end{definition}

Because every internal vertex of a Schröder tree has at least $2$ children, it follows immediately (by induction) that:
\begin{observation}\label{obs:size_Schroeder_tree}
For every Schr\"oder tree $t$, $2|t|\geq \#t +1$.
\end{observation}

A \emph{binary tree} is a Schröder tree where there are exactly $2$ subtrees attached to every internal vertex. 

In this article, we consider unlabeled trees.
Nevertheless, since we work with \emph{plane} trees (\emph{i.e.}, trees in which the subtrees attached to a vertex are ordered),
there is a canonical way to label the {\em leaves} of a tree (from left to right).
Then a subset of the set of leaves of a tree $t$ is canonically represented
by a subset $I$ of $[|t|]$.

\begin{definition}\label{dfn:common_ancestor}
Let $t$ be a Schröder tree and $u$ and $v$ be two vertices of $t$. 
Denote by $r$ the root of $t$.
The \emph{(first) common ancestor} of $u$ and $v$ is the vertex furthest away from $r$ that appears 
on both paths from $r$ to $u$ and from $r$ to $v$ in $t$. 
\end{definition}

\begin{definition}\label{dfn:induced_subtree}
Let $t$ be a Schröder tree. 
Any subset $I$ of the leaves of $t$ induces a subtree $t_I$ of $t$, which is also a Schröder tree, defined as follows: 
\begin{itemize}
 \item the leaves of $t_I$ are the elements of $I$; 
 \item the internal vertices of $t_I$ are the vertices of $t$ that are common ancestors of two leaves in $I$; 
 \item the ancestor-descendant relation in $t_I$ is inherited from the one in $t$; 
 \item the order between the children of an internal vertex of $t_I$ is inherited from $t$. 
\end{itemize}
Note that if $t$ is a binary tree, then so is $t_I$.
\end{definition}

\begin{definition}\label{dfn:signed_trees}
A \emph{signed} Schröder tree is a pair $(t,\eps)$, where $t$ is a Schröder tree and $\eps$ a function 
from the set of internal vertices of $t$ to $\{+,-\}$.
\end{definition}

\begin{definition}\label{dfn:induced_signed_subtree}
Let $(t,\eps)$ be a signed Schröder tree. 
Any subset $I$ of the leaves of $t$ induces a signed subtree $(t_I,\eps_I)$ of $(t,\eps)$, 
where $t_I$ is as in Definition~\ref{dfn:induced_subtree} and $\eps_I$ is the restriction of $\eps$ on the set of internal vertices of $t_I$. 
\end{definition}

\begin{example}\label{ex:SchroderTree}
  Consider the signed Schröder tree on the left-hand side of \cref{fig:Subtree}.
  To ease the presentation, we have indicated the canonical labeling of its leaves from left to right.
  We take the subset of leaves $I=\{2,4,7\}$ (circled on the picture).
  Then the signed subtree $(t_I,\eps_I)$ is represented on the right-hand side of \cref{fig:Subtree}.
  (Again labels on the leaves are here to simplify the presentation, our objects are in essence not labeled.)

\begin{figure}[htbp]
    \begin{center}
      \includegraphics{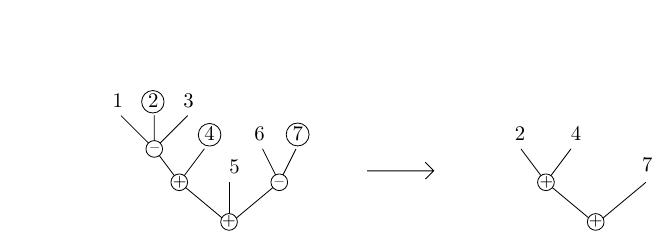}
    \end{center}
    \caption{Subtree of a signed Schröder tree. The labels on the leaves are here only to witness the embedding of the right tree in the left one.}
    \label{fig:Subtree}
\end{figure}
\end{example}  
\subsection{Separable permutations and Schröder trees}\label{sec:PermutationsAndSchroderTrees}

\begin{definition} \label{dfn:Tree2Perm}
Let $\pi$ and $\sigma$ be two permutations of respective sizes $k$ and $\ell$. 
Their \emph{direct sum} and \emph{skew sum} are the permutations of size $k + \ell$ defined by 
\begin{align*}
\oplus[\pi,\sigma] = \pi \oplus \sigma = & \pi_1 \ldots \pi_k (\sigma_1 +k) \ldots (\sigma_\ell +k) \textrm{ and } \\
\ominus[\pi,\sigma] = \pi \ominus \sigma = & (\pi_1+\ell) \ldots (\pi_k+\ell) \sigma_1 \ldots \sigma_\ell. 
\end{align*}
\end{definition}
\noindent The operators $\oplus$ and $\ominus$ being associative,
direct sums and skew sums with $r\geq 2$ components are defined in the obvious way. 
We will use the notation $\oplus[\pi^1,\ldots,\pi^r]$ 
instead of $\pi^1 \oplus \ldots \oplus \pi^r$ (and similarly for $\ominus$).
Examples with $k=3,\ell=2$ are provided in \cref{fig:sum_and_skew},
which also illustrates the graphical interpretation of these operations on permutation diagrams.

\begin{figure}[ht]
    \begin{center}
      \includegraphics[width=7cm]{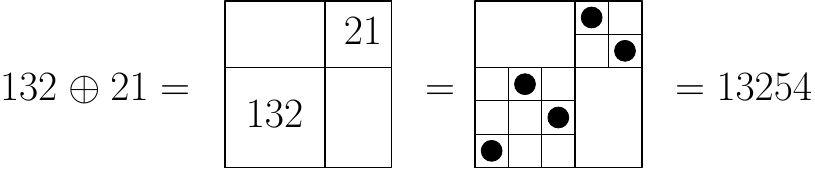} \qquad 
      \includegraphics[width=7cm]{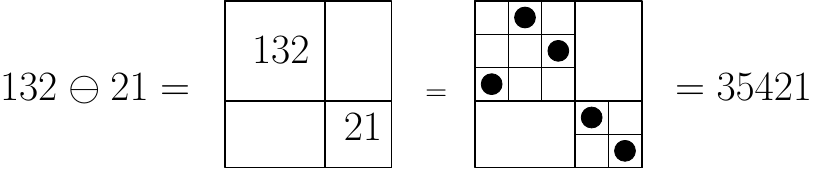}
    \end{center}
\caption{Direct sum and skew sum of permutations on their diagrams. 
Recall that the diagram of a permutation $\sigma$ of size $n$ is the set of points of the Cartesian plane at coordinates $(i,\sigma_i)$, for $1\leq i \leq n$. \label{fig:sum_and_skew}}
\end{figure}

\begin{definition}
Let $(t,\eps)$ be a signed Schröder tree. 
The permutation associated with $(t,\eps)$, denoted $\perm(t,\eps)$, is inductively defined by:
\begin{itemize}
 \item if $t$ is a leaf, then $\perm(t,\eps) =1$; 
 \item otherwise, denoting by $t_1, \ldots, t_r$ the children of the root of $t$ from left to right, 
 and $\eps_i$ the restriction of $\eps$ to the vertices of $t_i$, 
 \[
 \perm(t,\eps) = \begin{cases}
   \oplus[\perm(t_1,\eps_1), \ldots, \perm(t_r,\eps_r)] & \text{ if the root of $t$ has sign $+$;}\\
   \ominus[\perm(t_1,\eps_1), \ldots, \perm(t_r,\eps_r)] & \text{ if the root of $t$ has sign $-$.}
 \end{cases}\]
\end{itemize}
If $\perm(t,\eps)=\sigma$, we say that $(t,\eps)$ is {\em a signed tree} of $\sigma$. 
\label{dfn:separating_tree}
\end{definition}
Note that the $i$-th leaf of $(t,\eps)$ (from left to right) corresponds to the $i$-th element of $\perm(t,\eps)$.
\emph{Binary} signed trees of a permutation $\sigma$ are sometimes called {\em separation trees} of $\si$.

For example, let $(t,\eps)$ be the tree on the left-hand side of \cref{fig:Subtree} of \cref{ex:SchroderTree}.
  Then
  \[
   \perm(t,\eps)=\oplus \bigg[ \oplus\Big[ \ominus[1,1,1],1 \Big] ,1,\ominus[1,1] \bigg]
 = \oplus \Big[ \oplus[321,1],1 ,21 \Big]= \oplus[3214,1,21]=3214576.
  \]
 On the other hand, if we consider the tree $(t_I,\eps_I)$ on the right-hand side of \cref{fig:Subtree}, we have 
 \[ \perm(t_I,\eps_I)= \oplus \Big[ \oplus[1,1],1 \Big] = 123.\]
 In other words, $(t,\eps)$ and $(t_I,\eps_I)$ are signed trees of $3214576$ and $123$, respectively.

\begin{observation}
There is an important consequence of the definition of the signed tree $(t,\eps)$ of a permutation $\sigma$, which is easily seen on our figures. For $i<j$, then $\sigma_i<\sigma_j$ (resp. $\sigma_i>\sigma_j$)  if and only if the common ancestor of leaves $i,j$ has sign $+$ (resp. $-$).
\end{observation}

A consequence of this observation is that the map $\perm$ is compatible with taking substructures:

\begin{observation}\label{obs:patterns_and_subtrees}
Let $(t,\eps)$ be a signed Schröder tree. 
Let $I$ be a subset of $[|t|]$ and $(t_I,\eps_I)$ be the signed subtree of $(t,\eps)$ induced by $I$. 
Then it holds that $\pat_I(\perm(t,\eps)) = \perm(t_I,\eps_I)$. 
\end{observation}

  We continue \cref{ex:SchroderTree}.
  Taking $I=\{2,4,7\}$, we have seen that $\perm(t,\eps)=3214576$ and $\perm(t_I,\eps_I)=123$.
  Thus   \cref{obs:patterns_and_subtrees} asserts that $\pat_{\{2,4,7\}}(3214576)=123$, which is indeed the case.
  This is illustrated by \cref{fig:Schroder_Permutation_AvecMotif}. 

\begin{figure}[ht]
   \begin{center}
      \includegraphics[width=10cm]{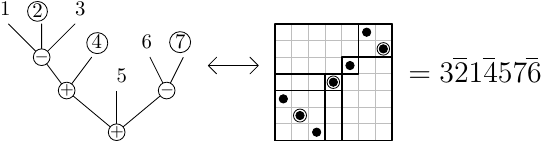}
    \end{center}
\caption{One of the signed trees associated with $\sigma=3214576$. The elements of the pattern $\pat_{\{2,4,7\}}(\sigma)$ are circled/overlined. \label{fig:Schroder_Permutation_AvecMotif}}
\end{figure}

In the introduction, separable permutations have been defined
as avoiding specific patterns.
In fact, they can also been characterized using signed Schröder trees.

\begin{proposition}\label{prop:BBL}
Separable permutations are exactly 
those that can be obtained as $\perm(t,\eps)$ for a signed binary tree $(t,\eps)$,
or equivalently,
those that can be obtained as $\perm(t,\eps)$ for a signed Schröder tree $(t,\eps)$.
\end{proposition}
\begin{proof}
  The first part (with binary trees) has been established in \cite{BBL98}.

  Assume now that a permutation $\si$ can be obtained as $\perm(t,\eps)$ for a signed Schröder tree $(t,\eps)$; we need to prove that $\si$ can also be obtained
  as $\perm(t',\eps')$ for a signed binary tree $(t',\eps')$.
  To obtain $(t',\eps')$ from $(t,\eps)$,
  it is enough to replace every vertex with a label $\delta$ 
  and $k > 2$ subtrees $t_1, \dots, t_k$
  by a binary tree with $k$ leaves on which $t_1, \dots, t_k$ are pending and with all internal vertices labeled $\delta$.
    
  \cref{fig:Form_100euros} (p.\pageref{fig:Form_100euros})
  shows {\em all} binary trees that can be obtained by the
  above construction from a specific signed Schröder tree.
\end{proof}

The advantage of Schröder trees over binary trees is that 
the correspondence tree-permutation can be made one-to-one in more natural way:
\begin{proposition}\label{prop:bij_separables_trees}
The map $\perm$ is a size-preserving bijection between separable permutations 
and signed Schröder trees in which the signs alternate (\emph{i.e.},
all internal vertices at even (resp. odd) distance from the root have the same sign as the root (resp. the opposite sign)).
\end{proposition}

Proposition~\ref{prop:bij_separables_trees} is a consequence of 
the more general substitution decomposition theorem for permutations~\cite[Proposition 2]{AA05}. 
In this context, the unique signed Schröder tree $(t,\eps)$ in which the signs alternate such that $\perm(t,\eps) = \pi$ 
is called the \emph{(substitution) decomposition tree} of $\pi$. 

Let us return to \cref{ex:SchroderTree}. 
The permutation $3214576$ which is induced by the tree on the left-hand side of \cref{fig:Subtree} is separable and the Schröder tree in which signs alternate that corresponds to it is shown in \cref{fig:SchroderAlterne_Permutation}.

\begin{figure}[ht]
   \begin{center}
      \includegraphics[width=10cm]{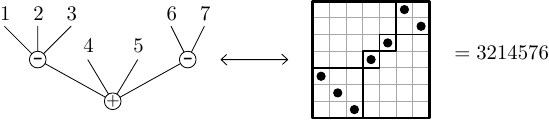}
    \end{center}
\caption{The decomposition tree of a separable permutation. \label{fig:SchroderAlterne_Permutation}}
\end{figure}
    
\medskip

Recall that $\Si_n$ denotes a uniform random separable permutation of size $n$. 
Throughout the paper,
we denote by $T_n$ a uniform random Schröder tree with $n$ leaves. 
From Proposition~\ref{prop:bij_separables_trees}, it follows that: 

\begin{corollary}\label{cor:same_distibution_separables_trees}
Let $B$ be a balanced Bernoulli variable with values in $\{+,-\}$, independent from $T_n$. 
Let $\Eps_n$ be the sign function on the internal vertices of $T_n$,
such that the signs alternate and the root of $T_n$ has sign $B$. 
Then $\Si_n$ has the same distribution as $\perm(T_n,\Eps_n)$. 
\end{corollary}

Pitman and Rizzolo observed (see \cite[Theorem 1]{PitmanRizzolo}) that $T_n$ behaves like a Galton-Watson tree conditioned to have $n$ leaves (we refer to \cite[Section 3]{Remco} for basics of Galton-Watson trees):

\begin{proposition}\label{prop:Schroeder_are_Galton-Watson}
Let $\nu$ be the probability distribution defined by:
\begin{equation*}
\nu(0)=2-\sqrt{2},\qquad \nu(1)=0,\qquad \nu(i)=\left(1-\tfrac{\sqrt{2}}{2}\right)^{i-1}  (\text{ for all }i\geq 2).
\end{equation*}
Then
\begin{itemize}
\item $\nu$ has mean $1$ and variance $4(\sqrt{2}-1)$;
\item $T_n$ has the same distribution as a Galton-Watson tree with offspring distribution $\nu$, conditioned to have $n$ leaves.
\end{itemize}
\end{proposition}
Note that there are actually infinitely many probability distributions $\nu$ 
such that $T_n$ has the same distribution as a Galton-Watson tree with offspring distribution $\nu$, conditioned to have $n$ leaves. 
The one chosen in Proposition~\ref{prop:Schroeder_are_Galton-Watson} is such that $\nu$ has mean $1$. 
The Galton-Watson tree model is then {\em critical},
and most convergence results
in the literature are established in this case
(see, {\em e.g.} \cref{PropLeavesUniform,prop:ConvergenceContour} below).

%This choice maximizes the probability that a Galton-Watson tree with offspring distribution $\nu$ has $n$ leaves, 
%when $n$ tends to infinity.
%This property will be  used in Proposition~\ref{prop:EstimationMarcheR} (p.\pageref{prop:EstimationMarcheR}).

\subsection{Contours and excursions}

\begin{definition}
An \emph{excursion} is a continuous function $f:[0,1] \to [0, +\infty)$ with $f(0)=f(1)=0$.
\end{definition}
% LG 13/02/2017 Modif suite au referee :
Note that with our convention we allow $f(t)=0$ for $t\notin \{0,1\}$.
We can canonically associate with a tree an excursion, called its \emph{contour}. 

\begin{definition}[Contour $C_t$]\label{dfn:contour}
Let $t$ be a (binary or Schröder) tree. Recall that $\#t$ denotes the total number of vertices of $t$, and denote by $V$ the set of vertices of $t$. 

Consider the function $\dfs_t: \{0,1, \ldots, 2 \#t-2\} \to V$ ($\dfs$ stands for {\em depth first search}) defined by: 
\begin{itemize}
 \item $\dfs_t(0)$ is the root of $t$; 
 \item if $\dfs_t(i) =v$, then $\dfs_t(i+1)$ 
 is the leftmost child of $v$ that has not yet been visited, if it exists,
 and the parent of $v$ otherwise. 
\end{itemize}
The \emph{contour} of $t$ is the function $C_t:[0,2 \#t-2] \to [0;+\infty)$ such that 
for all integers $i$ in $[0,2 \#t-2]$, $C_t(i)$ is the distance from the root of $t$ to the vertex $\dfs_t(i)$, 
and $C_t$ is linear between those points. 
\end{definition}

\begin{example}
  An example of a tree $t$ and its contour $C_t$ is given on \cref{fig:Contour}
  (the reader should disregard the signs for the moment).
  The tiny numbers beside each vertex $v$ indicate
  the times $i$ in $[0,2 \#t-2]$ such that $\dfs_t(i)=v$.
  \begin{figure}[ht]
    \begin{center}
      \includegraphics[width=\linewidth]{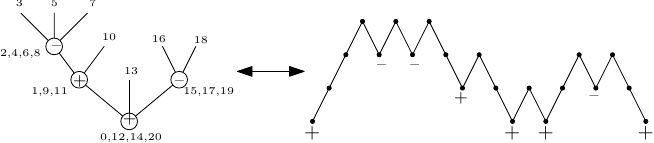}
    \end{center}
    \caption{A (signed) tree and its (signed) contour.}
    \label{fig:Contour}
  \end{figure}
\end{example}

\begin{definition}[Normalized contour $\widetilde{C_t}$]\label{dfn:normalized_contour}
Let $t$ be a Schröder tree. 
The \emph{normalized contour} of $t$ is the function $\widetilde{C_t} : [0,1] \to [0, +\infty)$ defined as follows: 
\begin{equation*}
    \textrm{for all } u \in [0,1], \quad \widetilde{C_t}(u)=\frac{1}{\sqrt{|t|}} C_t\big( (2\#t-2)\, u \big).
\end{equation*}
\end{definition}

By definition, both $C_t$ and $\widetilde{C_t}$ are continuous, so that $\widetilde{C_t}$ is an excursion. 

Some properties of the local maxima and local minima of $\widetilde{C_t}$ follow from its definition, 
and will be essential for our purpose. 

\begin{observation}
\label{obs:peaks=leaves}
Let $t$ be a tree with $|t|>1$.
If there is a {\em local maximum} of $\widetilde{C_t}$ at $u$,
then $u = \frac{i}{(2\#t-2)}$ for some integer $i$ and $\dfs_t(i)$ is a {\em leaf} of $t$.
This defines a bijection between the leaves of $t$ and the local maxima of $\widetilde{C_t}$. 
\end{observation}

This allows us to identify the leaves of $t$ with the $x$-coordinates of the local maxima of $\widetilde{C_t}$. 
We will often do so in the sequel, and we introduce the following notation:

\begin{definition}[$\ll_t$]
\label{dfn:leaves_coordinates}
For any tree $t$ with $|t|>1$, we define $\ell_1 <\dots<\ell_{|t|}$ to be the $x$-coordinates of the local maxima of $\widetilde{C_t}$ 
(which correspond to the leaves of $t$), 
and we set $\ll_t=\{\ell_1,\dots,\ell_{|t|}\} \subset [0,1]$.
\end{definition}

\begin{observation}
\label{obs:valleys=internal_nodes}
Let $t$ be a tree with $|t|>1$.
If there is a {\em local minimum} of $\widetilde{C_t}$ in $u$, 
then $u = \frac{i}{(2\#t-2)}$ for some integer $i$ and $\dfs_t(i)$ is an {\em internal vertex} of $t$. 
Note however that a single internal vertex can correspond to several local minima of $\widetilde{C_t}$.
\end{observation}

Using the description of Schröder trees as Galton-Watson trees (see \cref{prop:Schroeder_are_Galton-Watson}), 
Pitman and Rizzolo \cite{PitmanRizzolo} and Kortchemski \cite{Igor} have proved that 
the normalized contour of a uniform Schröder tree converges in distribution to 
a multiple of the Brownian excursion\footnote{ 
For the readers who need background on the Brownian excursion, 
we have collected in \cref{appendix} (one of) its definition(s), 
together with some useful properties, and bibliographical references.}.
Throughout the paper, we denote by $\Exc$ the Brownian excursion.
\begin{proposition}
\label{prop:ConvergenceContour}
The following convergence holds in distribution in the space $\CCC[0,1]$ of real-valued continuous functions on $[0,1]$:
\[
\left(\widetilde{C_{T_n}}(u)\right)_{u\in[0,1]} \stackrel{n\to+\infty}{\longrightarrow} \left(\lambda \cdot \Exc(u)\right)_{u\in[0,1]},
\]
where $\lambda=\sqrt{2+3/\sqrt{2}}$.
\end{proposition}

\bigskip

We now define signed analogues of contours and excursions. 

\begin{definition}
A {\em signed excursion} is a pair $(f,s)$ where $f$ is an excursion and $s$ a function
from the set of the local minima of $f$ to $\{+,-\}$. 
\end{definition}

From Observation~\ref{obs:valleys=internal_nodes}, we may define the signed contour of a signed tree as follows: 

\begin{definition}
\label{dfn:signed_contour}
Let $(t,\eps)$ be a signed Schröder tree. 
The \emph{signed contour} of $(t,\eps)$ is the pair $(\widetilde{C_t},s)$, 
where $s$ associates to every local minimum of $\widetilde{C_t}$ reached in $u = \frac{i}{(2\#t-2)}$ 
the sign of the internal vertex $\dfs_t(i)$ of $t$. 
\end{definition}

\begin{example}
  The reader is now invited to look again at \cref{fig:Contour},
  taking the signs into consideration.
\end{example}

We may also define the signed Brownian excursion by: 
\begin{definition}
\label{dfn:Brownian_exc}
The signed Brownian excursion is the pair $(\Exc,S)$
where $\Exc$ is the Brownian excursion
and $S$ is the function assigning balanced independent signs on the local minima of $\Exc$.
\end{definition}
The probability space and the $\sigma$-algebra on which this object is constructed
will be introduced in \cref{sec:measurability}.

\begin{remark}
  Some readers may have seen the term \emph{``signed Brownian excursion''}
  refer to an excursion of Brownian motion whose length has been normalized to 1,
  but that may be
  either positive on $(0, 1)$ or negative on $(0, 1)$.
  The meaning of this expression in our paper is different.
\end{remark}
\begin{remark}\label{rem:cv_to_signedBrownianExc}
Consider a uniform random Schröder tree $T_n$ of size $n$ and 
a random sign function $\Eps_n$ on its internal vertices
defined as in \cref{cor:same_distibution_separables_trees}.
If we use \cref{dfn:signed_contour} on $(T_n,\Eps_n)$,
we get a random signed excursion $(\widetilde{C_{T_n}},S_n)$.
Since $\widetilde{C_{T_n}}$ converges towards 
a multiple of the Brownian excursion $\Exc$,
a natural question is the convergence of $(\widetilde{C_{T_n}},S_n)$
towards the signed Brownian excursion.
A first difficulty in proving such a result would be to find a good topology on signed excursions.
Another, maybe deeper, obstacle is that the signs assigned 
to the local minima in $(\widetilde{C_{T_n}},S_n)$ are far from independent:
recall that on the corresponding tree, signs on internal vertices alternate.
Our work shows the convergence of $(\widetilde{C_{T_n}},S_n)$ to $(\Exc,S)$ in a weaker sense. 
Indeed, the next subsection explains how to extract a permutation of any fixed size from a signed excursion, 
and Theorem~\ref{thm:main}.\ref{item:main_joint} proves
the joint convergence in distribution of the permutations extracted from $(\widetilde{C_{T_n}},S_n)$ to those extracted from $(\Exc,S)$.
\end{remark}

\subsection{Extracting trees and permutations from excursions}\label{Section:Extracting}

Given $k$ points in an excursion,
there is a natural way to associate a tree with this data.
We now explain this construction.
This is classical in the random tree literature and can for example be found
in Le Gall's survey \cite[Section 2.5]{LeGall} 
(except that he considers geometric trees, \emph{i.e.}, with edge-lengths,
while we are only interested in combinatorial trees).
\begin{definition}[$\Tree$]
\label{dfn:Tree}
Let $f$ be an excursion and $\xx=\{x_1,\dots,x_k\}$ be a set of $k$ points in $[0,1]$.
Without loss of generality, we assume $x_1 < \dots < x_k$. 
For all $1 \leq i \leq k-1$, let $m_i$ be the minimum value of $f$ on $[x_i,x_{i+1}]$.

We associate to $(f,\xx)$ a Schröder tree $\Tree(f,\xx)$, defined recursively as follows. 

If $k=1$ then $\Tree(f,\xx)$ is a leaf. 

Otherwise, let $m = \min_i \{m_i\}$ and denote by $i_1 < i_2 <  \ldots <i_p$ all the indices $i_j$ such that $m_{i_j} =m$. 
For notational convention, let $i_0 = 0$ and $i_{p+1} =k$.
For all $0\leq j \leq p$, define $\xx^{(j)}=\{x_{i_{j}+1},\cdots,x_{i_{j+1}}\}$.
Then $\Tree(f,\xx)$ is the tree whose root has arity $p+1$ and whose children are $\Tree(f,\xx^{(0)}), \ldots, \Tree(f,\xx^{(p)})$.
\end{definition}

\begin{observation}
\label{obs:distinctMinImpliesBinary}
If all $m_i$ are distinct, then $\Tree(f,\xx)$ is a binary tree. 
\end{observation}

\begin{example} 
  An example of the construction presented in Definition~\ref{dfn:Tree} is given on \cref{fig:extraction}
  (the reader should disregard the signs for the moment).
  \begin{figure}[ht]
    \begin{center}
      \includegraphics{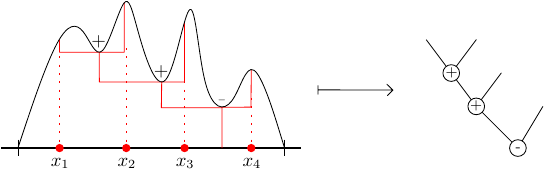}
    \end{center}
    \caption{Extracting a (signed) tree from a (signed) excursion.}
    \label{fig:extraction}
  \end{figure}
\end{example}

\begin{observation}
By construction, the leaves of $\Tree(f,\xx)$ are in one-to-one correspondence with the $x_i$'s. 
Moreover, to each $m_i$ corresponds an internal vertex of $\Tree(f,\xx)$;
note however that several $m_i$'s may correspond to the same vertex of $\Tree(f,\xx)$ 
(this happens exactly for $m_i$ and $m_j$ having the same value such that no $m_k$ for $i < k < j$ is smaller). 
In addition, the common ancestor of the leaves corresponding to $x_i$ and $x_{i+1}$ in $\Tree(f,\xx)$ 
is the internal vertex that corresponds to $m_i$. 
More generally, the common ancestor of the leaves corresponding to $x_i$ and $x_{j}$, with $i<j$, 
is the internal vertex that corresponds to $ \min_{i \leq h <j} \{m_h\}$. 
\end{observation}

An interesting special case consists
in considering the normalized contour $\widetilde{C_t}$ of a tree $t$, 
and choosing $\xx$ to be a subset of $\ll_t$ (defined in Definition~\ref{dfn:leaves_coordinates}).

\begin{observation}
\label{obs:ExtractedTree=Subtree}
Let $t$ be a Schröder tree, and $I$ be a subset of its set of leaves. 
Let $\xx$ be the subset of $\ll_t$ corresponding to the $x$-coordinates of these leaves. 
Then $\Tree(\widetilde{C_t},\xx) = t_I$, that is to say 
the tree extracted from $\xx$ in the contour of $t$ 
is the same as the subtree of $t$ induced by $I$. 
\end{observation}

A similar statement holds when extracting trees in excursions that are not necessarily contours; 
this observation (or rather its signed analogue) will be useful in Section~\ref{sec:moments}.

\begin{observation}
\label{obs:ExtractedTree=Subtree_generalCase}
Let $f$ be an excursion and $\xx=\{x_1,\dots,x_k\}$ be a set of $k$ points in $[0,1]$. 
Let $t = \Tree(f,\xx)$. 
Let $\yy$ be a subset of $\xx$, and $I$ be the corresponding subset of the set of leaves of $t$. 
Then the tree extracted from $\yy$ in $f$ is the substree of $t$ induced by $I$, \emph{i.e.}, $\Tree(f,\yy)=t_I$.
\end{observation}

We now discuss the signed analogue of the extraction of trees from excursions.
\begin{definition}[$\Tree_{\pm}$]
\label{dfn:Tree_pm}
Let $(f,s)$ be a signed excursion, and $\xx=\{x_1,\dots,x_k\}$ be a set of $k$ points in $[0,1]$, assumed to satisfy $x_1 < \dots < x_k$. 
As in Definition~\ref{dfn:Tree}, for all $1 \leq i \leq k-1$, denote by $m_i$ the minimum value of $f$ on $[x_i,x_{i+1}]$. 

We assume that the following condition holds: 
\begin{equation}
\begin{cases}
  \bullet \ m_i \textrm{ is reached in the interior of } [x_i,x_{i+1}]\textrm{, \emph{i.e.},} \\
  \textrm{for all } i, \textrm{ there exists } y \in (x_i,x_{i+1}) \textrm{ such that } f(y)=m_i. \\
\bullet \ \textrm{For all } i\textrm{, all the local minima with value } m_i \textrm{ on } (x_i,x_{i+1}) \\ 
\textrm{are given the same sign by }s \textrm{, abusively denoted } s(m_i). \\
\bullet \ \textrm{If } m_i = m_j \textrm{ and there is no } i < k <j \textrm{ such that } m_k < m_i, \\
\textrm{then } s(m_i) = s(m_j).
\tag{$\mathcal{C}$} \label{condition:unique_signs_on_minima}
\end{cases}
\end{equation}
Then $\Tree_{\pm}(f,s,\xx)$ is defined like  $\Tree(f,\xx)$, 
except that, at each stage of the construction,
we associate with the root the sign of the corresponding local minimum (or minima) of $f$.
Doing so, at the end of the construction,
every internal vertex of $\Tree_{\pm}(f,s,\xx)$ has a sign.
\end{definition}

Condition~\eqref{condition:unique_signs_on_minima} ensures that the signs are well-defined. 
Indeed, the first condition guarantees that $f$ has a minimum on $(x_i,x_{i+1})$, which is therefore a local minimum of $f$ (of value $m_i$), 
and consequently has a sign.
The second condition implies that,
for all $i$, all minima of $f$ on $(x_i,x_{i+1})$ have the same sign, 
and the third one ensures that (with the notation of Definition~\ref{dfn:Tree}) all $m_{i_j}$ 
have the same sign, and recursively so. 
\medskip

Because a separable permutation is associated with each signed Schröder tree 
(see \cref{dfn:separating_tree} and \cref{prop:BBL}), 
we can extract a separable permutation from a signed excursion in which $k$ points are chosen. 

\begin{definition}
\label{dfn:Perm}
Let $(f,s)$ be a signed excursion and $\xx=\{x_1,\dots,x_k\}$ be a set of $k$ points in $[0,1]$. 
Assume $(f,s,\xx)$ satisfies Condition~\eqref{condition:unique_signs_on_minima}. 
The separable permutation associated with $(f,s,\xx)$ is 
\[\Perm(f,s,\xx) := \perm \big( \Tree_{\pm}(f,s,\xx) \big). \]
\end{definition}

\begin{example}
  The reader is now invited to look again at \cref{fig:extraction},
  taking the signs into consideration.
  The permutation associated to the signed tree on the right is $2341$,
  so that in this case, $\Perm(f,s,\xx)=2341$.
\end{example}

\begin{observation}\label{obs:Perm_is_defined}
When considering a signed tree $(t,\eps)$ and its signed contour $(\widetilde{C_t},s)$, 
and choosing $\xx$ as a subset of $\ll_t$, then Condition~\eqref{condition:unique_signs_on_minima} 
is always satisfied. Indeed, all local minima of $\widetilde{C_t}$ that are required to have the same sign 
correspond to the same internal vertex of $t$. 
\end{observation}

Moreover, we have the signed equivalents of
Observations~\ref{obs:ExtractedTree=Subtree} and~\ref{obs:ExtractedTree=Subtree_generalCase}.

\begin{observation}
\label{obs:Tree_pm}
Let $(t,\eps)$ be a signed Schröder tree and $(\widetilde{C_t},s)$ be its signed contour. 
Let $I$ be a subset of the set of leaves of $t$,  
and $\xx$ be the corresponding subset of $\ll_t$. 
Then $\Tree_\pm(\widetilde{C_t},s,\xx) = (t_I,\eps_I)$, \emph{i.e.},
the signed tree extracted from $\xx$ in the signed contour of $(t,\eps)$ is defined and 
is the same as the signed subtree of $(t,\eps)$ induced by $I$. 
\end{observation}

\begin{observation}
\label{obs:Tree_pm_generalCase}
Let $(f,s)$ be a signed excursion and $\xx=\{x_1,\dots,x_k\}$ be a set of $k$ points in $[0,1]$. 
Assume that $\Tree_\pm(f,s,\xx)$ is defined and let $(t,\eps) = \Tree_\pm(f,s,\xx)$. 
Let $\yy$ be a subset of $\xx$, and let $I$ be the corresponding subset of the set of leaves of $(t,\eps)$. 
Then the signed tree extracted from $\yy$ in $(f,s)$ is defined and is the substree of $(t,\eps)$ induced by $I$, \emph{i.e.}, $\Tree_\pm(f,s,\yy)=(t_I,\eps_I)$.\vspace{-2mm}
\end{observation}

\begin{example}\label{ex:sous-arbre-extrait}
  For $(f,s,\xx)$ as in \cref{fig:extraction} and $I = \{1,3\}$, then $\Tree_\pm(f,s,\{x_1,x_3\})=(t_I,\eps_I) = \begin{array}{c} \tikz{
\begin{scope}[scale=.1,rotate=180,level distance=4cm,sibling distance=3cm]
\node[circle, draw, inner sep=0pt] {\tiny $+$} 
 child {node[inner sep=0pt] {~}}
 child {node[inner sep=0pt] {~}};
\end{scope}}\end{array}$.\vspace{-2mm}
\end{example}

Getting to permutations, we can combine Observations~\ref{obs:patterns_and_subtrees} and~\ref{obs:Tree_pm} to obtain: 

\begin{observation}
\label{obs:functionPerm=patterns}
Let $(t,\eps)$ be a signed Schröder tree, and $(f,s)$ be its signed contour. 
Consider a subset $I$ of the set of leaves of $t$, and 
denote by $\xx$ the corresponding subset of $\ll_t$. We have:
\[
\Perm(f,s,\xx) = \pat_I(\perm(t,\eps)).  
\]
\end{observation}

\begin{example}
  We consider the signed Scröder tree $(t,\eps)$ from \cref{ex:SchroderTree}
  and the set $I=\{2,4,7\}$.
  We have in this case $\pat_I(\perm(t,\eps))=123$.
  On the other hand, the construction of $\Perm(f,s,\xx)$
  is illustrated on \cref{fig:Extraction_From_Signed_Contour} and
  also yields the permutation $123$.
\end{example}
\begin{figure}[ht]
  \begin{center}
    \includegraphics{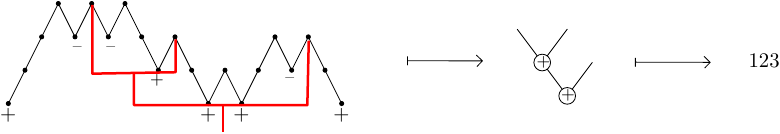}
  \end{center}
  \caption{Extracting a signed tree from the signed contour of a bigger tree
  and a subset of its leaves.}
  \label{fig:Extraction_From_Signed_Contour}
\end{figure}

Another situation where we will extract a permutation from a signed excursion is the following.
\begin{observation}
\label{obs:Brownian_satisfies_C}
 Let $(f,s)$ be a realization of the signed Brownian excursion $(\Exc,S)$ and $\xx$ be a set of $k$ points taken uniformly and independently at random in $[0,1]$. Then 
Condition~\eqref{condition:unique_signs_on_minima} is satisfied with probability $1$ for $(\Exc,S,\xx)$,
and  the extracted tree is binary (see \cref{Cor:OneSided,Lemma:MinimumSameLevel} in Appendix).
\end{observation}

\subsection{Measurability issues}
\label{sec:measurability}
In this section, we construct a suitable probability space and $\sigma$-algebra
for the signed Brownian excursion.
A possibility would be to take a Brownian excursion $\Exc$ and 
a family of i.i.d. signs $(s_x)_{x \in [0,1]}$,
independent from the excursion (and then consider only those indexed by minima of $\Exc$).
With this construction,
the function $(f,s,\xx) \mapsto \Tree_\pm(f,s,\xx)$ would however not be measurable.
We therefore give a slightly more subtle construction.

We fix an enumeration of the rational intervals in $[0,1]$,
i.e. a bijection
$$\Psi: \Z_{>0} \to \{[p,q] \in \Q^2,\, 0 \le p <q \le 1\}.$$
Denote $\Psi(i)=[p_i,q_i]$.
Let $b$ be the position of a local minimum in an excursion $f$.
We say that $b$ is associated with $[p_i,q_i]$ if
$b=\argmin_{[p_i,q_i]} f$ and if {\em $i$ is minimal with this property}.
Clearly, such an $i$ always exists.

To construct the signed Brownian excursion $(\Exc,S)$,
we now take a Brownian excursion $\Exc$
and an i.i.d. sequence of balanced signs $(S_i)_{i \ge 1}$, independent from $\Exc$.
Consider a (strict) local minimum $b$ of $\Exc$
and let $i$ be such that $b$ is associated with $[p_i,q_i]$.
We then think as $S_i$ as being the sign of $b$.

A given interval $[p_i,q_i]$ can be associated with at most one $b$
(the position of the minimum of $\Exc$ on $[p_i,q_i]$;
by \cref{Lemma:MinimumSameLevel}, almost surely,
this minimum is unique),
which guarantees the independence of the signs of the different local minima.
Note that some interval $[p_i,q_i]$ may be associated with no
local minimum $b$ (when the position of the minimum of $\Exc$ on $[p_i,q_i]$
is associated with $[p_j,q_j]$, for some $j<i$),
which does not create any problem.

We claim that, on this probability space,
the function $(f,s,\xx) \mapsto \Tree_\pm(f,s,\xx)$ is measurable.
The unsigned version $(f,\xx) \mapsto \Tree(f,\xx)$ is a classical
object in the literature \cite[Section 2.5]{LeGall}
and it is easy to prove its measurability.
The difficulty comes from the sign, and hence
we focus on the case $|\xx|=2$, the general one following easily.
We have
\begin{multline*}
\left\{ \left( f,s,x_1,x_2 \right):
\Tree_\pm(f,s,x_1,x_2) = \begin{array}{c} \tikz{
\begin{scope}[scale=.1,rotate=180,level distance=5cm,sibling distance=3cm]
\node[circle, draw, inner sep=0.2pt] {+} 
 child {node[inner sep=0pt] {~}}
 child {node[inner sep=0pt] {~}};
\end{scope}}\end{array} \right\}\\
= 
\bigcup_{i \ge 1}\left\{ 
\left( f,s,x_1,x_2 \right):
\argmin_{[x_1;x_2]} f \text{ is associated with }[p_i,q_i]
\text{ and }s_i=+ \right\}.
\end{multline*}
It is easy to check that 
this is a countable union of measurable sets, hence a measurable set itself.
Therefore our measurability claim is proved.

The function $\Perm$ is then also measurable since
$\Perm=\perm \circ \Tree_\pm$.
These are the only two functions of the signed Brownian excursion
of interest for this paper.

\section{The variables $\Lambda_\pi$: definition, properties, and proof schema of the main result}
\label{sec:main_result}
\subsection{Description of the limit variables $\Lambda_\pi$}

Recall from the introduction that $\occ(\pi,\Si_n)$ denotes the proportion of occurrences of $\pi$ in a uniform random separable permutation of size $n$. 
Our main result (see Theorem~\ref{thm:main}.\ref{item:main_distrib} and \ref{item:main_joint} p.\pageref{thm:main}) is the (joint) convergence in distribution of $\occ(\pi,\Si_n)$ 
towards a random variable $\Lambda_\pi$ that we now define. 

We denote by $\One[A]$ the characteristic function of an event $A$.

\begin{definition}
\label{dfn:Lambda_pi}
Let $\pi$ be a pattern and $(\Exc,S)$ be the signed Brownian excursion. 
Let also $k=|\pi|$. Let $X_1, X_2,  \ldots X_k$ be $k$ uniform and independent points in $[0,1]$,
independent from $(\Exc,S)$ and set $\XX=\{X_1, X_2,  \ldots X_k\}$. 
The random variable $\Lambda_\pi$ is defined by
\begin{equation}
\Lambda_\pi=\mathbb{P}\left(\ \Perm(\Exc,S,\XX)= \pi\ {\huge |}\ \Exc,S \right).
  \label{eq:def_Lambda_pi}
\end{equation}
Note that this may be rephrased as $\Lambda_\pi=\esper \left( \One[\Perm(\Exc,S,\XX)= \pi]\ |\ \Exc,S \right)$.

More precisely, 
we define the infinite-dimensional random vector $(\Lambda_\pi)_{\pi}$,
indexed by {\em all patterns $\pi$ of all sizes} as follows:
for each $\pi$, let $\XX^{(\pi)}$ be a set of $|\pi|$ independent uniform random
variables in $[0,1]$, taken independently for different patterns $\pi$,
then we set
\begin{equation}
  (\Lambda_{\pi})_\pi=\esper \bigg( \left( \One[\Perm(\Exc,S,\XX^{(\pi)})= \pi]\ \right)_\pi \ \bigg|\ \Exc,S  \bigg).
  \label{eq:def_Lambda_pi_vecteur}
\end{equation}
In this definition, and throughout the paper,
the event ``$\ \Perm(\Exc,S,\XX)=\pi$''
should be understood as 
``$\Perm(\Exc,S,\XX)\text{ is defined and is equal to }\pi$'' .
\end{definition}

Since the finite-dimensional marginals of $(\Lambda_{\pi})_\pi$ appear
in our main theorem, rather than the whole vector, let us write down explicitly
the definition of these finite-dimensional laws:
for any patterns $\pi_1$, \ldots, $\pi_r$ we have
\begin{equation}
(\Lambda_{\pi_1},\dots, \Lambda_{\pi_r})=\esper\bigg(\ \big(
\One[\Perm(\Exc,S,\XX^{(\pi_1)})= \pi_1],\, \dots, \, \One[\Perm(\Exc,S,\XX^{(\pi_r)})= \pi_r]\big)\, \ \bigg|\ \Exc,S \bigg).
  \label{eq:def_Lambda_pi_finidim}
\end{equation}

Note that in the definition of $\Lambda_\pi$, we condition on the {\em random variable} $(\Exc,S)$,
so that $\Lambda_\pi$ is a random variable itself.
Moreover, in \cref{eq:def_Lambda_pi_vecteur,eq:def_Lambda_pi_finidim},
all coordinates are defined by conditioning on the {\em same realization}
$(\Exc,S)$. 
Thus, the variables $\Lambda_\pi$ corresponding to different patterns $\pi$
are {\em not independent}.

\medskip

The reader less familiar with probability theory
might be more comfortable with the following equivalent description of 
$\Lambda_\pi$. 
For simplicity, we only discuss below the definition of the one-dimensional
random variable $\Lambda_\pi$ (for a fixed pattern $\pi$),
and not of the full infinite-dimensional vector $(\Lambda_\pi)_{\pi}$.

Let us define a function $\Psi_\pi$ on signed excursions as follows:
\[\Psi_\pi(f,s) = \mathbb{P}\big(\, \Perm(f,s,\XX)=\pi \, \big),\]
where  $\XX$ is a set of $|\pi|$ uniform and independent points in $[0,1]$.
Then we set \hbox{$\Lambda_\pi:=\Psi_\pi(\Exc,S)$} to be the image of the signed Brownian excursion by $\Psi_\pi$.

In other words, in \cref{eq:def_Lambda_pi},
the probability is taken with respect to $\XX$, 
while we consider a realization of $(\Exc,S)$.
In such situations, we will sometimes use a superscript on $\proba$
to record the source of randomness: namely we write \cref{eq:def_Lambda_pi} as
\[\Lambda_\pi=\mathbb{P}^{\XX}\left(\ \Perm(\Exc,S,\XX)=\pi\ \right).\]
  Similarly we use superscripts on expectation symbols $\esper$ to indicate
the source of randomness. 

\begin{observation}
\label{obs:expectation_of_expectation}
  With this notation, we have the obvious compatibility relation:
  \[ \esper^{\Exc,S} \Big[ \esper^{\XX} \big( g(\Exc,S,\XX) \big) \Big] = \esper^{\Exc,S,\XX} \big[ g(\Exc,S,\XX) \big],\]
  for any function $g$ such that these quantities are defined.
  If $g$ is the indicator function of an event $A$, this can be rewritten as:
  \[ \esper^{\Exc,S} \Big[ \proba^{\XX} \big( A \big) \Big] = \proba^{\,\Exc,S,\XX} \big( A  \big).\]
\end{observation}

To finish this section, we discuss trivial cases of our main theorem,
when $\pi$ is not separable 
and when $\pi$ is the permutation of size $1$.

\begin{remark}\label{rem:nonseparable}
Observe that if $\pi$ is not a separable permutation,
from \cref{prop:BBL} it cannot be obtained as $\Perm(\Exc,S,\XX)$ and thus $\Lambda_\pi$
is identically equal to $0$ in this case.
Clearly, $\occ(\pi,\Si_n)$ is also identically equal to $0$ in this case
since a separable permutation cannot have a non-separable pattern
(permutation classes are by definition downward-closed 
for the pattern relation).
\end{remark}

\begin{remark}\label{rem:motif1}
If $\pi =1$ is the only permutation of size $1$, 
then regardless of $\XX = (X_1)$ we have that $\Perm(\Exc,S,\XX)$ is the permutation of size $1$. 
Consequently, $\Lambda_\pi$ is identically equal to $1$ in this case.
Similarly, $\occ(\pi,\Si_n)$ is also identically equal to $1$, 
since every element of a permutation is an occurrence of the pattern $\pi=1$. 
\end{remark}

Thus our main theorem is vacuous in the special cases $\pi$ not separable or $\pi =1$.

\subsection{The leaf distribution function of a tree} 
%\subsection{Analogy between $\Lambda_\pi$ and $\occ(\pi,\Si_n)$}
\label{ssec:distribution_function}
Before going further,
we need a detour through distribution functions,
to encode the positions of the leaves 
in the renormalized contour of a tree $t$.
\medskip

A distribution function $F$ is a right-continuous non-decreasing function from $\R$ to $[0,1]$ with
$ \lim\limits_{x \to - \infty} F(x)=0$ and $ \lim\limits_{x \to +\infty} F(x)=1$.
A real-valued random variable $X$ has distribution function $F$ if,
for all $x$, one has: $F(x)=P(X \le x)$.
In the sequel, we only consider distribution functions $F$ such that $F(0)=0$ and $F(1)=1$ 
(equivalently, the associated random variables have values in $[0,1]$); 
and we identify $F$ with its restriction on the domain $[0,1]$. 

For such a distribution function $F$, the pseudo-inverse $F^*$ of $F$ is
defined as follows: for $u$ in $[0,1]$, we set
$F^*(u)=
\inf \{x \in [0,1] : F(x)\geq u\}$.
One can check 
that for all $u\in [0,1]$ and $\eps>0$, 
$$
F(F^*(u))\geq u\geq F(F^*(u)-\eps).
$$
Moreover, if $U$ is a uniform random variable on $[0,1]$,
then $F^*(U)$ has distribution function $F$.

\medskip

The following distribution functions will be of particular interest in this work.
\begin{itemize}
    \item For the uniform distribution on $[0,1]$, we have
       $F_U(x) = \min(x,1) \One[x \ge 0]$.
      \item With any tree $t$ with set of leaves $\ll_t=\{\ell_1,\cdots,\ell_{|t|}\}$
        ($|t|>1$), we associate its \emph{leaf distribution function} $F_t$ defined by
\begin{equation}
    F_t(x)= \frac{1}{|t|} \sum_{i=1}^{|t|} \One[\ell_i \le x].
    \label{eq:EqDFLeaves}
\end{equation}
%where $\mathbf{1}_{[\text{condition}]}$ is $1$ if the condition is fulfilled and $0$ otherwise. 
By definition, taking a random variable with distribution $F_t$ corresponds to choosing a leaf of $t$ uniformly at random
(more precisely the $x$-coordinate $\ell_i$ of the corresponding leaf in the normalized contour of $t$). 
\cref{fig:leaf_distrib_function} shows the leaf distribution function associated with the (unsigned) tree of \cref{fig:Contour} p.\pageref{fig:Contour}. 
\end{itemize}

\begin{figure}[ht]
\begin{center}
\begin{tikzpicture}[scale=4]
\draw[->] (-0.1,0) -- (2.1,0);
\draw[->] (0,-0.1) -- (0,1.1);
\node at (-0.08,-0.08) {\footnotesize $0$};
\node at (0.3,-0.15) {$\frac{3}{20}$};
\node at (0.5,-0.15) {$\frac{5}{20}$};
\node at (0.7,-0.15) {$\frac{7}{20}$};
\node at (1,-0.15) {$\frac{10}{20}$};
\node at (1.3,-0.15) {$\frac{13}{20}$};
\node at (1.6,-0.15) {$\frac{16}{20}$};
\node at (1.8,-0.15) {$\frac{18}{20}$};
\node at (2,-0.15) {\footnotesize $1$};
\draw (0.3,-0.03) -- (0.3,0.03);
\draw (0.5,-0.03) -- (0.5,0.03);
\draw (0.7,-0.03) -- (0.7,0.03);
\draw (1,-0.03) -- (1,0.03);
\draw (1.3,-0.03) -- (1.3,0.03);
\draw (1.6,-0.03) -- (1.6,0.03);
\draw (1.8,-0.03) -- (1.8,0.03);
\draw (2,-0.03) -- (2,0.03);
\node at (-0.1,0.143) {\footnotesize $1/7$};
\node at (-0.1,0.286) {\footnotesize $2/7$};
\node at (-0.1,0.429) {\footnotesize $3/7$};
\node at (-0.1,0.571) {\footnotesize $4/7$};
\node at (-0.1,0.714) {\footnotesize $5/7$};
\node at (-0.1,0.857) {\footnotesize $6/7$};
\node at (-0.1,1) {\footnotesize $1$};
\draw (-0.03,0.143) -- (0.03,0.143);
\draw (-0.03,0.286) -- (0.03,0.286);
\draw (-0.03,0.429) -- (0.03,0.429);
\draw (-0.03,0.571) -- (0.03,0.571);
\draw (-0.03,0.714) -- (0.03,0.714);
\draw (-0.03,0.857) -- (0.03,0.857);
\draw (-0.03,1) -- (0.03,1);
\draw[very thick, -<] (0,0) -- (0.32,0);
\draw[very thick, -<] (0.3,0.143) -- (0.52,0.143);
\draw[very thick, -<] (0.5,0.286) -- (0.72,0.286);
\draw[very thick, -<] (0.7,0.429) -- (1.02,0.429);
\draw[very thick, -<] (1,0.571) -- (1.32,0.571);
\draw[very thick, -<] (1.3,0.714) -- (1.62,0.714);
\draw[very thick, -<] (1.6,0.857) -- (1.82,0.857);
\draw[very thick] (1.8,1) -- (2,1);
\node at (0,0) {\tiny $\bullet$};
\node at (0.3,0.143) {\tiny $\bullet$};
\node at (0.5,0.286) {\tiny $\bullet$};
\node at (0.7,0.429) {\tiny $\bullet$};
\node at (1,0.571) {\tiny $\bullet$};
\node at (1.3,0.714) {\tiny $\bullet$};
\node at (1.6,0.857) {\tiny $\bullet$};
\node at (1.8,1) {\tiny $\bullet$};
\node at (2,1) {\tiny $\bullet$};
\draw[dotted] (0.3,0.143) -- (0.3,0);
\draw[dotted] (0.5,0.286) -- (0.5,0);
\draw[dotted] (0.7,0.429) -- (0.7,0);
\draw[dotted] (1,0.571) -- (1,0);
\draw[dotted] (1.3,0.714) -- (1.3,0);
\draw[dotted] (1.6,0.857) -- (1.6,0);
\draw[dotted] (1.8,1) -- (1.8,0);
\draw[dotted] (2,1) -- (2,0);
\draw[dotted] (0.3,0.143) -- (0,0.143);
\draw[dotted] (0.5,0.286) -- (0,0.286);
\draw[dotted] (0.7,0.429) -- (0,0.429);
\draw[dotted] (1,0.571) -- (0,0.571);
\draw[dotted] (1.3,0.714) -- (0,0.714);
\draw[dotted] (1.6,0.857) -- (0,0.857);
\draw[dotted] (1.8,1) -- (0,1);
\end{tikzpicture}
\end{center}
\caption{The leaf distribution function $F_t$ for the tree $t$ shown in \cref{fig:Contour}. 
Note that $F_t$ is piecewise constant and all gaps are of height $\frac{1}{|t|}$ (here $\frac{1}{7}$), 
but pieces may have different widths (here, $\frac{2}{20}$ or $\frac{3}{20}$).
Informally, \cref{PropLeavesUniform} states that these widths
are asymptotically close to each other.}
\label{fig:leaf_distrib_function}
\end{figure}
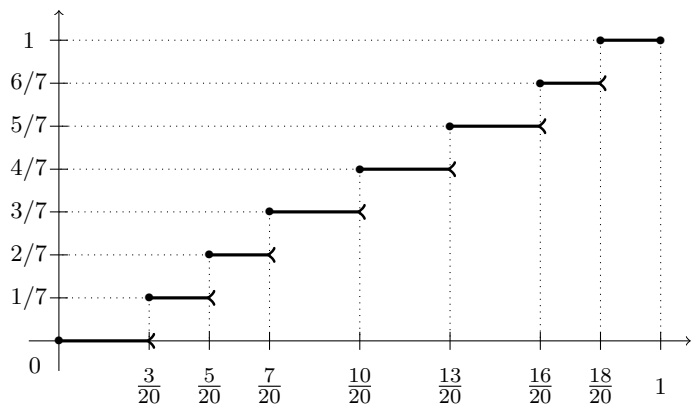

The following statement was essentially proved by Pitman and Rizzolo 
in \cite{PitmanRizzolo}.
\begin{proposition}
  Let $T_n$ be a uniform random Schr\"oder tree with $n$ leaves.
    Defining the random distribution function $F_{T_n}$ as above,
    we have the following convergence in probability:
    \[ ||F_{T_n} - F_U||_\infty \to 0. \vspace{-6mm}\]
    \label{PropLeavesUniform}
\end{proposition}
Informally, choosing uniformly at random a leaf in $T_n$ 
(or rather the corresponding $x$-coordinate in the normalized contour $\widetilde{C_{T_n}}$) 
amounts in the limit to choosing a uniform point in $[0,1]$. 
\begin{proof}
  Recall from \cref{prop:Schroeder_are_Galton-Watson},
  that $T_n$ can be seen as a Galton-Watson tree with a specific offspring
  distribution, conditioned to have $n$ leaves.

  In the proof of \cite[Theorem 8]{PitmanRizzolo}, 
  Pitman and Rizzolo established for such models
  the convergence of the empirical distribution of
  the location of leaves (that they denote $\nu^n$) to the Lebesgue measure on $[0,1]$.
  This convergence holds in probability for the weak topology.

  Weak convergence is equivalent to the convergence of distribution functions
  at continuity points of the limit,
  so that their statement correspond to the convergence of $F_{T_n}$ to $F_U$
  in probability for the pointwise convergence topology.
  Moreover, since distribution functions are non-decreasing and 
  since the limit $F_U$ is continuous, it is well-known and easy to see that 
  pointwise convergence implies uniform convergence, so that
  the proposition is proved.
\end{proof}

We finish by a simple lemma, that is used in \cref{sec:signs,sec:continuity}.
\begin{lemma}
    \label{lem:FF*Pareil}
    If $F$ is a distribution function,
    then $||F^*-F^*_U||_\infty \le ||F - F_U||_\infty$.
\end{lemma}
Note that on $[0,1]$, $F^*_U=F_U = \id_{[0,1]}$ .
\begin{proof}
Choose $u\in[0,1]$. If $u\geq F^*(u)$ then
\begin{align*}
|F^*(u)-F_U(u)|=u-F^*(u)&\leq F(F^*(u))-F^*(u)\\
&\leq  \sup_{t\in [0,1]} |F(t)-t|=||F - F_U||_\infty.
\end{align*}
If on the contrary $u< F^*(u)$ then write for a small $\eps>0$
\begin{align*}
|F^*(u)-F_U(u)|=F^*(u)-u &\leq 
 F^*(u)-\eps-F(F^*(u)-\eps)+\eps\\
&\leq \sup_{t \in [-\eps,1-\eps]} |t-F(t)|+\eps\\
&\leq  \eps + \sup_{t \in [0,1-\eps]} |t-F(t)|+\eps \leq ||F - F_U||_\infty+2\eps,
\end{align*}
and let $\eps\to 0$.
\end{proof}

\subsection{Reformulation of the main theorem}

To prove the convergence of $\occ(\pi,\Si_n)$ towards $\Lambda_\pi$, 
it is useful to describe these two random variables in a similar manner. 
More precisely, we define a function $\ProbPerm$ in such a way that both $\occ(\pi,\Si_n)$ and $\Lambda_\pi$ 
have a natural expression using $\ProbPerm$. 
Along the way, we also define a tree analogue $\ProbTree$ of $\ProbPerm$, 
which we shall use in the proof of Theorem~\ref{thm:main}.\ref{item:main_joint}. 

\begin{definition}[$\ProbTree$]
Let $\patterntree$ be a tree with $k$ leaves, $f$ be an excursion and $F$ be a distribution function with $F(0)=0$ and $F(1)=1$.
Let $X_1,\dots,X_k$ be independent random variables with distribution function $F$ and $\XX$ be the set $\{X_1,\dots,X_k\}$.
Then $\ProbTree(\patterntree;f,F)$ is defined as the probability that 
$\Tree(f,\XX)=\patterntree$.
\label{dfn:ProbTree}
\end{definition}

\begin{definition}[$\ProbPerm$]
Let $\pi$ be a permutation of size $k$, $(f,s)$ be a signed excursion and $F$ be a distribution function with $F(0)=0$ and $F(1)=1$.
Let $X_1,\dots,X_k$ be independent random variables with distribution function $F$ and $\XX$ be the set $\{X_1,\dots,X_k\}$.
Then $\ProbPerm(\pi;f,s,F)$ is defined as the probability
that $\Perm(f,s,\XX)=\pi$.
\label{dfn:ProbPerm}
\end{definition}

Note that in the above two definitions, if the event $\Tree(f,\XX)=\patterntree$ (resp. $\Perm(f,s,\XX)=\pi$)
is realized, then there is no repetition in $X_1,\dots,X_k$.

\begin{observation}\label{obs:lamba_from_PrPerm}
By definition, it follows that $\Lambda_\pi = \ProbPerm(\pi;\Exc,S,F_U)$. 
\end{observation}

\begin{lemma}
Let $(t,\eps)$ be a signed tree and $\sigma=\perm(t,\eps)$ be the corresponding permutation.
Let $(\widetilde{C_t},s)$ be the signed contour of $(t,\eps)$
and  $F_t$ be the distribution function associated with $t$.
It holds that:
\[ \ProbPerm(\pi;\widetilde{C_t},s,F_t) = \frac{(|t|)_{k}}{|t|^k}\, \occ(\pi, \sigma),\]
where $(x)_k$ denotes the falling factorial $x(x-1) \cdots (x-k+1)$.
\label{Lem:OurProcess=PropPatterns}
\end{lemma}

\begin{proof}
Let $X_1,\dots,X_k$ be independent random variables with distribution function $F_t$ 
(in this case, $\Perm(\widetilde{C_t},s,\XX)$ is always defined -- see \cref{obs:Perm_is_defined}).
In other terms, we take $k$ leaves of $t$ independently uniformly at random.
The probability that they are different is clearly $(|t|)_{k}/|t|^k$.
Conditioned to that event, $\XX=\{X_1,\dots,X_k\}$ is a uniform subset of $k$ leaves of $t$.
From \cref{obs:functionPerm=patterns}, we have
\[\Perm(\widetilde{C_t},s,\XX)=\pat_I(\sigma), \]
where $I$ is a uniform random $k$-element subset of $[|t|]$. 
The probability that the right hand-side is equal to $\pi$ is by definition $\occ(\pi, \sigma)$.
Thus 
\[\proba\Big( \Perm(\widetilde{C_t},s,\XX)=\pi \big| \XX\text{ is repetition-free} \Big) = \occ(\pi, \sigma),\]
which ends the proof.
\end{proof}

\begin{corollary}\label{cor:occ_from_PrPerm}
Recall that $T_n$ denotes a uniform random Schröder tree with $n$ leaves, and 
$\Eps_n$ the sign function on the internal vertices of $T_n$,
such that the signs alternate and the root of $T_n$ has a balanced sign. 
Denote by $(\widetilde{C_{T_n}},S_n)$ the signed contour of $(T_n,\Eps_n)$. 
Then 
\[ \occ(\pi, \Si_n) = \ProbPerm(\pi;\widetilde{C_{T_n}},S_n,F_{T_n})\left(1+ \mathcal{O}\left(\frac{1}{n}\right)\right).\]
\end{corollary}

Using \cref{obs:lamba_from_PrPerm} and \cref{cor:occ_from_PrPerm},
our main result --- {\em i.e.} the joint convergence in distribution of $\occ(\pi, \Si_n)$ to $\Lambda_\pi$, see Theorem~\ref{thm:main}.\ref{item:main_joint} ---
can now be written as the joint convergence 
(for any finite family of patterns $\pi$):
\begin{equation}
  \ProbPerm(\pi;\widetilde{C_{T_n}},S_n,F_{T_n}) \stackrel{d}{\to} 
\ProbPerm(\pi;\Exc,S,F_U),
\label{eq:MainResult_Reformulation}
\end{equation}

\subsection{Outline of the proof}\label{sec:OutlineProof}

The proof of \eqref{eq:MainResult_Reformulation} consists in several steps,
as follows.
Recall that the normalized contour $\widetilde{C_{T_n}}$ and the leaf
distribution function $F_{T_n}$ of $T_n$ converge to the Brownian excursion $\Exc$
and the uniform distribution function $F_U$, respectively.

A natural way to proceed would be to prove the convergence of $(\widetilde{C_{T_n}},S_n)$ to the signed Brownian excursion, 
and the continuity of $\ProbPerm(\pi;\cdot)$. 
As discussed in \cref{rem:cv_to_signedBrownianExc}, a major difficulty when trying to prove that $(\widetilde{C_{T_n}},S_n)$ converges to $(\Exc,S)$ 
is that the signs on the local minima of $\widetilde{C_{T_n}}$ are far from independent. 
Instead of attacking this difficulty head-on, we have developed a proof along a different line. 
Recall that our goal is to prove the convergence in distribution of $\occ(\pi, \Si_n)$ to $\Lambda_\pi$.

The main part of our proof is to show the convergence of $\occ(\pi, \Si_n)$ to $\Lambda_\pi$ \emph{in expectation}.
In Section~\ref{sec:moments}, we prove that this is enough to have convergence in distribution. 
Indeed we shall see that all moments of $\Lambda_\pi$ are determined by expectations of $\Lambda_\rho$,
for larger permutations $\rho$. 
Moreover we shall see that a similar statement holds in the limit for $\occ(\pi, \Si_n)$. 
Therefore, convergence of expectations implies convergence of all (joint) moments. 
And since the random variables are bounded by $1$, 
convergence of (joint) moments implies (multi-dimensional) convergence in distribution.

The goal of Section~\ref{sec:continuity} is to prove the convergence in distribution 
of the \emph{unsigned} trees extracted from the normalized contour of a uniform Schröder tree 
to the \emph{unsigned} trees extracted from the Brownian excursion. 
This is achieved by proving the continuity of $\ProbTree(\patterntree;\cdot)$, 
and using the convergence of $\widetilde{C_{T_n}}$ to $\Exc$ and the convergence of $F_{T_n}$ to $F_U$.
Since this step does not involve signs, there is no major difficulty here.

In Section~\ref{sec:signs}, we return to signed objects, proving that the signs on extracted trees are balanced and independent. 
More precisely, we examine the signs of the internal vertices of a signed tree $(\patterntree,\eps)$ extracted from $(\widetilde{C_{T_n}},S_n)$, 
and we prove that, when $n$ goes to infinity, 
these signs are balanced and independent. 
Note that a similar statement for $(\Exc,S)$ is obvious, because the signs are chosen balanced and independent in $(\Exc,S)$.
The proof for $(\widetilde{C_{T_n}},S_n)$ involves the fact that 
the relative distances between the points of $\widetilde{C_{T_n}}$ corresponding to the vertices of $\patterntree$ tend to infinity, 
and a subtree exchangeability argument.

Finally, we put all these results together in Section~\ref{sec:main_proof}, 
to conclude the proof of Theorem~\ref{thm:main}.\ref{item:main_joint}.
More precisely, using Sections~\ref{sec:continuity} and \ref{sec:signs}
we show the convergence of $\occ(\pi, \Si_n)$ to $\Lambda_\pi$ in expectation,
which implies Theorem~\ref{thm:main}.\ref{item:main_joint} using Section~\ref{sec:moments}.

\section{Expectations determine joint moments}
\label{sec:moments}

Recall that, for each $n$, $\Si_n$ denotes a uniform random separable permutation of size $n$, 
and that we want to prove the convergence in distribution of the random variables $\occ(\pi,\Si_n)$ to $\Lambda_\pi$. 
Since 
the random variables $\occ(\pi,\Si_n)$ and the candidate limit $\Lambda_\pi$ are bounded by $1$, 
(multi-dimensional) convergence in distribution is equivalent to convergence of (joint) moments.
(This can be seen, {\em e.g.}, as a consequence of Stone-Weierstrass theorem,
which ensures that polynomials are dense in the set of continuous functions from $[0,1]^r$ to $\R$).
%A proof of this fact in the one-dimensional case can be found, e.g., in \cite[Chapter 30]{BillingsleyProbMeasure}; it is easily adapted to the multi-dimensional case.)
%Le referee a raison pour des variables bornées, c'est facile\ldots
Concretely, Theorem~\ref{thm:main}.\ref{item:main_joint} (p.\pageref{thm:main})
is equivalent to the following statement:

\begin{theorem}
\label{thm:ConvJointMoments}
For any list of patterns $\pi_1,\dots,\pi_r$ (possibly with repetitions),
\begin{equation*}
  \esper\left[ \prod_{i=1}^r \occ(\pi_i,\Si_n) \right] 
  \longrightarrow \esper\left[ \prod_{i=1}^r \Lambda_{\pi_i} \right]. 
\end{equation*}
\end{theorem}

Instead of proving a convergence in distribution of random variables,
we can therefore limit ourselves to proving a convergence of real numbers, which is a lot more tractable. 
To make our task even simpler, we show in \cref{cor:expectation_is_enough} below
that Theorem~\ref{thm:ConvJointMoments} follows if we prove, 
for all $\pi$, the convergence of $\esper[\occ(\pi,\Si_n)]$ to $\esper[\Lambda_{\pi}]$. 
The key point is that $\prod_{i=1}^r \occ(\pi_i,\si)$ can be expressed combinatorially as a linear combination of $\occ(\rho,\si)$ for larger patterns $\rho$, 
and that the same relation holds between $\prod_{i=1}^r \Lambda_{\pi_i}$ and the $\Lambda_\rho$'s. 

\begin{definition}
\label{dfn:ordered_set-partition}
Let $\rho$ be a permutation of size $K$ and $\pi_1,\cdots,\pi_r$ be permutations such that $\sum_{i=1}^r |\pi_i| = K$.
An ordered set-partition $(I_i)_{1 \leq i \leq r}$ of $\{1, \dots K\}$ is \emph{compatible} with $\rho, \pi_1,\cdots,\pi_r$
if for all $i$, $\pat_{I_i}(\rho) = \pi_i$.

Let $d_{\pi_1,\ldots,\pi_r}^\rho$ be the number of ordered set-partitions compatible with $\rho, \pi_1,\cdots,\pi_r$.
We denote by
$$
c_{\pi_1,\ldots,\pi_r}^\rho = \frac{d_{\pi_1,\ldots,\pi_r}^\rho}{\binom{K}{|\pi_1|,\ldots,|\pi_r|}}
$$
the proportion of ordered set-partitions of $\{1, \dots K\}$ which are compatible with $\rho, \pi_1,\cdots,\pi_r$.
\end{definition}
\begin{example}
Let $\rho = 1342$, $\pi_1 = 21$ and $\pi_2 = 12$.
There are $6$ ordered set-partitions of $\{1, \dots 4\}$ which are
{\small $(\{1,2\},\{3,4\})$, $(\{1,3\},\{2,4\})$, $(\{1,4\},\{2,3\})$, $(\{2,3\},\{1,4\})$, $(\{2,4\},\{1,3\})$, $(\{3,4\},\{1,2\})$}.
Let ($I_1$,$I_2$) be one of the first four ordered pairs,
then $\pat_{I_1}(\rho) = 12 \neq \pi_1$
thus ($I_1$,$I_2$) is not compatible with $\rho, \pi_1,\pi_2$.
On the contrary, the last two pairs are compatible with $\rho, \pi_1,\pi_2$.
Then $d_{21,12}^{1342} = 2$ and  $c_{21,12}^{1342} = 1/3$.
\end{example}

\begin{proposition}
Fix a list of patterns $\pi_1,\cdots,\pi_r$, 
and denote by $K$ the sum of their sizes.
Then
\[\prod_{i=1}^r \Lambda_{\pi_i} =  \sum_{\rho \in \Sn_K} c_{\pi_1,\ldots,\pi_r}^\rho \Lambda_\rho. \]
\label{prop:moments_in_terms_of_expectation_limit_case}
\end{proposition}

\begin{proof}
We fix a realization $(\Exc,S)$ of the Brownian excursion all along the proof.
The main idea to prove the equality in \cref{prop:moments_in_terms_of_expectation_limit_case} is to show that both parts represent the probability of the same event.
First we describe the event.

We fix a list of patterns $\pi_1,\cdots,\pi_r$ 
and we denote by $K$ the sum of their sizes.
Let us take $K$ independent uniform random variables in $[0,1]$.
We denote the $|\pi_1|$ first ones (in the order of sampling) by $X_1^1, \dots, X_{|\pi_1|}^1$,
the following ones  $X_1^2, \dots, X_{|\pi_2|}^2$,
and so on until the $|\pi_r|$ last ones which are denoted $X_1^r, \dots, X_{|\pi_r|}^r$. 
Let $\XX^{(i)} = \{X_j^i \mid 1\leq j \leq |\pi_i|\}$ for all $i$ and $\XX = \cup_{i=1}^r \XX^{(i)}$. 
The event we are considering is then: for each $i$, $\Perm(\Exc,S,\XX^{(i)}) = \pi_i$.
Since $(\Exc,S)$ is a fixed realization of the Brownian excursion,
all the probabilities below should be understood with respect to
the random variables $\XX$, as indicated by the notation $\proba^\XX$.

Let $\mathcal{P}^{\pi_1,\cdots,\pi_r} = \proba^\XX\big(\forall i, \Perm(\Exc,S,\XX^{(i)}) = \pi_i \big)$. 
Since the $X^i_j$ are independent, we have :
\[\mathcal{P}^{\pi_1,\cdots,\pi_r} = \prod_{i=1}^r \proba^\XX\big(\Perm(\Exc,S,\XX^{(i)}) = \pi_i \big) = \prod_{i=1}^r \Lambda_{\pi_i}.\]

It remains to prove that $\mathcal{P}^{\pi_1,\cdots,\pi_r} =  \sum_{\rho \in \Sn_K} c_{\pi_1,\ldots,\pi_r}^\rho \Lambda_\rho$.
We compute this probability by conditioning on the value of $\Perm(\Exc,S,\XX)$.
Since $(\Exc,S,\XX)$ satisfies Condition~\eqref{condition:unique_signs_on_minima} (p.\pageref{condition:unique_signs_on_minima})
with probability~$1$ (see Observation~\ref{obs:Brownian_satisfies_C} p.\pageref{obs:Brownian_satisfies_C}), 
the permutation $\Perm(\Exc,S,\XX)$ is almost surely defined and has size~$K$.
Thus
\begin{align*}
\mathcal{P}^{\pi_1,\cdots,\pi_r} &= 
\sum_{\rho \in \Sn_K}  \mathcal{P}_\rho \times \proba^\XX\big( \Perm(\Exc,S,\XX)=\rho \big)  \\
\text{ where }\mathcal{P}_\rho &= \proba^\XX\big(\forall i, \Perm(\Exc,S,\XX^{(i)}) = \pi_i \mid \Perm(\Exc,S,\XX)=\rho \big).
\end{align*}

By definition, $\proba^\XX\big( \Perm(\Exc,S,\XX)=\rho \big)=\Lambda_\rho$,
so we just need to prove that $\mathcal{P}_\rho = c_{\pi_1,\ldots,\pi_r}^\rho$
to finish the proof.

Consider a realization of the random variables $\XX$.
We say that $X_j^i$ has rank $k$ if it is the $k$-th smallest value 
among all variables of $\XX$.
Then we can associate with $\XX$ an ordered set-partition of $\{1, \dots, K\}$ that we denote abusively $\Part(\XX)$ and that is defined as follows: 
the $i$-th subset of $\Part(\XX)$ is obtained from $\XX^{(i)}$
by replacing each $X_j^i$ by its rank. For example, 
\[\Part(\{\Blue{0.7}, \Blue{0.9}, \Blue{0.2}\}, \{\Green{0.5}, \Green{0.8}\}, \{\Red{0.3}\}) =
(\{\Blue1,\Blue4,\Blue6\}, \{\Green3,\Green5\}, \{\Red2\})\]
since $\Blue{0.2}<\Red{0.3}<\Green{0.5}<\Blue{0.7}<\Green{0.8}<\Blue{0.9}$.

Let $\mathcal{I}$ be the set of all ordered set-partitions of $\{1, \dots, K\}$ in $r$ subsets such that the $i$-th one has $|\pi_i|$ elements.
Then, by conditioning on the value of $\Part(\XX)$, we have
\begin{multline*}
    \mathcal{P}_\rho =
\sum_{(I_i) \in \mathcal{I}} \proba^\XX\big(\Part(\XX) = (I_i) \big) \times \\
\proba^\XX\big(\forall i, \Perm(\Exc,S,\XX^{(i)}) = \pi_i \mid \Perm(\Exc,S,\XX)=\rho \text{ and } \Part(\XX) = (I_i) \big).
\end{multline*}
Recall that, by construction, $\Perm=\perm \circ \Tree_{\pm}$.
From Observations~\ref{obs:Tree_pm_generalCase} (p.\pageref{obs:Tree_pm_generalCase}) and~\ref{obs:patterns_and_subtrees} (p.\pageref{obs:patterns_and_subtrees}), 
\begin{multline*}
    \proba^\XX\big(\forall i, \Perm(\Exc,S,\XX^{(i)}) = \pi_i \mid \Perm(\Exc,S,\XX) =\rho \text{ and } \Part(\XX) = (I_i) \big) 
    \\=\begin{cases}
        1 &\text{ if $(I_i)$ is compatible with $\rho,\pi_1,\ldots,\pi_r$;}\\
        0 & \text{otherwise.}
    \end{cases}
\end{multline*}

In addition, we claim that $\Part(\XX)$ is uniformly distributed in $\mathcal{I}$,
that is, for each $(I_i)$ in $\mathcal{I}$, we have
\hbox{$\proba^\XX\big(\Part(\XX) = (I_i) \big) = 1/\binom{K}{|\pi_1|,\ldots,|\pi_r|}$}.
Indeed each possible relative order of the $(X_j^i)$ occurs with the same probability.
Moreover, the block sizes of ordered set-partitions in $\mathcal{I}$
are prescribed: $\mathrm{card}(I_i)=|\pi_i|$.
Therefore each ordered set-partition in $\mathcal{I}$ corresponds 
to the same number of relative orders of the variables $(X_j^i)$; 
and this proves our claim.

Finally, we obtain that $\mathcal{P}_\rho = d_{\pi_1,\ldots,\pi_r}^\rho / \binom{K}{|\pi_1|,\ldots,|\pi_r|} = c_{\pi_1,\ldots,\pi_r}^\rho$.
Thus $\mathcal{P}^{\pi_1,\cdots,\pi_r} =  \sum_{\rho \in \Sn_K} c_{\pi_1,\ldots,\pi_r}^\rho \Lambda_\rho$, 
which concludes the proof.
\end{proof}

\begin{proposition}
Fix a list of patterns $\pi_1,\cdots,\pi_r$, 
and denote by $K$ the sum of their sizes.
Then for any permutation $\si$ of size $n$, 
\[ \prod_{i=1}^r \occ(\pi_i, \si) = \sum_{\rho \in \Sn_K} c_{\pi_1,\ldots,\pi_r}^\rho \occ(\rho, \si)\ +\ \mathcal{O}\left(\frac1n\right).\]
The constant hidden in the $\mathcal{O}$ symbol does depend on $\pi_1,\ldots,\pi_r$,
but is uniform on $\sigma$.
\label{prop:moments_in_terms_of_expectation_discrete_case}
\end{proposition}

\begin{proof}
Let $k_i$ be the size of $\pi_i$.
We denote by $S_i$ the set of $k_i$-element subsets, $I$, of $[n]$ such that $\pat_I(\sigma) = \pi_i$. 
Then
\[ \prod_{i=1}^r \occ(\pi_i, \si) = \prod_{i=1}^r \frac{\mathrm{card}(S_i)}{\binom{n}{k_i}} = \frac{\prod_{i=1}^r \mathrm{card}(S_i)}{\prod_{i=1}^r \binom{n}{k_i}}. \]
We set $S_\rho = \{ (s_1, \dots, s_r) \in S_1 \times S_2 \times \dots \times S_r \text{ such that } \pat_I(\sigma)=\rho \text{ for } I = \cup_{i=1}^r s_i\}$.
Then
\[ \prod_{i=1}^r \mathrm{card}(S_i) = \mathrm{card}(S_1 \times S_2 \times \dots \times S_r)
= \mathrm{card}( \cup_{\rho \in \Sn_K} S_\rho) + \mathrm{card}(S_1 \times S_2 \times \dots \times S_r \setminus \cup_{\rho \in \Sn_K} S_\rho ). \]
Note that the sets $S_\rho$ are disjoints: indeed $S_\rho \cap S_{\rho'} \neq \emptyset \Rightarrow \rho=\pat_I(\sigma)=\rho'$.
Therefore, denoting $R = S_1 \times S_2 \times \dots \times S_r \setminus \cup_{\rho \in \Sn_K} S_\rho$, we have:
\[ \prod_{i=1}^r \mathrm{card}(S_i) \ = \sum_{\rho \in \Sn_K} \mathrm{card}(S_\rho) \quad + \quad \mathrm{card}(R). \]

We first study $R$.
Let $(s_1, \dots, s_r) \in S_1 \times S_2 \times \dots \times S_r$, and define $I = \cup_{i=1}^r s_i$. 
Note that $\mathrm{card}(I) \leq K$, and that the inequality is strict if and only if $(s_1, \dots, s_r) \in R$. 
Consequently, it holds that $R = \{ (s_1, \dots, s_r) \in S_1 \times S_2 \times \dots \times S_r \text{ such that }\mathrm{card}(\cup_{i=1}^r s_i) < K \}$. 
It follows that $\mathrm{card}(R) = \mathcal{O}(n^{K-1})$, 
where the constant hidden in the $\mathcal{O}$ symbol depends on the $k_i$'s but not on~$n$. 
Since $\prod_{i=1}^r \binom{n}{k_i} \sim \text{cst} \ n^K$ for some constant \text{cst} that depends on the $k_i$'s only, 
we obtain that $\frac{\mathrm{card}(R)}{\prod_{i=1}^r \binom{n}{k_i}} = \mathcal{O}\left(\frac1n\right)$, 
where the constant hidden in the $\mathcal{O}$ symbol depends on the $k_i$'s only.

We now study $S_\rho$.
Each element of $S_\rho$ corresponds to an occurrence of $\rho$ in $\sigma$.
Conversely, a given occurrence of $\rho$ in $\sigma$ may lead to one, several or no element(s) of $S_\rho$:
this depends on the number of ways to partition the set of indices of $\sigma$ corresponding to the occurrence of $\rho$ 
into an ordered sequence $(s_1, \dots, s_r)$ of $r$ sets such that each $s_i$ induces the pattern $\pi_i$. 
The number of ways to do this ordered partition does not depend on the occurrence of $\rho$ that is considered, 
and there are by definition $d_{\pi_1,\ldots,\pi_r}^\rho$ ways to do this partition.
Thus 
\begin{align*}
\mathrm{card}(S_\rho) &= d_{\pi_1,\ldots,\pi_r}^\rho \times \text{ number of occurrences of } \rho \text{ in } \sigma \\
&= \binom{K}{k_1, \dots, k_r} c_{\pi_1,\ldots,\pi_r}^\rho \times  \binom{n}{K} \occ(\rho, \sigma).
\end{align*}

Putting things together,
\begin{align*}
\prod_{i=1}^r \occ(\pi_i, \si) &= 
\sum_{\rho \in \Sn_K} \frac{\mathrm{card}(S_\rho)}{\prod_{i=1}^r \binom{n}{k_i}} \quad + \quad \frac{\mathrm{card}(R)}{\prod_{i=1}^r \binom{n}{k_i}}\\
&= \sum_{\rho \in \Sn_K} c_{\pi_1,\ldots,\pi_r}^\rho \occ(\rho, \sigma)
\frac{\binom{K}{k_1, \dots, k_r}\binom{n}{K}}{\prod_{i=1}^r \binom{n}{k_i}}\ +\ \mathcal{O}\left(\frac1n\right). 
\end{align*}
By simple computations, we have $\frac{\binom{K}{k_1, \dots, k_r}\binom{n}{K}}{\prod_{i=1}^r \binom{n}{k_i}}
=\frac{n!}{K!(n-K)!}\ \frac{K!}{k_1!\ \dots\ k_r!}\ \frac{k_1! (n-k_1)!}{n!} ... \frac{k_r! (n-k_r)!}{n!} = 1 + \mathcal{O}\left(\frac1n\right)$, 
where again the constant hidden in the $\mathcal{O}$ symbol depends on the $k_i$'s only. 
To conclude, we need to remark that the sum $\sum_{\rho \in \Sn_K} c_{\pi_1,\ldots,\pi_r}^\rho \occ(\rho, \sigma)$ is bounded independently of $n$; 
indeed, each term $c_{\pi_1,\ldots,\pi_r}^\rho \occ(\rho, \sigma)$ is bounded by $1$, and there are $K!$ terms. 
It then follows that 
\[
\prod_{i=1}^r \occ(\pi_i, \si) = \sum_{\rho \in \Sn_K} c_{\pi_1,\ldots,\pi_r}^\rho \occ(\rho, \sigma)\ +\ \mathcal{O}\left(\frac1n\right)
\]
where the constant hidden in the $\mathcal{O}$ symbol does depend on $\pi_1,\ldots,\pi_r$,
but is uniform on $\sigma$.
\end{proof}

\begin{corollary}\label{cor:expectation_is_enough}
Theorem~\ref{thm:ConvJointMoments} (and therefore Theorem~\ref{thm:main}.\ref{item:main_joint}) is equivalent to the following statement: 
\begin{equation}
  \text{For any pattern }\pi, \ \esper\left[ \occ(\pi,\Si_n) \right] \longrightarrow \esper\left[ \Lambda_{\pi} \right]. 
  \label{eq:cv_exp}
\end{equation}
\label{cor:Exp_is_enough}
\end{corollary}

\begin{proof}
  We assume that Eq.~\eqref{eq:cv_exp} holds. 
Let $\pi_1,\ldots,\pi_r$ be a list of patterns, and denote by $K$ the sum of their sizes.
Then, from \cref{prop:moments_in_terms_of_expectation_limit_case,prop:moments_in_terms_of_expectation_discrete_case},
we have:
\begin{multline*}
  \esper\left[ \prod_{i=1}^r \occ(\pi_i,\Si_n) \right]
  = \sum_{\rho \in \Sn_K} c_{\pi_1,\ldots,\pi_r}^\rho \esper \left[ \occ(\rho, \Si_n)\right] +o(1)
  = \sum_{\rho \in \Sn_K} c_{\pi_1,\ldots,\pi_r}^\rho \left(\esper \left[ \Lambda_\rho \right] +o(1)\right) +o(1)\\
  = \sum_{\rho \in \Sn_K} c_{\pi_1,\ldots,\pi_r}^\rho \esper \left[ \Lambda_\rho \right] +o(1)
  = \esper\left[ \prod_{i=1}^r \Lambda_{\pi_i} \right] +o(1).
\end{multline*}
Note that to get the second line, we used that
$\sum_{\rho \in \Sn_K} c_{\pi_1,\ldots,\pi_r}^\rho$ does not depend on $n$. 
The above computations show that Eq.~\eqref{eq:cv_exp} implies Theorem~\ref{thm:ConvJointMoments}.
\end{proof}

The next three sections focus on proving Eq.~\eqref{eq:cv_exp}; 
the proof will be completed in~\cref{thm:ConvExp}.

%%%%%%%%%%%%%%%%%%%%%%%%%%%%%%%%%%%%%%%%%%%%%%%%%%%%%%%%%%%%%%%%%%%%%%%%%%%%%%%%%%%%%%%%%%%%%%%%%%%%%%%%%%%%%%
\section{Continuity of $\ProbTree$}
\label{sec:continuity}

In this section, we prove the convergence in distribution 
of the trees extracted from the normalized contour of a uniform Schröder tree 
to the trees extracted from the Brownian excursion (in the unsigned case). 
This result is stated in \cref{cor:PrTree_cv_in_distribution}.
It follows easily from earlier results and a continuity property of $\ProbTree$ (defined in \cref{dfn:ProbTree} p.\pageref{dfn:ProbTree})
that we establish in \cref{lem:continuityProbTree} below. 
Let us first set up some notation.

Fix a tree $\patterntree$ with $k$ leaves.
We consider $\ProbTree(\patterntree;\dots)$ that is the map
$(f,F) \mapsto \ProbTree(\patterntree;f,F)$.
We use the uniform topology (that is, the topology induced by the supremum norm)
both on the set of excursions and on the set of distribution functions.

Throughout the section, if $f$ is an excursion and $x \le y$ in $[0,1]$,
we denote $m(f;x,y)=\min_{[x,y]} f$. 
Let $X_1 \le X_2 \cdots \le X_k$ be the reordering of $k$ uniform i.i.d. random variables in $[0,1]$.
We say that an excursion $f$ has the \emph{distinct minima property} 
if $m(f,X_1,X_2)$, \ldots, $m(f,X_{k-1},X_k)$ are distinct with probability $1$ (the probability here is taken with respect to $X_1,\cdots,X_k$).

\begin{lemma}
  Let $\Exc$ be the Brownian excursion, 
  then a.s. $\Exc$ has the distinct minima property.
  \label{Lem:Exc_DMP}
\end{lemma}
\begin{proof}
  This can be seen as a consequence of \cref{lem:tout_a_une_densite},
  but since this lemma uses a deep result of \cite{LeGall},
  let us give a more elementary proof.

If $\Exc$ is the Brownian excursion, and $X_1 < \cdots <X_k$
is the reordering of $k$ i.i.d. uniform random variables in $[0,1]$,
then from \cref{Cor:OneSided} with probability $1$ no $X_i$ is a one-sided minimum.
When this is the case,
for each $i$, $m(f,X_i,X_{i+1})$ cannot be reached on the extremities
of the interval $[X_i,X_{i+1}]$ and is therefore a local minimum.
Consequently, using \cref{Lemma:MinimumSameLevel},
the values $m(f,X_1,X_2)$, \ldots, $m(f,X_{k-1},X_k)$ are distinct, almost surely.
\end{proof}

\begin{lemma}
Let $\patterntree$ be a fixed tree with $k$ leaves.
If $f$ is an excursion with the distinct minima property
then
$\ProbTree(\patterntree;\dots)$ is continuous in $(f,F_U)$
with respect to the uniform topology.
\label{lem:continuityProbTree}
\end{lemma}

\begin{proof}
  We prove that $\ProbTree(\patterntree;g,G) \rightarrow \ProbTree(\patterntree;f,F_U)$
  as soon as $g \rightarrow f$ and $G \rightarrow F_U$.
  Fix $\eps,\delta>0$ and consider an excursion $g$ and a distribution function $G$ with
  $||f-g||_\infty \leq \eps$ and $ ||F_U-G||_\infty\leq \delta$.

  Recall from \cref{lem:FF*Pareil} (p.\pageref{lem:FF*Pareil})
  that $||F_U^\star-G^\star||_\infty \leq ||F_U-G||_\infty\leq \delta$
  and that $F_U^\star$ is the identity on $[0,1]$.
 Let $X_1<\dots <X_k$ be the reordering of $k$ independent uniform random variables, 
 and let $\xx=(X_1,\dots,X_k)$. 
 Furthermore, set $Y_i=G^\star(X_i)$. 
 Note that $G^\star$ is non-decreasing, so that $Y_1 \leq \dots\leq Y_k$. We write $\yy=(Y_1, \dots, Y_k)$. 
By construction we have that:
\begin{itemize}
  \item $\yy$ has the distribution of the reordering of $k$ i.i.d. random variables of distribution $G$;
\item and, for each $i$, one has $|Y_i - X_i|\le \delta$.
\end{itemize}
It follows from $||f-g||_\infty \leq \eps$ that for all $i \le k-1$, 
$|m(g,Y_i,Y_{i+1}) - m(f,Y_i,Y_{i+1})| \le \eps$.
This implies that, for each $i \le k-1$,
\begin{equation}
 |m(g,Y_i,Y_{i+1}) - m(f,X_i,X_{i+1})| \le \eps + \omega(f,\delta),
  \label{eq:Conv_Val_And_Min}
\end{equation}
where $\omega$ is the modulus of continuity of $f$ defined by
\begin{equation}
\omega(f,\delta)= \sup_{|r-s|\leq \delta}\big|f(r)-f(s)\big|.
  \label{eq:def_modulus_continuity}
\end{equation}
Recall that $\Tree(f,\xx)$ is the tree extracted from the set of points $\xx$ on the excursion $f$. We have
\begin{align}
\left|\ProbTree(\patterntree;f,F_U)-\ProbTree(\patterntree;g,G)\right|&=\left|\mathbb{P}^\xx\left(\Tree(f,\xx)=\patterntree\right)-\mathbb{P}^{\xx}\left(\Tree(g,{\yy})=\patterntree\right)\right| \nonumber\\
&\leq \mathbb{P}^{\xx}(\Tree(f,\xx)\neq \Tree(g,{\yy})). 
\label{eq:Prob_Diff_Trees}
\end{align}
Above and in what follows, the randomness is only given by $\xx$, 
the variables $\yy$ are functions of the $\xx$'s and hence random as well,
while $f$ is non-random.
We emphasize this by using the notation $\mathbb{P}^\xx$.

By construction (see \cref{dfn:Tree} p.\pageref{dfn:Tree}), the tree $\Tree(f,\xx)$ only depends on the relative 
order of
  $m(f,X_1,X_2),\ \ldots,\ m(f,X_{k-1},X_k)$.
We denote $\gap(f,\xx)$ the minimal difference between any two of these values.
Since $f$ is assumed to have the distinct minima property,
$\gap(f,\xx)$ is non-zero with probability $1$.

From \cref{eq:Conv_Val_And_Min}, we see that the numbers
$m(g,Y_1,Y_2), \ldots, m(g,Y_{k-1},Y_k)$
are in the same relative order as $m(f,X_1,X_2)$, $\ldots$, $m(f,X_{k-1},X_k)$ 
as soon as $\eps + \omega(f,\delta) < \gap(f,\xx)/2$.
If this is the case, then $\Tree(f,\xx)= \Tree(g,{\yy})$.
Therefore,
\begin{align*}
  \mathbb{P}^{\xx}(\Tree(f,\xx)\neq \Tree(g,{\yy}))
& \le \mathbb{P}^{\xx}\big(\gap(f,\xx)/2 \leq \eps + \omega(f,\delta)\big)\\
& \!\! \!\stackrel{\eps,\delta\to 0}{\longrightarrow} \ \mathbb{P}^{\xx}\big( \gap(f,\xx) =0 \big)=0.
\end{align*}
With \cref{eq:Prob_Diff_Trees}, this completes the proof.
\end{proof}

Recall that the \emph{(substitution) decomposition tree} of $\sigma$ is 
the unique signed Schröder tree $(t,\eps)$ in which the signs alternate such that $\perm(t,\eps) = \sigma$. 
By \cref{cor:same_distibution_separables_trees} (p.\pageref{cor:same_distibution_separables_trees}),
in the case where $\Si_n$ is a uniform random separable permutation of size $n$, 
then its (unsigned) decomposition tree is a uniform Schr\"oder tree with $n$ leaves, denoted $T_n$ all along the article.

\begin{proposition}
    Let $\Si_n$ be a uniform random separable permutation of size $n$
    and $(\widetilde{C_{T_n}})$ be the normalized contour of its decomposition tree.
    We also define $F_{T_n}$ as usual (see \cref{eq:EqDFLeaves} p.\pageref{eq:EqDFLeaves}).
    Let $\Exc$ be the Brownian excursion and $\lambda=\sqrt{2+3/\sqrt{2}}$.
    Then,  we have
    \[ (\widetilde{C_{T_n}},F_{T_n}) \stackrel{d}{\to} (\lambda\cdot\Exc,F_U)\]
   in distribution, with respect to the uniform topology.
    \label{PropCvCF}
\end{proposition}
\begin{proof}
From Propositions \ref{prop:ConvergenceContour} (p.\pageref{prop:ConvergenceContour}) and \ref{PropLeavesUniform} (p.\pageref{PropLeavesUniform}),
we know that $\widetilde{C_{T_n}} \stackrel{d}{\to} \lambda\cdot\Exc$
and $F_{T_n} \stackrel{d}{\to} F_U$.
It remains to prove the joint convergence.
Note that the limit $F_U$ of $F_{T_n}$ is deterministic.

Thus we want to use a theorem of Billingsley \cite[Theorem 3.9]{Billingsley} that asserts that
if $X_n' \stackrel{d}{\to} X'$ and $X_n'' \stackrel{d}{\to} a''$
(where $X_n'$ and $X'$ are random variables with values in a metric space $S'$,
$X_n''$ are random variables with values in a metric space $S''$ and $a''$
is a deterministic element of $S''$),
then $(X_n',X_n'') \stackrel{d}{\to} (x',a'')$, 
{\em provided that $T = S' \times S''$ is separable}.

The hypothesis that $T$ is separable is in fact only needed
to ensure that the Borel $\sigma$-algebra on $T$
is the product of the Borel $\sigma$-algebras on $S'$ and $S''$.
For this to be the case, it is sufficient that
{\em either $S'$ or $S''$} is separable; see \cite[Lemma 6.4.2]{Bogachev}.

In our case,
the functions $\widetilde{C_{T_n}}$ and the limit $\lambda\cdot\Exc$ are random elements
in the space $S'=\CCC[0,1]$ of continuous functions on $[0,1]$,
which is known to be separable (see \cite[Example 1.3]{Billingsley}).
We can therefore use Billingsley's theorem and the proposition is proved.
\end{proof}
%On the other hand, the functions $F_{T_n}$ and their limit $F_U$ live
%in the space $S''$ of right-continuous functions on $[0,1]$ with left-sided limits
%that admits only a finite number of jumps at rational points.
%We clearly have
%\[S'' \simeq \bigcup_{k \ge 0} \bigcup_{a_1,\ldots,a_k \in \Q \atop 0 < a_1 < \ldots <a_k < 1}
%\CCC[0,a_1] \times \CCC[a_1,a_2] \times \cdots \times \CCC[a_k,1], \]
%and hence $S''$ is separable.
%\end{proof}
%\begin{remark}
%  Note the space of right-continuous functions on $[0,1]$ with left-sided limits
%  is not separable for the uniform topology,
%  which explains why we used a more restricted space in the proof above.
%  Another option would have been to use the Skorohod topology and the fact
%  that when the candidate limit is continuous, convergence in uniform and
%  Skorohod topologies are equivalent.
%  
%  For more details on uniform and Skorohod topologies
%  on the space of right-continuous functions with left-sided limits
%  (including explanations of the above claims),
%  we refer the reader to \cite[Section 15]{Billingsley}.
%\end{remark}

\begin{corollary}\label{cor:PrTree_cv_in_distribution}
    With the above notation, for any fixed $\patterntree$, we have the convergence in distribution 
    \[\ProbTree(\patterntree;\widetilde{C_{T_n}},F_{T_n})  \stackrel{d}{\to} \ProbTree(\patterntree;\Exc,F_U).\]
    In particular, since these random variables are bounded, we have
    \[ \esper^{\Si_n} \big[\ProbTree(\patterntree;\widetilde{C_{T_n}},F_{T_n})\big]
    \to 
    \esper^\Exc \big[\ProbTree(\patterntree;\Exc,F_U) \big].\]
\end{corollary}
\begin{proof}
  This corollary is a simple application of the mapping theorem
  \cite[Theorem 2.7]{Billingsley} because of the following facts:
\begin{itemize}
\item $(\widetilde{C_{T_n}},F_{T_n}) \stackrel{d}{\to} (\lambda\cdot\Exc,F_U)$, which is proved in \cref{PropCvCF};
\item $\ProbTree(\patterntree;\ldots)$ is continuous at $(\Exc,F_U)$ for almost  every $\Exc$, which follows from Lemmas~\ref{Lem:Exc_DMP} and~\ref{lem:continuityProbTree}.\qedhere
\end{itemize}

%%%
%  Recall from \cref{PropCvCF}, that
%  \begin{equation}
%    (\widetilde{C_{T_n}},F_{T_n}) \stackrel{d}{\to} (\lambda\cdot\Exc,F_U). 
%    \label{eq:convCouple}
%  \end{equation}
%  Let $C \subset \CCC[0,1]$ be the set of excursions 
%  with the distinct minima property.
%From \cref{Lem:Exc_DMP}, we know that
%\[\proba ( (\Exc,F_U) \in C \times \{F_U\} ) = \proba( \Exc \in C) =1.\]
%The same of course holds, replacing $\Exc$ by $\lambda\cdot\Exc$.
%But from \cref{lem:continuityProbTree}, the set $D$ of discontinuity points of $\ProbTree(\patterntree;\ldots)$
%is included in the complement of $C \times \{F_U\}$ and thus $\proba( (\lambda \cdot \Exc,F_U) \in D)=0$.
%Therefore, using \cite[Theorem 25.7]{BillingsleyProbMeasure}, Eq.~\eqref{eq:convCouple}
%implies that 
%\[\ProbTree(\patterntree;\widetilde{C_{T_n}},F_{T_n})  \stackrel{d}{\to} \ProbTree(\patterntree;\lambda \cdot \Exc,F_U).\]
%The first part of the corollary follows, observing that $\ProbTree(\patterntree;f,F)$ is invariant by multiplication of the excursion $f$ by a scalar.

%Since the variables in consideration take value in $[0,1]$,
%convergence of distribution implies the convergence  of expectations,
%which completes the proof.
\end{proof}

%%%%%%%%%%%%%%%%%%%%%%%%%%%%%%%%%%%%%%%%%%%%%%%%%%%%%%%%%%%%%%%%%%%%%%%%%%%%%%%%%%%%%%%%%%%%%%%%%%%%%%%%%%%%%%%%%%
\section{Signs are asymptotically balanced and independent}
\label{sec:signs}
We now return to the study of signed trees. The purpose of this section is to justify that 
asymptotically (when $n$ goes large) the $k-1$ signs in the tree $\Tree_{\pm}(\widetilde{C_{T_n}},S_n,\ll)$ are independent and balanced, 
when $\ll$ is a $k$-element subset of the set of leaves of $T_n$ chosen uniformly at random:

\begin{proposition} \label{prop:balanced_signs_discrete}
    Let $\Si_n$ be a uniform random separable permutation of size $n$ and $(\patterntree,\eps)$
    be a signed binary tree of size $k$.
    As usual, let $(\widetilde{C_{T_n}},S_n)$ be the signed contour of the decomposition tree of $\Si_n$.
    We also consider a uniform random $k$-element subset $\ll$ of leaves of $T_n$. Then
    \begin{equation}                                                                                      
        \proba^{\Si_n,\ll}\big( \Tree_{\pm}(\widetilde{C_{T_n}},S_n,\ll) = (\patterntree,\eps) \big) = 
        \frac{1}{2^{k-1}} \proba^{\Si_n,\ll}\big( \Tree(\widetilde{C_{T_n}},\ll)=\patterntree \big)+o(1).
        \label{EqKeyFini}                                                                                     
    \end{equation}
\end{proposition}
The proof is given at the end of the present section. 
The core of this proof is \cref{lem:balanced_heights} below,
regarding heights of branching points of
marked leaves in uniform random Schröder trees.
We first introduce notation.

Recall that $T_n$ is a uniform Schr\"oder tree of size $n$. We take a set of $k$ leaves of $T_n$ uniformly at random.  
As in Definition~\ref{dfn:leaves_coordinates} (p.\pageref{dfn:leaves_coordinates}), denote their $x$-coordinates by $\ell_1 <\cdots< \ell_k$, and $\ll=(\ell_1,\dots,\ell_k)$ and let $H_i$ be the height of $\ell_i$ in $T_n$.
For $i=1,\dots,k-1$, denote by $M_i$ the height of the common ancestor of the leaves $\ell_i$ and $\ell_{i+1}$
(defining the height as the distance from the root).
\begin{lemma}\label{lem:balanced_heights}
When $n\to+\infty$, the parities of $M_1,\dots,M_{k-1}$ are asymptotically balanced, independent
from each other and from the subtree $\patterntree$ induced by the $k$ leaves.

More formally, we fix a binary tree $\patterntree$ with $k$ leaves and we condition on the event that the subtree of $T_n$ induced by $\ll$ is $\patterntree$. 
For all $\delta_1,\dots,\delta_{k-1}\in\{0,1\}$,
$$
\proba \left(M_1 \equiv\delta_1,\dots, M_{k-1} \equiv\delta_{k-1}\ |\ \Tree\big(\widetilde{C_{T_n}},\ll\big)=\patterntree \right)\stackrel{n\to +\infty}{\longrightarrow} \frac{1}{2^{k-1}},
$$
where $a\equiv b$ means $a=b \mod 2$.
\end{lemma}

\begin{proof}
The proof essentially relies on a subtree exchangeability argument.
In order to give intuition we start by the case $k=2$.

    In this case, there is a unique binary tree with two leaves,
    so the conditioning is not relevant.
    We want to prove that the common ancestor of two leaves
   $\ell_1$ and $\ell_2$, chosen uniformly at random in $T_n$ among all $2$-element subsets of leaves,
    has an even height with probability asymptotically $1/2$.

Fix two integers $n,h\geq 1$. We consider the set $\mathcal{T}^h_{n}$ of Schröder trees $t$ with $n$ leaves,
with two marked leaves $\ell_1$ and $\ell_2$ ($\ell_1<\ell_2$), such that $\ell_1$ has height $h$. Such a tree $t$ can be canonically described as follows:
\begin{itemize}
\item Take a chain $v_0,\cdots,v_h=\ell_1$ going from the root of $t$ to the first marked leaf.
\item Then for each $x<h$, glue a tree $T^l_x$  and a tree $T^r_x$ respectively on the left and on the right of $v_x$.
\item One of the trees $T^r_{x}$, say $T^r_{x_0}$, contains the second marked leaf $\ell_2$.
Then the vertex $v_{x_0}$ is the common ancestor of $\ell_1$ and $\ell_2$ and has height $M_1=x_0$. 
\end{itemize}
For every permutation $\alpha$ of $\{0,1,\dots,h-1\}$, replacing each pair $(T^l_x,T^r_x)$ by $(T^l_{\alpha(x)},T^r_{\alpha(x)})$ in the above decomposition provides a bijection from $\mathcal{T}^h_{n}$ to itself. This bijection maps a tree with $M_1=x_0$ onto a tree with $M_1=\alpha(x_0)$.

Applying this construction to our random tree $T_n$ with its two marked leaves, we get that for each $h,x$ with $0\leq x\leq h-1$
\[
\mathbb{P}(M_1=x\ |\ H_1=h)=\frac{1}{h}.
\] 
Therefore, we obtain 
\begin{align*}
\mathbb{P}(M_1\text{ is odd}) & =\sum_{h\geq 1}\mathbb{P}(M_1\text{ is odd}\ |\ H_1=h)\mathbb{P}(H_1=h) \\
& =\sum_{h\geq 1}\frac{\lfloor h/2\rfloor}{h}\mathbb{P}(H_1=h) = \frac{1}{2} - \mathcal{O}\left(\esper\left(\frac{1}{H_1}\right)\right),
\end{align*}
which goes to $1/2$, as soon as $H_1$ goes to infinity in probability. This will be proved in \cref{Lem:HauteurFeuille} below.
\bigskip

We now prove \cref{lem:balanced_heights} in the general case $k\geq 2$. 
Since the proof uses a lot of notation, the reader is invited
to look regularly at \cref{Fig:SchemaPreuveEchangeabilite},
which illustrates the main definitions.
\begin{figure}[ht]
\[ \includegraphics{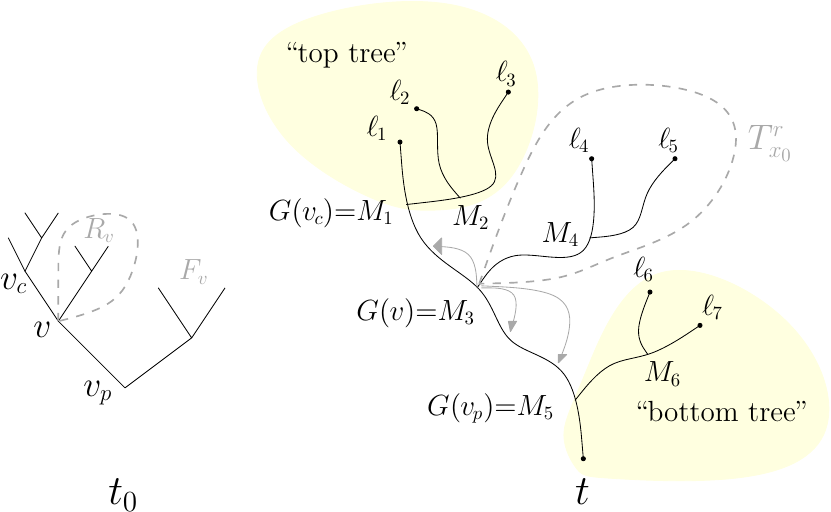} \]
\caption{The subtree exchangeability argument.}\label{Fig:SchemaPreuveEchangeabilite}
\end{figure}

We fix $k\geq 2$ and a binary tree $\patterntree$ with $k$ leaves. We denote by $\mathcal{T}_{n}(\patterntree)$ the set of Schröder trees of size $n$ with $k$ marked leaves $\ell_1<\dots<\ell_k$ which induce the subtree $\patterntree$. 
For such a tree $t$, each vertex $v$ in $\patterntree$ corresponds to a vertex $G(v)$ in $t$: 
$G(v)$ is either one of the marked leaves of $t$ (if $v$ is a leaf of $\patterntree$) or a common ancestor of two marked leaves of $t$ (if $v$ is an internal vertex of $\patterntree$).
We denote by $h_t(v)$ the height of $G(v)$ in $t$. 

We fix an internal vertex $v$ of $\patterntree$. 
Define $v_c$ as the left child of $v$ in $\patterntree$ and $v_p$ as the parent of $v$.
For $t$ in $\mathcal{T}_{n}(\patterntree)$, 
the distance between $G(v_c)$ and $G(v_p)$ is $d_t(v):=h_t(v_c)-h_t(v_p)$.
(If $v$ is the root, then $v_p$ is not defined
and we shall think of $G(v_p)$ as a virtual new root of $t$
that is an ancestor of every vertex in $t$ and has height $-1$;
in particular, we take $h_t(v_p)=-1$.)

The vertices of $\patterntree$ (both leaves and internal vertices, $v$ excluded) are partitioned in two sets:
the right subtree $R_v$ of $v$ and its complement set $F_v$.
For a sequence of heights $\mathbf{h}=(h_u)_{u\in F_v}$, we define 
$$
\mathcal{T}^{v,\mathbf{h}}_{n}(\patterntree)=\left\{
t\in \mathcal{T}_{n}(\patterntree),\, \text{for all }u\in F_v,\, h_t(u)=h_u
\right\}.
$$
Note that, for every $t$ in this set, $d_t(v)=h_{v_c}-h_{v_p}$;
we denote this common value by $d_{\mathbf{h},v}$. 
(If $v$ is the root of $\patterntree$, we set $h_{v_p}=-1$.)

Similarly to the decomposition in the case $k=2$, each tree in $\mathcal{T}^{v,\mathbf{h}}_{n}(\patterntree)$ can be decomposed as follows.
\begin{itemize}
\item A chain $w_0,\dots,w_{d_{\mathbf{h},v}}$ going from $G(v_p)$ to $G(v_c)$. 
\item For each $0<x< d_{\mathbf{h},v}$, a tree $T^l_x$  and a tree $T^r_x$ respectively on the left and on the right of $w_x$.
\item One of the trees $T^r_{x}$, say $T^r_{x_0}$, contains $G(R_v)$. Then the vertex $w_{x_0}$ is $G(v)$
 and has height $h_{v_p}+x_0$.
\item A "top tree" which consists of the offspring of $G(v_c)$.
\item All others vertices of $t$ form a connected subtree called the "bottom tree". (If $v$ is the root of $\patterntree$, there is no bottom tree.)
\end{itemize}
For every permutation $\alpha$ of $\{1,\dots,d_{\mathbf{h},v}-1\}$, replacing each pair $(T^l_x,T^r_x)$ by $(T^l_{\alpha(x)},T^r_{\alpha(x)})$ in the above decomposition provides a bijection from $\mathcal{T}^{v,\mathbf{h}}_{n}(\patterntree)$ to \hbox{itself}. Clearly, this bijection maps a tree with $h_t(v)=h_{v_p}+x_0$ onto a tree with \hbox{$h_t(v)=h_{v_p}+\alpha(x_0)$}.

We now work with a uniform Schröder tree $T_n$ of size $n$ with $k$ marked leaves, and we condition on the fact that these leaves induce the subtree $\patterntree$.
We use the notation $\proba_{\patterntree}(T_n\in A):=\proba(T_n\in A\, | \, T_n\in\mathcal{T}_n(t_0))$.
The above discussion implies that for every $h_{v_p}< x< h_{v_c}$ we have, 
\[
\proba_{\patterntree}\left(h_{T_n}(v)=x |\ T_n\in\mathcal{T}^{v,\mathbf{h}}_{n}(\patterntree)\right) =\frac{1}{h_{v_c}-h_{v_p}-1}
=\frac{1}{d_{\mathbf{h},v}-1}=\frac{1}{d_{T_n}(v)-1}.
\]
Therefore
$$
\Big|\proba_{\patterntree}\left(h_{T_n}(v)\text{ is even} |\ T_n\in\mathcal{T}^{v,\mathbf{h}}_{n}(\patterntree)\right) -\frac12\Big| \leq \frac{1}{d_{T_n}(v)-1}.
$$
We introduce a subset $E^v_n$ of $\mathcal{T}_{n}(\patterntree)$ of {\em well-behaved} trees
\[E^v_n := \big\{ t \in \mathcal{T}_{n}(\patterntree),
\, d_t(v) \ge n^{1/4} \big\}. \]
We will prove in \cref{Lem:HauteurFeuille} that for all $v$, $\proba_{\patterntree}(E^v_n)\to 1$.

Note that if $\mathbf{h}$ is such that $\mathcal{T}^{v,\mathbf{h}}_{n}(\patterntree)\subset E^v_n$, then
\[
\Big|\proba_{\patterntree}\left(h_{T_n}(v)\text{ is even} |\ T_n\in\mathcal{T}^{v,\mathbf{h}}_{n}(\patterntree)\right) -\frac12\Big| \leq \frac{1}{n^{1/4}-1}.
\]
As a consequence, we have 
\begin{equation}\label{Eq:PresqueUnDemi}
\proba_{\patterntree}\left(h_{T_n}(v)\text{ is even} |\ T_n\in\mathcal{T}^{v,\mathbf{h}}_{n}(\patterntree)\right)  = \frac12 + o(1),
\end{equation}
where the error term $o(1)$ is bounded independently of $\mathbf{h}$, provided that $\mathcal{T}^{v,\mathbf{h}}_{n}(\patterntree)\subset E^v_n$.

Internal vertices of $\patterntree$ can be canonically indexed by $\{1,\dots,k-1\}$:
$v_i$ is the common ancestor of $\ell_i$, $\ell_{i+1}$.
By definition, $M_i=h_t(v_i)$ and we denote \hbox{$E^i_n:=E^{v_i}_n$}.

Fix $i \ge 1$. For $j<i$ the vertex $v_j$ lies in $F_{v_i}$,
so that the fact that $T_n$ belongs to $\mathcal{T}^{v_i,\mathbf{h}}_{n}(\patterntree)$ 
determines the random variables $(M_j)_{j<i}$.
Clearly, this also determines whether $T_n$ belongs to $E^i_n$ or not.
Consequently for all $i\leq k-1$, the event
\[ \big\{M_1 \equiv \delta_1, \dots,M_{i-1} \equiv \delta_{i-1},T_n\in E^{i}_n\big\} \]
can be written as a disjoint union (on $\mathbf{h}$) of events of the form $T_n \in \mathcal{T}^{v_i,\mathbf{h}}_{n}(\patterntree)$.
\cref{Eq:PresqueUnDemi} yields
$$
\proba_{\patterntree}\left(M_{i}\equiv\delta_{i} |\ M_1 \equiv \delta_1, \dots,M_{i-1} \equiv \delta_{i-1},T_n\in E^{i}_n\right) =\frac12 +o(1).
$$
Furthermore, since $\proba_{\patterntree}(T_n\in E^{k-1}_n)\to 1$ from \cref{Lem:HauteurFeuille}, we have $\proba_{\patterntree}(A)=\proba_{\patterntree}(A,T_n\in E^{k-1}_n)+o(1)$ for any event $A$. Using this twice, we obtain
\begin{align*}
\proba_{\patterntree}\left( M_1 \equiv \delta_1, \dots,M_{k-1} \equiv \delta_{k-1}\right)\hspace{-3cm}&\\
&=\proba_{\patterntree}\left( M_1 \equiv \delta_1, \dots,M_{k-1} \equiv \delta_{k-1},T_n\in E_n^{k-1}\right)+o(1)\\
&=\proba_{\patterntree}\left(M_{k-1} \equiv \delta_{k-1}| M_1 \equiv \delta_1, \dots,M_{k-2} \equiv \delta_{k-2},T_n\in E_n^{k-1}\right)\\
&\qquad \times \big(\proba_{\patterntree}\left(M_1 \equiv \delta_1, \dots,M_{k-2} \equiv \delta_{k-2}\right)+o(1)\big)+o(1)\\
&= \tfrac{1}{2}\proba_{\patterntree}\left(M_1 \equiv \delta_1, \dots,M_{k-2} \equiv \delta_{k-2}\right)+o(1).
\end{align*}
Repeating this argument $k-2$ times concludes the proof.
\end{proof}

It remains to prove that with high probability a uniform Schröder tree $T_n$ is \emph{well-behaved}:
the heights $H_i$ of marked leaves and $M_j$ of their common ancestors and the distances $|M_i-M_j|$ and $|M_i-H_j|$ are all larger than $n^{1/4}$ (in fact larger than $\eps_{n} n^{\frac12}$ for all $\eps_n\to 0$).
\begin{lemma}\label{Lem:HauteurFeuille}
We re-use the notation of the proof of \cref{lem:balanced_heights}. 
In particular, $E^v_n$ denotes the set of trees $t$ in $\mathcal{T}_{n}(\patterntree)$ such that $d_t(v) \ge n^{1/4}$. 
For all internal vertices $v$ of $\patterntree$, 
$$
\proba(T_n\notin E^v_n\, | \, T_n\in\mathcal{T}_n(t_0))\to 0.
$$

In the special case where 
$\patterntree = \begin{array}{c} \tikz{
\begin{scope}[scale=.1,rotate=180,level distance=4cm,sibling distance=3cm]
\node[circle, draw, inner sep=1pt] {} 
 child {node[inner sep=0pt] {~}}
 child {node[inner sep=0pt] {~}};
\end{scope}}\end{array}$ (i.e. $k=2$), 
every pair $(\ell_1,\ell_2)$ of marked leaves of a Schröder tree with $n \geq 2$ leaves induces $\patterntree$, 
and the height of $\ell_1$ 
goes to infinity in probability when $n \to +\infty$.
\end{lemma}

\begin{proof}
  First, \cref{cor:PrTree_cv_in_distribution} gives 
  the probability of the conditioning event in the limit: 
  \[\lim_{n \to +\infty} \proba(T_n\in\mathcal{T}_n(t_0)) = \proba(\Tree(\Exc,\UU)=\patterntree),\]
  where $\UU$ is a set of $k$ i.i.d uniform random variables in $[0,1]$. 
  From \cref{Lemma:ArbreBinaireUniforme}, we know that this quantity is strictly larger than $0$. 
Therefore it is enough to prove that \hbox{$\proba(T_n\notin E^v_n) \to 0$} without the conditioning.

By definition, for each internal vertex $v$ in $\patterntree$, $d_t(v)$ is one of the following:
\begin{itemize}
  \item the height $H_i$ of a marked leaf;
  \item the height $M_i$ of the common ancestor of two marked leaves;
  \item the difference of heights $|M_i-M_j|$ between two common ancestors;
  \item the difference of heights $|M_i-H_j|$ between a leaf and a common ancestor.
\end{itemize}
Thus the first statement of the lemma will follow if we show that
$$
\mathbb{P}
\begin{pmatrix}
\text{ There are }i\leq k, j\leq k-1\text{ such that }\\
H_i \leq n^{1/4}\text{ or }\  M_j \leq n^{1/4}\text{ or }\ \left|H_i-M_{j} \right|\leq n^{1/4}\\
 \text{ or }\left|M_i-M_j \right|\leq n^{1/4} (i \neq j) 
\end{pmatrix}
\to 0.
$$ 
Actually, we prove the stronger statement that for any two $i,j$, all sequences of random variables
$
\frac{1}{\sqrt{n}}H_i, \frac{1}{\sqrt{n}}M_j, \frac{1}{\sqrt{n}}|M_i-M_j|, \frac{1}{\sqrt{n}}|H_i-M_j|
$
converge in distribution to positive random variables. 

We will only write the details for the sequence $\left(\frac{1}{\sqrt{n}}|M_i-M_j|\right)$,
the proof being identical in the other cases.
Recall that the $x$-coordinate of $\ell_{i}$ has the same distribution as $F_{T_n}^\star(U_{i})$, where
\begin{itemize}
\item $U_{i}$ is the $i$-th smallest value among $k$ i.i.d. uniform random variables in $[0,1]$. 
\item $F_{T_n}^\star$ is the pseudo-inverse of $F_{T_n}$
  defined in \cref{eq:EqDFLeaves} (p.\pageref{eq:EqDFLeaves});
  combining \cref{PropLeavesUniform,lem:FF*Pareil} (p.\pageref{PropLeavesUniform}),
  we know that $F_{T_n}^\star(x)$ tends to $x$ when $n$ tends to infinity (uniformly in $x$).
\end{itemize}

By construction we have (in what follows $\widetilde{C_{T_n}}$ is the normalized contour function of $T_n$, and the $U_i$'s are independent from $\widetilde{C_{T_n}}$):  
\[
M_i = \min_{[F_{T_n}^\star(U_{i}),F_{T_n}^\star(U_{i+1})]}\sqrt{n}\times \widetilde{C_{T_n}}
%\quad \text{ and thus } \quad 
%\left| \frac{M_i}{\sqrt{n}} - \min_{[U_{i},U_{i+1}]} \widetilde{C_{T_n}} \right|
%\leq \omega(\widetilde{C_{T_n}},\delta_n)
\]
%where $\delta_n=\sup_x|F_{T_n}^{\star}(x)-x|$ and
%$\omega$ is the modulus of continuity already defined in \cref{eq:def_modulus_continuity} (p.\pageref{eq:def_modulus_continuity}).
%It follows that 
%\begin{align}
%\left| \frac{|M_i-M_j|}{\sqrt{n}} - \big|\min_{[U_{i},U_{i+1}]} \widetilde{C_{T_n}} -\min_{[U_{j},U_{j+1}]} \widetilde{C_{T_n}}\big| \right|
%\leq 2\omega(\widetilde{C_{T_n}},\delta_n).\label{Eq:MajorationHiHj}
%\end{align}
%
%But from \cref{lem:FF*Pareil,PropLeavesUniform}, $\delta_n\stackrel{\text{(prob.)}}{\to} 0$.
%Moreover, the sequence of random functions $\big(\widetilde{C_{T_n}}\big)_n$ converges in $\mathcal{C}([0,1])$ from \cref{prop:ConvergenceContour} (p.\pageref{prop:ConvergenceContour}) and therefore is tight,
%implying that $\omega(\widetilde{C_{T_n}},\delta_n)\stackrel{\text{(prob.)}}{\to}0$ (see \cite[Theorem 7.3]{Billingsley}). 
%Now, for every deterministic $u_i \le u_{i+1} \le u_j \le u_{j+1}$ in $[0,1]$,
%we have
%$$
%\big|\min_{[u_{i},u_{i+1}]} \widetilde{C_{T_n}}
%-\min_{[u_{j},u_{j+1}]} \widetilde{C_{T_n}}\big|\stackrel{(d)}{\to} \lambda \big|\min_{[u_{i},u_{i+1}]} \Exc
%-\min_{[u_{j},u_{j+1}]} \Exc\big|
%$$
%by \cref{prop:ConvergenceContour}. We now integrate $u_i,u_{i+1},u_j,u_{j+1}$
%with respect to the joint distribution of $(U_{i},U_{i+1},U_{j},U_{j+1})$
Since $\big(\widetilde{C_{T_n}}\big)_n$ converges to $\lambda \cdot \Exc$
(\cref{prop:ConvergenceContour}, p.\pageref{prop:ConvergenceContour})
and $F_{T_n}^\star(U_{i})$ tends to $U_i$,
 we have:
$$
%\big|\min_{[U_{i},U_{i+1}]} \widetilde{C_{T_n}}
%-\min_{[U_{j},U_{j+1}]} \widetilde{C_{T_n}}\big|
\frac{1}{\sqrt{n}}|M_i-M_j|
\stackrel{(d)}{\to} \lambda \big|\min_{[U_{i},U_{i+1}]} \Exc
-\min_{[U_{j},U_{j+1}]} \Exc\big|,
$$
This completes the proof of the first statement of \cref{Lem:HauteurFeuille}. 

In particular, we have proved that $H_1$ goes to infinity in probability when $n \to +\infty$.
The second statement of the lemma is just a rephrasing of this, 
in the particular case where
$\patterntree = \begin{array}{c} \tikz{
\begin{scope}[scale=.1,rotate=180,level distance=4cm,sibling distance=3cm]
\node[circle, draw, inner sep=1pt] {} 
 child {node[inner sep=0pt] {~}}
 child {node[inner sep=0pt] {~}};
\end{scope}}\end{array}$.
\end{proof}
\begin{remark}
  The joint convergence of the $\tfrac{1}{\sqrt{n}}H_i$ and $\tfrac{1}{\sqrt{n}}M_i$
  is observed by Aldous in the proof of \cite[Theorem 20]{AldousCRT3},
  with the slight difference that he works with uniform random {\em vertices} in the tree,
  while we consider uniform random {\em leaves}.
\end{remark}

\begin{proof}[Proof of \cref{prop:balanced_signs_discrete}]
Let $\Si_n$ be a uniform separable permutation of size $n$.
From \cref{cor:same_distibution_separables_trees} (p.\pageref{cor:same_distibution_separables_trees}), it has the same distribution
as $\perm(T_n,\Eps_n)$ where $T_n$ is a uniform Schröder tree with $n$ leaves 
and $\Eps_n$ is the sign function on the internal vertices of $T_n$,
such that the signs alternate and such that the root $r$ of $T_n$ has a balanced sign:
$\Eps_n(r)=+$ with probability $1/2$.
Recall that $\widetilde{C_{T_n}}$ (resp. $(\widetilde{C_{T_n}},S_n)$) denotes the normalized contour of $T_n$ (resp. the signed contour of $(T_n,\Eps_n)$). 

As in the statement of the proposition, consider a uniform random $k$-element subset $\ll$ of leaves of $T_n$. 
Recall from Observations~\ref{obs:ExtractedTree=Subtree} and~\ref{obs:Tree_pm} that $\Tree(\widetilde{C_{T_n}},\ll)$ 
(resp. $\Tree_\pm(\widetilde{C_{T_n}},S_n,\ll)$)
is the subtree of $T_n$ 
(resp. the signed subtree of $(T_n,\Eps_n)$)
induced by these leaves.
We condition on the fact that $\Tree(\widetilde{C_{T_n}},\ll)=t_0$.
It is enough to prove that
the probability that the signs in $\Tree_\pm(\widetilde{C_{T_n}},S_n,\ll)$
coincide with a fixed sign function $\eps$ on the internal vertices of $\patterntree$ is $1/2^{k-1}+o(1)$.
Recall that the signs in $\Tree_\pm(\widetilde{C_{T_n}},S_n,\ll)$ correspond to signs
of the common ancestors of the marked leaves in $(T_n,\Eps_n)$.

We further condition on the fact that the root of $T_n$ has sign $+$.
Then a vertex in $T_n$ has sign $+$ if it has even height and sign $-$ if it has odd height.
Therefore the probability that the common ancestors have given signs
correspond to the probability that their heights have given parities.
This probability is $1/2^{k-1}+o(1)$ by \cref{lem:balanced_heights}.

The same holds if we had conditioned on the fact that the root of $T_n$ has sign $-$.
We can therefore conclude, that, conditioning only on the event $\Tree(\widetilde{C_{T_n}},\ll)=t_0$,
the probability that the common ancestors have given signs
is $1/2^{k-1}+o(1)$, which completes the proof of the proposition.
\end{proof}

%%%%%%%%%%%%%%%%%%%%%%%%%%%%%%%%%%%%%%%%%%%%%%%%%%%%%%%%%%%%%%%%%%%%%%%%%%%%%%%%%%%%%%%%%%%%%%%%%%%%%%%%%%%%%%%%%%%%%%%%%
\section{Conclusion of the proof}
\label{sec:main_proof}

In this section we prove that the convergence of the expectation for $\ProbTree$ implies the one of the 
expectation for  $\ProbPerm$. In order to do that, the expectation for $\ProbPerm$ is expressed in terms of the one 
for $\ProbTree$ (as a linear combination) in the continuous and discrete cases. This is possible since the signs are  balanced and independant as it has been proven in Proposition~\ref{prop:balanced_signs_discrete}.  

\begin{lemma} \label{lem:expectation_tree_to_perm_continuous}
    Let $\pi$ be a pattern of size $k$ and $(\Exc,S)$ be the signed Brownian excursion.
    \[\esper^{\Exc,S} \big[ \ProbPerm(\pi;e,S, F_{U}) \big] = \frac{1}{2^{k-1}} \sum_{(\patterntree,\eps)} \esper^\Exc \left[  \ProbTree(\patterntree;\Exc,F_U)\right] , \]
    where the sum runs over all signed binary trees $(\patterntree,\eps) $ of $\pi$.
\end{lemma}
\begin{proof}
    By definition,
    \[ \ProbPerm(\pi;e,S, F_{U}) = \proba^{\XX}( \Perm(\Exc,S,\XX) =\pi ),\]
    where $\XX$ is a set of $k$ independent uniform variables in $[0,1]$.
    Thus, with Observation~\ref{obs:expectation_of_expectation} (p.\pageref{obs:expectation_of_expectation}), 
    \[ \esper^{\Exc,S}\big[ \ProbPerm(\pi;e,S, F_{U}) \big]  = \proba^{\,\Exc,S,\XX}(  \Perm(\Exc,S,\XX) =\pi ).\]
    Since $\Perm(\Exc,S,\XX) =\perm(\Tree_{\pm}(\Exc,S,\XX))$ this is exactly the probability that 
    $\Tree_{\pm}(\Exc,S,\XX)$ is equal to one of the signed trees that are pre-images of $\pi$ by $\perm$.
    In other terms,
    \[ \esper^{\Exc,S}\big[ \ProbPerm(\pi;e,S, F_{U}) \big] =
    \sum_{(\patterntree,\eps)} \proba^{\,\Exc,S,\XX}\big( \Tree_{\pm}(\Exc,S,\XX) =(\patterntree,\eps) \big),\]
    where the sum runs over all signed trees of $\pi$.
    From \cref{obs:distinctMinImpliesBinary} (p.\pageref{obs:distinctMinImpliesBinary}),
    if $t_0$ is not binary, then $\Tree(\Exc,S,\XX)=t_0$ has probability $0$,
    since  from \cref{Lemma:MinimumSameLevel} a.s. all minima of $\Exc$ have distinct values.
    Consequently, in this case $\Tree_{\pm}(\Exc,S,\XX) =(\patterntree,\eps)$ has also probability $0$, 
    and the sum above can be restricted to the
    signed {\em binary} trees of $\pi$.
    For each such signed binary tree $(\patterntree,\eps)$, we have
    \begin{align*}
        \proba^{\,\Exc,S,\XX}\big( \Tree_{\pm}(\Exc,S,\XX) = (\patterntree,\eps) \big) &=  \proba^{\,\Exc,S,\XX}\big( \Tree_{\pm}(\Exc,S,\XX) = (\patterntree,\eps) | \Tree(\Exc,\XX)=\patterntree \big) \\& \qquad  \times\proba^{\,\Exc,\XX}\big( \Tree(\Exc,\XX)=\patterntree \big)\\ &= \frac{1}{2^{k-1}} \proba^{\,\Exc,\XX}\big( \Tree(\Exc,\XX)=\patterntree \big).
        \label{EqKey}
    \end{align*}
    Indeed, signs are taken independently on local minima of $\Exc$ and thus on vertices of $\patterntree$.
    Finally, with Observation~\ref{obs:expectation_of_expectation} again, 
    the probability on the right hand-side is 
    \begin{equation}
    \proba^{\,\Exc,\XX}\big( \Tree(\Exc,\XX)=\patterntree )= \esper^{\Exc} \big[ \proba^{\XX}(\Tree(\Exc,\XX)=\patterntree )\big]= \esper^{\Exc} \big[ \ProbTree(\patterntree;\Exc,F_U)\big],
    \label{eq:ExpInProbForPrTree} 
    \end{equation}
    which ends the proof.
\end{proof}

Using Proposition~\ref{prop:balanced_signs_discrete}, we obtain a discrete analogue of this lemma:

\begin{lemma}
    Let $\pi$ be a pattern of size $k$ and $\Si_n$ be a uniform random separable permutation of size $n$.
    As usual, let $(\widetilde{C_{T_n}},S_n)$ be the signed contour of the decomposition tree of $\Si_n$.
    \[\esper^{\Si_n} \big[ \ProbPerm(\pi;\widetilde{C_{T_n}},S_n, F_{T_n}) \big]
    =\frac{1}{2^{k-1}} \sum_{(\patterntree,\eps)} \esper^{\Si_n} \left[  \ProbTree(\patterntree;\widetilde{C_{T_n}},F_{T_n})\right] + o(1),\]
where the sum runs over all signed binary trees $(\patterntree,\eps) $ of $\pi$. 
    \label{lem:expectation_tree_to_perm_discrete}
\end{lemma}

\begin{proof}
  Let $\XX$ be a set of $k$ independent variables taken with distribution $F_{T_n}$. 
  Remember that this amounts to choosing $k$ leaves of $T_n$ independently and uniformly at random. 
Using an argument similar to that in the previous proof, we have
\begin{equation}
  \esper^{\Si_n} \big[ \ProbPerm(\pi;\widetilde{C_{T_n}},S_n, F_{T_n}) \big]
=  \sum_{(\patterntree,\eps)} 
\proba^{\Si_n, \XX}\big( \Tree_{\pm}(\widetilde{C_{T_n}},S_n,\XX) = (\patterntree,\eps) \big), 
\label{Eq:esperProbPermFinite_SumTrees}
\end{equation} 
where the sum runs over all signed trees $(\patterntree,\eps) $ of $\pi$. 

It holds that  
\begin{equation*}
  \proba^{\Si_n, \XX}\big( \Tree_{\pm}(\widetilde{C_{T_n}},S_n,\XX) = (\patterntree,\eps) \big)
\le \proba^{\Si_n, \XX}\big( \Tree(\widetilde{C_{T_n}},\XX)=\patterntree\big) 
\end{equation*}
and 
\begin{multline*}
\proba^{\Si_n, \XX}\big( \Tree(\widetilde{C_{T_n}},\XX)=\patterntree\big) 
= \esper^{\Si_n} \big[ \proba^{\XX}\big( \Tree(\widetilde{C_{T_n}},\XX)=\patterntree\big)\big]
=\esper^{\Si_n} \big[ \ProbTree(t_0;\widetilde{C_{T_n}},F_{T_n})\big].
\end{multline*}
If $t_0$ is not binary, using \cref{cor:PrTree_cv_in_distribution} (p.\pageref{cor:PrTree_cv_in_distribution}), we further have
\begin{equation*}
\esper^{\Si_n} \big[ \ProbTree(t_0;\widetilde{C_{T_n}},F_{T_n})\big] \to 
\esper^e \big[  \ProbTree(t_0;e,F_U) \big]=0.
\end{equation*}
Recall indeed that, since a.s. all minima of $\Exc$ have distinct values from \cref{Lemma:MinimumSameLevel},
the trees extracted from $\Exc$ at uniformly distributed points are binary with probability $1$.

Therefore the sum in \cref{Eq:esperProbPermFinite_SumTrees} can be replaced by a sum
over signed {\em binary} trees $(\patterntree,\eps) $ of $\pi$,
making only an error of $o(1)$.
For signed binary trees $(\patterntree,\eps) $, 
we use Proposition~\ref{prop:balanced_signs_discrete} (p.\pageref{prop:balanced_signs_discrete}) to derive that  
\begin{align*}
\proba^{\Si_n,\XX}\big( &\Tree_{\pm}(\widetilde{C_{T_n}},S_n,\XX) = (\patterntree,\eps) \big) \\
& = \proba^{\Si_n,\XX}\big( \Tree_{\pm}(\widetilde{C_{T_n}},S_n,\XX) = (\patterntree,\eps) \mid \text{ there is no repetition in } \XX \big) +o(1) \\
& = \frac{1}{2^{k-1}} \proba^{\Si_n,\XX}\big( \Tree(\widetilde{C_{T_n}},\XX)=\patterntree  \mid \text{ there is no repetition in } \XX \big)+o(1) \\
& =\frac{1}{2^{k-1}} \proba^{\Si_n,\XX}\big( \Tree(\widetilde{C_{T_n}},\XX)=\patterntree \big)+o(1),
\end{align*}
Finally, since 
\[\proba^{\Si_n,\XX}\big( \Tree(\widetilde{C_{T_n}},\XX)=\patterntree \big)
= \esper^{\Si_n} \big[ \proba^{\XX}\big( \Tree(\widetilde{C_{T_n}},\XX)=\patterntree \big)\big]
= \esper^{\Si_n} \big[ \ProbTree(\patterntree;\widetilde{C_{T_n}},F_{T_n})\big],\]
this ends the proof.
\end{proof}

Using these expressions of the expectation for $\ProbPerm$ in terms of the one for $\ProbTree$ in the continuous and discrete cases, 
we can now prove our main result:

\begin{theorem}
For any pattern $\pi$,
\[ \esper\left[ \occ(\pi,\Si_n) \right] \longrightarrow \esper\left[ \Lambda_{\pi} \right]. \]
\label{thm:ConvExp}
\end{theorem}

\begin{proof}
    From \cref{cor:PrTree_cv_in_distribution}, we have that, for each tree $\patterntree$,
    \[ \esper^{\Si_n} \big[ \ProbTree(\patterntree;\widetilde{C_{T_n}},F_{T_n}) \big]  \to \esper^\Exc \big[ \ProbTree(\patterntree;\Exc,F_U) \big].\]
Combining this with \cref{lem:expectation_tree_to_perm_continuous,lem:expectation_tree_to_perm_discrete}, we get that for each pattern $\pi$,
\[ \esper^{\Si_n} \big[ \ProbPerm(\pi;\widetilde{C_{T_n}},S_n, F_{T_n}) \big]  \to \esper^{\Exc,S} \big[  \ProbPerm(\pi;e,S, F_{U}) \big].\]
From \cref{cor:occ_from_PrPerm} (p.\pageref{cor:occ_from_PrPerm}),  $\occ(\pi,\Si_n) = \ProbPerm(\pi;\widetilde{C_{T_n}},S_n, F_{T_n}) \left(1+ \mathcal{O}(\frac{1}{n})\right)$.
Together with Observation~\ref{obs:lamba_from_PrPerm} (stating that $\Lambda_{\pi} = \ProbPerm(\pi;e,S, F_{U})$), this ends the proof
of \cref{thm:ConvExp}.
\end{proof}

As shown by Corollary~\ref{cor:Exp_is_enough} (p.\pageref{cor:Exp_is_enough}), Theorems~\ref{thm:ConvJointMoments} and~\ref{thm:main}.\ref{item:main_joint} 
follow from Theorem~\ref{thm:ConvExp} above. 

\section{Permuton interpretation of our main result}
\label{sec:Permutons}
The goal of this section is to prove \cref{thm:Main_Permuton} (p.\pageref{thm:Main_Permuton}).
We first need additional material on permutons.
Recall from \cref{dfn:permuton} that permutons, which were introduced in \cite{Permutons},
are probability measures on $[0,1]^2$ with uniform marginals. 
Given a (deterministic) permuton $\mu$ and an integer $k$,
there is a natural way to define a random permutation 
$\Pi_k^\mu$ of size $k$.
  
\begin{definition}
  \label{dfn:cv_Permutons}
  Let $\mu$ be a permuton and $k$ be an integer.
  Take $k$ points in $[0,1]^2$
  independently according to $\mu$. 
  A.s. these $k$ points have distinct 
  $x$-coordinates and distinct $y$-coordinates 
  (since $\mu$ has uniform marginals). 
  Therefore we can order these points $(X_1,Y_1),\ldots, (X_k,Y_k)$
  such that $Y_1<Y_2<\cdots<Y_k$.
  Then $\Pi_k^\mu$ is the unique permutation
  such that $X_{\Pi_k^\mu(1)}<X_{\Pi_k^\mu(2)}<\cdots<X_{\Pi_k^\mu(k)}$.
  
  For a permutation $\pi$ of size $k$, 
  following the notation of \cite{Permutons}, 
  we then define $t(\pi,\mu)$ as the probability that $\Pi_k^\mu=\pi$.

  We say that a (deterministic) sequence of permutations $(\si_n)$ 
  with sizes going to infinity
  converges to the permuton $\mu$ if, for all $\pi$,
  $\occ(\pi,\si_n)$ tends to $t(\pi,\mu)$.
\end{definition}
It is noticed in \cite[Eq.(49)]{Permutons} that the convergence of $(\si_n)$
to a permuton $\mu$ is equivalent to the weak convergence of the associated permutons
$(\mu_{\sigma_n})$ to $\mu$.
In particular, this implies the uniqueness of limits of 
sequences of permutations. 
Moreover from \cite[Theorem 1.6 (i)]{Permutons} 
if $(\si_n)$ is a deterministic sequence of permutations 
such that $\occ(\pi,\si_n)$ has a limit for all $\pi$,
then there exists a (necessarily unique) permuton $\mu$ such that
$(\sigma_n)$ tends towards $\mu$.

We can now proceed to the proof of \cref{thm:Main_Permuton}. 
\begin{proof}[Proof of \cref{thm:Main_Permuton}]
  \cref{item:main_joint} in \cref{thm:main} asserts
  that the finite-dimensional laws of $(\occ(\pi,\Si_n))_\pi$
  converge to those of $(\Lambda_\pi)_\pi$ (here, vectors are 
  indexed by all permutations). 
%\comment{Lucas pour Valentin : Est-ce qu'on ne peut pas court-circuiter un peu l'argument Skorokhod, maintenant qu'on a écrit les arguments de façon plus directe pour le 2d papier? Est-ce que ça te paraît malhonnête de citer notre papier (en préparation donc) et d'en résumer la Section 2 en écrivant :\\
%A consequence of Lemmas 3.5 and 4.2 and the proof of Theorem 1.6 (i) in \cite{Permutons} is that this implies the existence of a permuton $\Mu$ such that  $\mu_{\Si_n}$ converges in distribution to $\Mu$, for the narrow topology.
%Note that the procedure given in the proof of \cite[Theorem 1.6 (i)]{Permutons} ensures that $\Mu$ is measurable.
%
%It should be noted that by construction, for any $\pi$, $\Lambda_{\pi} \stackrel{(d)}{=} t(\pi,\Mu)$, using the notation $t$ from \cref{dfn:cv_Permutons}, 
%since both are the limit in distribution of $\occ(\pi,\Si_n)$. Since  from Theorem~\ref{thm:main}.\ref{item:main_variance} $\Lambda_{12}$ is not deterministic,
%neither is $\Mu$.
%}
%V: Perso je trouve pas mal de laisser la version Skorohod,
%car on n'a pas besoin de regarder les détails de l'article Hoppen & co\ldots
%En plus ça évite de citer un article futur, ce qui est toujours un peu bizarre.
%J'ai un peu ajusté les arguments de mesurabilité.
%
%\comment{
%Voici l'ancienne version : (j'ai fait 2 modifs suite au referee)}

  This is equivalent to the convergence in distribution of the infinite-dimensional
  vector $(\occ(\pi,\Si_n))_\pi$ towards $(\Lambda_\pi)_\pi$ in the product topology
  (from the definitions of the convergence in distribution and of the product topology).
  From Skorohod's representation theorem \cite[Theorem 3.30]{Kallenberg},
  there exists a probability space $\Omega$,
  random variables $(\bm{O}_n)_{n \ge 1}$ and $\bm{\Lambda}'$
  such that: 
  \begin{enumerate}
    \item for each $n\ge 1$, $\bm{O}_n$ is a random vector $(O_{n,\pi})_\pi$
      indexed by all permutations that has the same law as $(\occ(\pi,\Si_n))_\pi$;
    \item $\bm{\Lambda}'$ is a random vector $(\Lambda'_\pi)_\pi$
      indexed by all permutations that has the same law as $(\Lambda_\pi)_\pi$;
    \item \label{item:conv_ps} we have sure convergence (in the product topology)
      of $\bm{O}_n$ to $\bm{\Lambda}'$ when $n$ tends to
      infinity, \emph{i.e.}, for any $\omega \in \Omega$, we have                                                   
      \[\text{for any permutation }\pi,\, O_{n,\pi}(\omega)
          \to \Lambda'(\omega)_\pi.\]
  \end{enumerate}

We now aim at constructing random variables $\Si'_n$ such that $(O_{n,\pi})_\pi=(\occ(\pi,\Si'_n))_\pi$, 
so that we can apply Theorem 1.6 (i) of \cite{Permutons} in each $\omega \in \Omega$. 

% LG 13/02 : modif suite au referee
Fix $n \ge 1$. Since $(O_{n,\pi})_\pi$ and $(\occ(\pi,\Si_n))_\pi$ have the same distribution and since this distribution is supported on a finite set,
one can assume that they have the same image set.
Thus, for any $\omega$ in $\Omega$, $(O_{n,\pi}(\omega))_\pi$ is equal to
$(\occ(\pi,\si))_\pi$ for some permutation $\si$ of size $n$ that we denote $\Si'_n(\omega)$.
This defines a sequence of random permutations $\Si'_n$ defined on the probability space $\Omega$
such that $(O_{n,\pi})_\pi=(\occ(\pi,\Si'_n))_\pi$.
%VF 15/02 autre modif suite au referee
Note that $\Si'_n$ is measurable since, for a permutation $\tau$ of size $n$,
we have $\{\omega: \Si'_n(\omega)=\tau\} = \{\omega: O_{n,\tau}(\omega)=1\}$.

Next, we claim that for any fixed $n$, $\Si'_n$ has the same distribution as $\Si_n$. 
Indeed, $n$ being fixed, considering the restriction of $(\occ(\pi,\Si_n))_\pi$ (resp. $(\occ(\pi,\Si'_n))_\pi$) 
to patterns $\pi$ of size $n$ gives the distribution of $\Si_n$ (resp. $\Si'_n$). 
The claim then follows since $(\occ(\pi,\Si_n))_\pi$ and $(\occ(\pi,\Si'_n))_\pi$ have the same distribution (both the same as that of $(O_{n,\pi})_\pi$). 

From \cref{item:conv_ps} above, and by definition of $\Si'_n$, for any $\omega \in \Omega$, the following holds:
  \[\text{for any permutation }\pi,\, \occ(\pi,\Si'_n(\omega)) 
  \to \Lambda'(\omega)_\pi.\]
  Using \cite[Theorem 1.6 (i)]{Permutons}, this implies that,
  for any $\omega \in \Omega$, there exists a (unique) permuton
  $\Mu(\omega)$ such that $\Si'_n(\omega)$ tends to $\Mu(\omega)$
  in the sense of \cref{dfn:cv_Permutons},
  or equivalently $\mu_{\Si'_n(\omega)}$ converges weakly to $\Mu(\omega)$.
  % LG 13/02 : ajout suite au referee
  %Modifié VF 15/02
  Since $\omega \mapsto \Mu(\omega)$ is the pointwise limit of the measurable functions
  $\omega \mapsto \mu_{\Si'_n(\omega)}$ in the {\em metrizable} weak topology,
  we know that $\Mu$ is measurable.

  Using that pointwise convergence implies convergence in distribution
  and that $\Si'_n$ and $\Si_n$ have the same distribution,
  we deduce that $\mu_{\Si_n}$ tends in distribution to $\Mu$ 
  in the weak convergence topology, as claimed.

  Finally, it should be noted that, for any $\pi$, 
  $\Lambda_{\pi} \stackrel{(d)}{=} t(\pi,\Mu)$,
  using the notation $t$ from \cref{dfn:cv_Permutons}, 
  since both are the limit in distribution of $\occ(\pi,\Si_n)$.
  From Theorem~\ref{thm:main}.\ref{item:main_variance} $\Lambda_{12}$ is not deterministic,
  so neither is $\Mu$.
\end{proof}

\section{Some properties of $\Lambda_\pi$}\label{sec:properties_of_moments}
In this section, we give some additional results concerning the limit variables
$\Lambda_\pi$. These results are not needed in the proof of our main theorem (\cref{thm:main}.\ref{item:main_joint}).
However, the more information we have on $\Lambda_\pi$,
the more interesting our theorem is.
\medskip

The first result, given in  \cref{ssec:Exp_LPi}, is a way of computing (joint) moments of the variables
$\Lambda_\pi$.
This has been used in the introduction
to give explicit values for the limit of (joint) moments of $\occ(\pi,\Si_n)$.
Our second result presented in \cref{ssec:variance} is the fact 
that the variables $\Lambda_\pi$ are not deterministic when $\pi$ is separable. 
This corresponds to \cref{item:main_variance} in \cref{thm:main}.

\subsection{Computing expectation and other moments of $\Lambda_\pi$} \label{ssec:Exp_LPi}

Recall that $N_\pi$ denotes the number of 
signed binary trees associated with the permutation $\pi$. 
\begin{proposition}
For any permutation $\pi$ of size $k$, 
\[ \esper\left[ \occ(\pi,\Si_n) \right] \longrightarrow \esper\left[ \Lambda_{\pi} \right] = \frac{N_\pi }{2^{k-1} \cdot \Cat_{k-1}}. \]
\label{prop:expectation_Lamdba_pi}
\end{proposition}

\begin{proof}
For any $\pi$, the convergence of $\esper\left[ \occ(\pi,\Si_n) \right]$
to $\esper\left[ \Lambda_{\pi} \right]$ is given by~\cref{thm:ConvExp} (p.\pageref{thm:ConvExp}). 
Let $\pi$ be a permutation of size $k$ 
and let us prove that $\esper\left[ \Lambda_{\pi} \right] = \frac{N_\pi}{2^{k-1} \cdot \Cat_{k-1}}$. 

We consider the signed Brownian excursion $(\Exc,S)$. 
We take $k$ points $X_1, \ldots, X_k$ in $[0,1]$ uniformly at random, independently from each other and from $(\Exc,S)$, 
and we let $\XX = (X_1, \ldots, X_k)$. 
From \cref{lem:expectation_tree_to_perm_continuous} (see also \cref{eq:ExpInProbForPrTree} in its proof),
we have
 \[\esper^{\Exc,S} \big[ \ProbPerm(\pi;e,S, F_{U}) \big] =
 \frac{1}{2^{k-1}} \sum_{(\patterntree,\eps)} 
    \proba^{\,\Exc,\XX}\big( \Tree(\Exc,\XX)=\patterntree ),\]
    where the sum runs over all signed binary trees $(\patterntree,\eps) $ of $\pi$.

    \cref{Lemma:ArbreBinaireUniforme}, a deep result on Brownian excursion,
    states that $\Tree(\Exc,\XX)$ is a uniform
    binary tree with $k$ leaves, \emph{i.e.}, for any binary tree $t_0$ with $k$ leaves,
    \[\proba^{\,\Exc,\XX}\big( \Tree(\Exc,\XX)=\patterntree ) = \frac{1}{\Cat_{k-1}}.\]
The proposition then follows immediately.
\end{proof}

In addition, $N_\pi$, the number of signed binary trees of a given permutation $\pi$, 
can be expressed combinatorially from the decomposition tree of $\pi$. More precisely:

\begin{observation}\label{obs:N_pi}
For any permutation $\pi$, denoting $t_\pi$ its (unique) decomposition tree, it holds that 
\[ N_\pi := \mathrm{card}\{(\patterntree,\eps),\ \patterntree \text{ binary}\, : \, \perm(\patterntree,\eps)=\pi \} =
\prod_{v \text{ internal vertex of }t_\pi} \Cat_{deg(v)-1}\]
where $deg(v)$ is the number of children of $v$.
\end{observation}

\begin{proof}
Given a signed tree of $\pi$, and a vertex $v$ with sign $\eps$ in this tree, 
the following transformation produces a tree which is still a signed tree of $\pi$: 
assuming that $k$ subtrees $T_1, \dots, T_k$ are attached to $v$, replace $v$ by 
a binary tree with $k$ leaves and all internal vertices labeled $\eps$
and where the $i$-th leaf is replaced by $T_i$. 
Conversely, each signed binary tree of $\pi$ can be obtained from the decomposition tree $t_\pi$ of $\pi$ 
applying this transformation on all internal vertices $v$ of $t_\pi$.
An example is given in \cref{fig:Form_100euros}.
\end{proof}

\begin{figure}[ht]
\begin{center}
\includegraphics[width=12cm,angle=180]{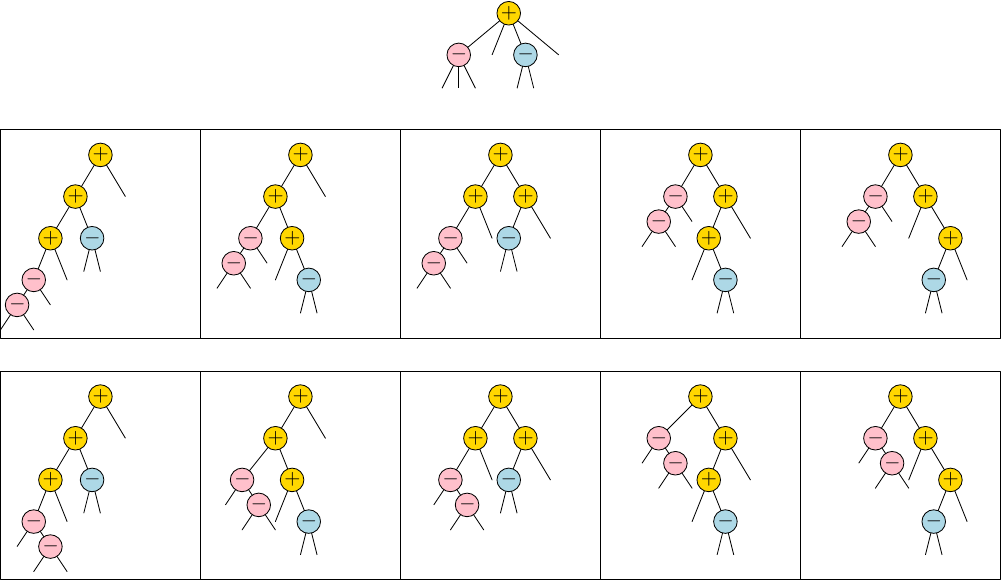}
\end{center}
\caption{The decomposition tree of 
$1324765$ 
and its $\Cat_{3}\times\Cat_2\times \Cat_1 = 10$ signed binary trees.}
\label{fig:Form_100euros}
\end{figure}

\begin{proposition}
\label{prop:computing_moments}
Let $\pi_1,\cdots,\pi_r$ be a list of patterns, and let $K=\sum_{i=1}^r |\pi_i|$. 
We have: 
\[
 \esper\left[ \prod_{i=1}^r \occ(\pi_i,\Si_n) \right] \longrightarrow \esper\left[ \prod_{i=1}^r \Lambda_{\pi_i} \right]= \sum_{\rho \in \Sn_K} c_{\pi_1,\ldots,\pi_r}^\rho \esper \left[ \Lambda_\rho \right],
\]
where for any $\rho \in \Sn_K$, $c_{\pi_1,\ldots,\pi_r}^\rho$ 
denotes the proportion of ordered set-partitions of $\{1, \dots K\}$ which are compatible with $\rho, \pi_1,\cdots,\pi_r$ -- see~\cref{dfn:ordered_set-partition} p.\pageref{dfn:ordered_set-partition}. 
\end{proposition}
\begin{proof}
See the proof of~\cref{cor:expectation_is_enough} for the convergence result, 
and \cref{prop:moments_in_terms_of_expectation_limit_case} for the expression of $\esper\left[ \prod \Lambda_{\pi_i} \right]$.
\end{proof}

The results of this section enable the automatic computation
of joint moments of $\Lambda_\pi$.
As mentioned in the introduction, we have implemented this in Sage.
Let us discuss quickly algorithmic questions behind this implementation.
\begin{itemize}
  \item From \cref{prop:expectation_Lamdba_pi}, computing the expectation essentially amounts to computing $N_\pi$.
    Finding the degree of the root of the decomposition tree of $\pi$
    is easy. To simplify the discussion (but without loss of generality), 
    suppose that this root has sign $+$.
    Then its degree is the number of integers $i \le k$ such that
    \[\{\pi(1),\ldots,\pi(i)\}=\{1,\cdots,i\}, \text{ or equivalently }\max_{j \le i} \pi(j)=i.\]
    Thanks to the second formulation, this can be computed in linear time
    reading the permutation from left-to-right,
    keeping only in memory the maximum of the values already read.
    
    Degrees of all vertices can then be computed recursively,
    and we find $N_\pi$ using \cref{obs:N_pi}, in quadratic time.

  \item Higher moments or joint moments are more complex to compute.
    Fix a list of patterns $\pi_1,\ldots,\pi_r$ of respective sizes $k_1,\cdots,k_r$
    and set $K=k_1+\cdots+k_r$. 
    While efficient for theoretical purpose, the definition of the coefficients 
    $c_{\pi_1,\ldots,\pi_r}^\rho$ (see \cref{dfn:ordered_set-partition} p.\pageref{dfn:ordered_set-partition}) is not optimal from a practical point of view.
    It is however possible to generate directly the multiset of permutations of $ \Sn_K$,
    where each $\rho$ appears with multiplicity $d_{\pi_1,\ldots,\pi_r}^\rho$.
    For each pair of ordered set-partitions $\PPos=(\Pos_1,\cdots,\Pos_r)$ and $\VVal=(\Val_1,\cdots,\Val_r)$ of $\{1,\dots,K\}$
    with $ \mathrm{card}(\Pos_i)=\mathrm{card}(\Val_i)=k_i$, we can define a permutation $\rho(\PPos,\VVal)$
    as follows: $\rho(\PPos,\VVal)$ assigns the $\pi_i(j)$-th smallest value in $\Val_i$
    to the $j$-th smallest value in $\Pos_i$.
    It is easy to check that any given $\rho$ is constructed this way from
    exactly $d_{\pi_1,\ldots,\pi_r}^\rho$ pairs $(\PPos,\VVal)$.
    We have used this construction in our implementation of the computation of joint moments.

    Even if this is more efficient than a naive implementation of \cref{dfn:ordered_set-partition},
    the complexity still grows very quickly.
    Indeed, the number of pairs $(\PPos,\VVal)$ as above is equal to $\binom{K}{k_1,\cdots,k_r}^2$. 
    In practice, we have only been able 
    to compute the first four moments of $\Lambda_{12}$ with this algorithm
    (the fifth moment given in the introduction has been inferred from the first four 
    using the symmetry $\Lambda_{12} \stackrel{(d)}{=} 1-\Lambda_{12}$).
    We believe that modifying slightly the program to enable parallel computing
    may allow to compute the fifth moment with the above algorithm in a reasonable time
    but that the sixth moment, which involves more
    than $5 \times 10^{13}$ pairs $(\PPos,\VVal)$, will remain out of reach.
\end{itemize}

\subsection{$\Lambda_\pi$ is not deterministic} \label{ssec:variance}

Finally we will prove Theorem~\ref{thm:main}.\ref{item:i}. 
We have already seen in \cref{rem:nonseparable} (p.\pageref{rem:nonseparable}) that, when $\pi$ is not separable,
$\Lambda_\pi$ is identically $0$ (in particular, deterministic). 
We have also seen in \cref{rem:motif1} (p.\pageref{rem:motif1}) that for the permutation $\pi$ of size $1$, 
$\Lambda_\pi$ is identically $1$ (and thus, again deterministic). 
On the contrary, we now prove that if $\pi$ is separable of size at least $2$, 
then $\Lambda_\pi$ is a true \emph{random} variable and not a deterministic constant. 

When $\pi$ is separable of size at least $2$, to prove that $\Lambda_\pi$ is not deterministic, it would be sufficient to check that $\mathrm{Var}(\Lambda_\pi)>0$. 
Although \cref{prop:computing_moments} provides an expression of the variance of $\Lambda_\pi$, 
this expression does not allow us to prove directly that $\mathrm{Var}(\Lambda_\pi)>0$.  
Instead, we prove in \cref{prop:not-equal} that $\Lambda_\pi \neq \mathbb{E}[\Lambda_\pi]$ with positive probablility. 
From the expression of $\mathbb{E}[\Lambda_\pi]$ given in \cref{prop:expectation_Lamdba_pi}, 
it follows immediately that $\mathbb{E}[\Lambda_\pi] >0$ if $\pi$ is separable. 
So the proof of Theorem~\ref{thm:main}.\ref{item:i} will be completed as soon as we prove that 
$\Lambda_\pi$ takes values as close to $0$ as wanted with positive probability. This is done in \cref{lem:IdentiteATousLesCoups}.

\medskip

\begin{lemma}\label{lem:IdentiteATousLesCoups}
  Let $\pi$ be a separable permutation of size $k \geq 2$,
  $(\Exc,S)$ be the signed Brownian excursion and $\N$ be an integer.
  There exists an event $E_\N=E_\N(\Exc,S)$ such that
  \begin{itemize}
    \item $E_\N$ occurs with positive probability;
    \item If $(\Exc,S)$ is such that $E_\N$ occurs, then $\mathbb{P}^{\XX}(\Perm(\Exc,S,\XX)= \pi )\leq \frac{k^2}{\N}$,
where $\XX=(X_1, \ldots, X_k)$ and the $X_i$'s are uniform independent points in $[0,1]$.
  \end{itemize}
\end{lemma}
\begin{proof}
Assume first that $\pi \neq 123\dots k$.
The key idea of the proof is that for some realizations of $(\Exc,S)$ (see a sketch in \cref{fig:evenement_chelou}), $\Perm(\Exc,S,\XX)=123\dots k$ with high probability.
\begin{figure}[ht]
  \begin{center}
\includegraphics[width=13cm]{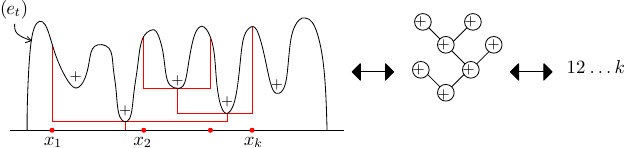}
  \end{center}
  \caption{A realization of $(\Exc,S)$ for which $\Perm(\Exc,S,\XX)=123\dots k$ with high probability.}
  \label{fig:evenement_chelou}
\end{figure}

We set $\eta = \tfrac{1}{4\N^2}$ and
define intervals as follows:
\begin{itemize}
  \item for integers $\ell$ between $0$ and $\N$,
    we set $J_\ell=[\tfrac{\ell}{\N} -\eta;\tfrac{\ell}{\N} + \eta]\cap[0,1]$;
  \item for integers $\ell$ between $1$ and $\N$,
    we set $I_\ell=[\tfrac{\ell-1}{\N} + \eta;\tfrac{\ell}{\N} - \eta]$.
\end{itemize}
These intervals are represented on \cref{fig:intervals}. 
We have $\mathrm{Length}(I_\ell)=1/\N-2\eta$ for each $\ell$. Moreover, $\mathrm{Length}(J_\ell)=2\eta$ for $\ell \neq 0$ and $\ell \neq \N$, while $\mathrm{Length}(J_0)=\mathrm{Length}(J_{\N})=\eta$.

\begin{figure}[ht]
  \begin{center}
    \includegraphics[width=12cm]{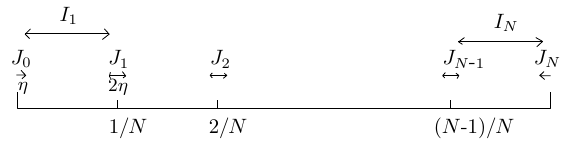}
  \end{center}\vspace{-8mm}
  \caption{Intervals $I_\ell$ and $J_\ell$.}
  \label{fig:intervals}
\end{figure}

Let $E_\N=E_\N(\Exc,S)$ be the event defined as follows:
\begin{enumerate}
  \item \label{item:i} For each $\ell$, $\min_{J_\ell} \Exc\leq 1/10$. 
  \item \label{item:ii} $\min_{I_1\cup I_2 \cup \dots \cup I_\N} \Exc\geq 2/10$.
\item The signs of the $\N-1$ local minima of $\Exc$ on $J_1,\dots J_{\N-1}$ are all $+$.
\end{enumerate}
Taking $f$ a piecewise linear excursion such that $f(\tfrac{\ell}{\N})=0$ for integers $\ell$ between $0$ and $\N$ 
and $f(x)=3/10$ for $x \in I_1\cup I_2 \cup \dots \cup I_\N$, 
and with $\delta=1/10$, 
it follows from \cref{lem:BrownienTubes} that 
$$
\mathbb{P}^{\Exc,S}(E_\N) =\frac{1}{2^{\N-1}} 
\mathbb{P}^\Exc(\Exc\text{ satisfies \cref{item:i,item:ii}})>0.
$$
Consider now a realization $(\Exc,S)$ of the signed Brownian excursion such that $E_\N$ is realized, and 
let $\XX=(X_1, \ldots, X_k)$ be $k$ uniform independent points in $[0,1]$.
The construction of $E_\N$ ensures that $\Perm(\Exc,S,\XX)= 123\dots k$ as soon as each $X_i$ is in an interval $I_\ell$ such that $I_\ell$ does not contain any other $X_j$. 
Equivalently, if $\Perm(\Exc,S,\XX)\neq 123\dots k$ then one of the two following events occurs:
\begin{itemize}
\item For some $i$, $X_i$ belongs to an interval $J_\ell$;
\item For some $i\neq j$, $X_i$ and $X_j$ belong to the same interval $I_\ell$.
\end{itemize}
Therefore we get
\begin{align*}
\mathbb{P}^{\XX}(\Perm(\Exc,S,\XX)= \pi)&\leq \mathbb{P}^{\XX}(\Perm(\Exc,S,\XX)\neq 123\dots k )\\
&\leq k \sum_{\ell=0}^\N\mathbb{P}^{\XX}(X_1\in J_\ell)  + \binom{k}{2}\sum_{\ell=1}^\N \mathbb{P}^{\XX}(X_1,X_2\in I_\ell)\\
&\leq k2\N\eta+ \binom{k}{2} \N \left(\frac{1}{\N}-2\eta\right)^2\\
&\leq \frac{k}{2\N}+ \binom{k}{2} \frac{1}{\N} \ = \frac{k^2}{\N}.\qquad 
\text{(recall that $\eta = \tfrac{1}{4\N^2}$)}.
\end{align*}
This finishes the proof of \cref{lem:IdentiteATousLesCoups} when $\pi\neq 123\dots k$. 
By symmetry, the same result holds for $\pi\neq k\dots 321$. 
Therefore, the statement of \cref{lem:IdentiteATousLesCoups} holds for all separable permutations of size at least $2$. 
\end{proof}

We can now establish  the announced proposition.
\begin{proposition}\label{prop:not-equal}
For any separable pattern $\pi$ of size at least $2$, 
$$
\mathbb{P}^{\Exc,S}(\Lambda_\pi< \mathbb{E}[\Lambda_\pi]) >0.
$$
In particular, $\Lambda_\pi$ is not almost surely constant.
\end{proposition}
\begin{proof} 
Let $k \geq 2$ be the size of $\pi$. Since $\mathbb{E}[\Lambda_\pi]>0$ (from \cref{prop:expectation_Lamdba_pi}) and $k$ is fixed,
one can choose $\N$ big enough such that $k^2/\N<\mathbb{E}[\Lambda_\pi]$.
For such a value of $\N$, let $E_\N$ be the event given by \cref{lem:IdentiteATousLesCoups}.
If $(\Exc,S)$ is such that $E_\N$ is realized, then 
(with $\XX=(X_1, \ldots, X_k)$ a tuple of $k$ uniform independent points in $[0,1]$)
\[\Lambda_\pi = 
\mathbb{P}^{\XX}(\Perm(\Exc,S,\XX)= \pi )\leq \frac{k^2}{\N} 
< \mathbb{E}[\Lambda_\pi].\]
Since $E_\N$ has positive probability,
the event $\Lambda_\pi< \mathbb{E}[\Lambda_\pi]$ also occurs with positive probability.
\end{proof}

\appendix
\section{Useful facts regarding the Brownian excursion} \label{appendix}
For the convenience of the reader, we now record several useful properties of a typical realization of the Brownian excursion $\Exc$.

There are several ways to define the Brownian excursion, the most convenient for us is to draw a realization of $\Exc$ from a realization of the Brownian motion $(B_t)_{t\geq 0}$, as follows (see \cite[Section 0.2]{Pitman}). Consider
\begin{equation}
a=\sup\{t\leq 1 : B_t=0\}, \qquad b=\inf\{t\geq 1 : B_t=0\} 
\label{Eq:Bornes_dilatation}
\end{equation}
(almost surely $a<1<b$), and set 
$$
\left(\Exc(s)\right)_{0\leq s\leq 1}:=\left(\frac{1}{\sqrt{b-a}}\big|B_{a+s(b-a)}\big|\right)_{0\leq s\leq 1}.
$$
Thus the Brownian excursion $\Exc$ is seen as a dilatation of a piece of $B$. 
It follows that some almost-sure properties of the set of local extrema of $B$ remain true for $\Exc$.

Recall that by definition,
$x$ is a one-sided local minimum for $f$ if, for some $\eps>0$
$$
f(x)=\min_{[x-\eps,x]}f\quad \text{ or }\quad f(x)=\min_{[x,x+\eps]}f. 
$$
\begin{lemma}
\label{Cor:OneSided}
  The set of one-sided local minima of the Brownian excursion $\Exc$  
  has Lebesgue measure $0$, almost surely.
\end{lemma}
\begin{proof}
We first prove an analogous statement for the Brownian motion $(B_t)_{t\geq 0}$.
We denote $\Omin(B)$ the set of one-sided local minima of $B$, and $\mathrm{Leb}(\Omin(B))$ its Lebesgue measure.
We have
$$
\mathrm{Leb}(\Omin(B))=\int_0^{+\infty} \mathbf{1}_{u\in \Omin(B)} du,
$$
so that, taking the expectation with respect to $B$, 
$$
\mathbb{E}[\mathrm{Leb}(\Omin(B))]=\int_0^{+\infty} \mathbb{E}[\mathbf{1}_{u\in \Omin(B)}] du
= \int_0^{+\infty} \mathbb{P}(u\in \Omin(B)) du =0.
$$
In the last equality we used the fact that for every fixed $u$, $\mathbb{P}(u\in \Omin(B))=0$: 
Theorem 1.27 in \cite{Peres} gives a similar result for local minima (\emph{i.e.}, two-sided minima),
but the proof is easily adapted to the case of one-sided minima.
Thus $\mathrm{Leb}(\Omin(B))$ is a nonnegative random variable with expectation $0$, and therefore is equal to $0$ almost surely.

The statement then follows for $\mathrm{Leb}(\Omin(\Exc))$ by dilatation,
since the dilatation of a set of measure zero has measure zero as well. 
\end{proof}

We now discuss values of local minima.
\begin{lemma}\label{Lemma:MinimumSameLevel}
With probability one  the Brownian excursion $\Exc$ has no two local minima with the same value.
\end{lemma}
\begin{proof}
For the Brownian motion it is the statement of \cite[Lemma 11.15]{Kallenberg}. This is also true for $\Exc$ since it is a dilatation of the Brownian motion.
\end{proof}
An important consequence of \cref{Lemma:MinimumSameLevel} in the present paper is that, for all set of $k$ distinct points $\xx$ of $[0,1]$ 
and for almost all realizations of $\Exc$, the tree $\Tree(\Exc,\xx)$ obtained in \cref{Section:Extracting} is a binary tree (because of \cref{obs:distinctMinImpliesBinary}).

A remarkable fact is that if $\xx$ is uniformly distributed this random binary tree is uniform (see \cite[Section 2.6]{LeGall}):
\begin{lemma}\label{Lemma:ArbreBinaireUniforme}
Fix $k\geq 2$ and a binary tree $t_0$ with $k$ leaves.
Let $U_1,\dots, U_{k}$ be $k$ uniform and independent random variables in $[0,1]$, independent from $\Exc$. Then
$$
\proba(\Tree(\Exc,\{U_1,\dots,U_k\})=t_0)=\frac{1}{\Cat_{k-1}}.
$$
\end{lemma}

It is in fact even possible to describe the law of the {\em geometric tree}
extracted from $\Exc$ and $U_1,\cdots,U_k$, \emph{i.e.}, a tree with edge-lengths that are nonnegative real numbers
(see \cite[Th. 2.11]{LeGall}).
In this paper, we use a rather weak consequence of this result.
\begin{lemma}\label{lem:tout_a_une_densite}
  Take $k$ i.i.d. uniform random variables in $[0,1]$ independently from $\Exc$
  and call them $U_1 < \cdots<U_k$.
Set $M_i=\min_{[U_i,U_{i+1}]}\Exc$. Then the random vector
\begin{equation*}
\vv= \Big(\Exc(U_1),\dots,\Exc(U_k),M_1,\dots,M_{k-1}\Big)
  \label{Eq:RandomVectorEq:RandomVector}
\end{equation*}
has distinct coordinates with probability $1$.
\end{lemma}
\begin{proof}
  As said above, the law of the geometric tree extracted from $\Exc$ and $U_1,\cdots,U_k$ 
  is given in \cite[Th. 2.11]{LeGall}; this law has a density with respect
  to the uniform distribution on geometric trees.
  Conditioning on the fact that $\Tree(\Exc,\{U_1,\dots,U_k\})$ is a given $t_0$ with $k$ leaves,
  the coordinates of $\vv$ 
  are sums of edge-lengths of the geometric tree.
  Hence the vector $\vv$ has a density with respect to the Lebesgue measure
  on $\R^{2k-1}$.
  Without conditioning, $\vv$  has also a density,
  which is simply the mean of the conditional densities.
  This implies the lemma.
\end{proof}
\smallskip

Finally we need the fact that the Brownian excursion $\Exc$ is arbitrary close to any fixed Lipschitz excursion
with positive probability. 
(A Lipschitz excursion
is simply an excursion that is also a Lipschitz function, \emph{i.e.}, 
there exists $C>0$ such that $|f(x)-f(y)| \le C|x-y|$ for all $x,y$ in $[0,1]$.)
\begin{lemma}\label{lem:BrownienTubes}
For every Lipschitz excursion $f$ and $\delta>0$,
$$
\proba\left( \sup_{0\leq s\leq 1}|\Exc(s)-f(s)|\leq \delta\right) >0.
$$
\end{lemma}
\begin{proof}
  The proof relies on a similar result for Brownian motion \cite[Sec.1.4]{Freedman}. 
  Let us give some details.

  Fix a Lipschitz excursion $f$ and $\delta>0$ as in the statement of the lemma. 
  Without loss of generality, assume that $\delta < 1/2$. 
  Define $g_\delta(t)=\min(\delta,t,1-t)$, so that $||g_\delta||_\infty=\delta$.
  Then 
  \[ |\Exc(s)-f(s)-g_{\delta/2}(s)| \leq \delta/2 \ \Rightarrow |\Exc(s)-f(s)|\leq \delta. \]
  Therefore replacing if necessary $f$ by $f+g_{\delta/2}$ and $\delta$ by $\delta/2$,
  we can assume without loss of generality that the following holds:
  \begin{equation}
    \text{for any }\eta \le \delta/2\text{ and }s \in [\eta,1-\eta],
    \text{ one has } f(s) \geq \eta.
    \label{Eq:Exc_loin_zero}
  \end{equation}
    
  We extend $f$ to the interval $[-\delta,1+\delta]$ by setting $f(t)=t$ if $t \le 0$ 
  and $f(t)=1-t$ if $t \ge 1$.
  Let $(B_t)_{t \ge 0}$ be a realization of the Brownian motion.
  Let $\eta \in (0,\delta/2]$.
  Using the results of \cite[Sec.1.4]{Freedman}, we know that with positive probability we have
  \begin{equation}
  \sup_{s \in [1/2 - \delta,3/2+\delta]} |B_s - f(s-\tfrac12)| < \eta.
    \label{eq:tube_Brownien}
  \end{equation}
  Together with \eqref{Eq:Exc_loin_zero},
  this implies in particular that 
  \[\begin{cases}
    B_s>0 &\text{ if }s \in [\frac12+\eta,\frac32-\eta];\\
    B_s<0 &\text{ if }s \in[\frac12-\delta,\frac12-\eta] \cup [\frac32+\eta,\frac32+ \delta].
  \end{cases}\]
  Therefore, if we define $a$ and $b$ as in \eqref{Eq:Bornes_dilatation},
  we have
  \begin{equation}
    \tfrac12-\eta < a < \tfrac12+\eta, \qquad \tfrac32-\eta < b < \tfrac32+\eta.
    \label{Eq:AB_Controlled}
  \end{equation}
  In particular $|b-a-1| < 2 \eta$. 
  For $s$ in $[0,1]$, we can write 
  \begin{align*}
    |e(s)-f(s)|=&\left| \frac{1}{\sqrt{b-a}} B_{a+s(b-a)} -f(s)\right| \\
    \le &\frac{1}{\sqrt{b-a}} \left|  B_{a+s(b-a)} -f(a-\tfrac12+s(b-a)) \right|\\
    &+\left| \frac{1}{\sqrt{b-a}} -1 \right| \cdot \left| f(a-\tfrac12+s(b-a)) \right|+ \big| f(a-\tfrac12+s(b-a)) -f(s)\big|.
  \end{align*}
  Using \cref{eq:tube_Brownien,Eq:AB_Controlled}, the inequality $|b-a-1| < 2 \eta$ and 
  the fact that $f$ is a bounded Lipschitz function, 
  it is not hard to see that this upper bound is smaller than $C \eta$ for some constant $C$.   
  Since $\eta$ can be chosen as small as wanted, 
  we may assume $C \eta \le \delta$ 
  and we get $|e(s)-f(s)| \le \delta$ (for all $s$ in $[0,1]$).
  
  In summary, for $\eta$ sufficiently small, 
  \eqref{eq:tube_Brownien} implies $\sup_{0 \le s \le 1} |e(s)-f(s)| \le \delta$,
  so that the latter occurs with positive probability, as wanted.
\end{proof}
\section*{Acknowledgements}
Many thanks to Carine Pivoteau for providing Boltzmann samplers of permutations in classes, for running experiments with us using those, 
and for producing the simulations shown in the introduction. 
In addition, we thank Dominique Rossin for sharing the results of early experiments he did on random permutations in classes. 

We are also very grateful to Grégory Miermont for discussions at several stages of the project;
he suggested in particular the exchangeability argument at the core of the proof of Proposition~\ref{prop:balanced_signs_discrete}. 
We also thank warmly Igor Kortchemski for helping us find our way in the literature on the Brownian motion and excursion. 
Finally, we thank Douglas Rizzolo for pointing out that the value of the variance in \cref{prop:Schroeder_are_Galton-Watson} was wrong in the first version of the present paper.

Our work was supported by a Swiss-French PHC Germaine de Stael grant (project 2015-09). FB, LG and AP also benefited from the hospitality of the Institut f\"ur Mathematik of Z\"urich in May and December 2015. 
MB is supported by a Marie Heim-V\"ogtlin grant of the Swiss National Science Foundation.

Finally, we thank an anonymous referee for his/her constructive comments,
in particular for pointing out the measurability issue
discussed in \cref{sec:measurability} and several remarks and bibliographical pointers that allowed us to shorten some proofs.

\end{document}